\documentclass[10pt, oneside]{article}
\usepackage{mathrsfs}
\usepackage{amsfonts}
\usepackage{amsmath,amsfonts,amssymb,amsthm}
\usepackage{caption}
\usepackage{subfigure}
\usepackage{cite}
\usepackage{color}
\usepackage{graphicx}
\usepackage{subfigure}
\usepackage{float}
\usepackage{paralist}
\usepackage{indentfirst}
\usepackage{authblk}

\usepackage[colorlinks,
linkcolor=red,
citecolor=blue
]{hyperref}

\usepackage[T1]{fontenc}
\DeclareMathOperator{\dive}{div}

\newtheorem{theorem}{Theorem}[section]
\newtheorem{lemma}{Lemma}[section]
\newtheorem{prop}{Proposition}[section]

\newtheorem{defn}{Definition}[section]

\def\bma#1\ema{{\allowdisplaybreaks\begin{split}#1\end{split}}}

\numberwithin{equation}{section}

\begin{document}
\title{{\LARGE \textbf{Global well-posedness of one-dimensional compressible Navier-Stokes-Vlasov system}}}
\author[a,b]{Hai-Liang  Li  \thanks{
E-mail:		hailiang.li.math@gmail.com (H.-L. Li).}}

\author[a,b]{Ling-Yun Shou \thanks{Corresponding author. E-mail: shoulingyun11@gmail.com(L.-Y Shou).}}
    \affil[a]{School of Mathematical Sciences,
	Capital Normal University, Beijing 100048, P.R. China}
\affil[b]{Academy for Multidisciplinary Studies, Capital Normal University, Beijing 100048, P.R. China}

\date{}
\renewcommand*{\Affilfont}{\small\it}
\maketitle
\begin{abstract}
A fluid-particle model is investigated in the present paper, which consists of the compressible Navier-Stokes equations coupled with the Vlasov equation though a nonlinear drag force. We consider the initial value problem for the one-dimensional compressible Navier-Stokes-Vlasov system and establish the global existence and uniqueness of the weak solution for general initial data in either spatial periodic domain or spatial real line, which is shown to be a strong solution for regular initial data. Moreover, it is showed that for spatial periodic domain, both the fluid velocity and the macroscopic velocity of the particles converge to the same speed exponentially in time, and the compact support of distribution function associated with velocity variable converges to a point set exponentially in time. 
\end{abstract}
\noindent{\textbf{Key words:} Fluid-particle model; Nonlinear drag force; Global well-posedness, Large time behavior  }
\section{Introduction}
The fluid-particle interaction models are introduced to simulate the motion of particles dispersed in a fluid \cite{bernard1,bernard2,caflisch1,constantin1,desv1,jabin1,lin1,o1,williams2}, and are used widely in many applied scientific areas such as diesel engines, rocket propulsors, aerosols and sprays, pollution settling processes, chemical engineering, wastewater treatment or biomedical flows \cite{am1,baranger1,berres2,berres1,boudin2,o1,williams1,williams2}, etc.

 In this paper, we consider the initial value problem (IVP) of the following one-dimensional compressible Navier-Stokes-Vlasov system for fluid-particle motion:
\begin{equation}\label{m1}
\left\{
 \begin{split}
& \rho_{t}+(\rho u)_{x}=0,\\
& (\rho u)_{t}+(\rho u^2)_{x}+(P(\rho))_{x}=(\mu(\rho)u_{x})_{x}-\int_{\mathbb{R}} F_{d}fdv,\\
&f_{t}+vf_{x}+(F_{d}f)_{v}=0,\quad (x,v)\in \Omega\times\mathbb{R},~t>0,
\end{split}
\right .
\end{equation}
 with the initial data
 \begin{align}
&(\rho(x,0),u(x,0),f(x,v,0))=(\rho_{0}(x),u_{0}(x),f_{0}(x,v)),\quad (x,v)\in \Omega\times\mathbb{R},\label{d}
\end{align}
where $\rho=\rho(x,t)$ and $u=u(x,t)$ denote the density and velocity associated with the dense phase (fluid) respectively, governed by the compressible Navier-Stokes equations $(\ref{m1})_{1}$-$(\ref{m1})_{2}$, and $f=f(x,v,t)$ is the distribution function associated with the dispersed phase (particles), governed by the Vlasov equation $(\ref{m1})_{3}$. The spatial domain $\Omega$ is either the periodic domain $\mathbb{T}=\mathbb{R}/ \mathbb{Z}$ or the real line $\mathbb{R}$. The pressure $P(\rho)$ and the viscosity coefficient $\mu(\rho)$ satisfy
\begin{align}
P(\rho)=A\rho^{\gamma},\quad \mu(\rho)=\mu_{0}+\mu_{1}\rho^{\beta},\label{n111}
 \end{align}
and the drag force $F_{d}$ takes the form
 \begin{align}
F_{d}=\kappa\rho(u-v),\label{kappa}
\end{align}
where $A,\mu_{0}, \mu_{1},\kappa, \gamma$ and $\beta$ are constants satisfying
$$
 A>0,\quad \mu_{0}>0,\quad \mu_{1}>0,\quad \kappa>0,\quad \gamma>1,\quad \beta\geq 0.
$$
 Without loss of generality, we take $A=\mu_{0}=\mu_{1}=\kappa=1$ in the present paper.

There are many important progress made recently on the analysis of existence and dynamical behaviors of solutions for fluid-particle models \cite{bae1,baranger2,boudin1,carrillo4,choi1,choi3,duan1,glass1,goudon1,hamdache1,han2,han3,lf1,wangd1}.  Among them, concerning with multi-dimensional incompressible Navier-Stokes-Vlasov equations, the global existence of weak solutions in spatial periodic domain was established in \cite{boudin1}, and the asymptotic stability of stationary states in a pipe with partially absorbing boundary conditions was investigated in \cite{glass1}. The global well-posedness and large time behavior of weak solutions for three-dimensional incompressible Navier-Stokes-Vlasov system were showed for small initial data in both spatial periodic domain \cite{han2} and spatial whole space \cite{han3} respectively. And the global existence of weak solutions for three-dimensional inhomogeneous Navier-Stokes-Vlasov equations was studied in \cite{wangd1}. For compressible Navier-Stokes-Vlasov system, the exponential time stability of global solutions in spatial periodic domain was obtained in \cite{bae1,choi1,lhl1} provided that the global solutions existed and satisfied some uniform estimates, and the finite time blow-up phenomena of classical solutions in the presence of vacuum was proved in \cite{choi3} for any space dimension.

However, there are no any results about the global well-posedness problem for compressible Navier-Stokes-Vlasov equations as we know so far. The purpose of this paper is to deal with the existence, uniqueness and regularity of global weak solution to the IVP $(\ref{m1})$-$(\ref{kappa})$ for general initial data.

Before stating the main results, we present some notations below.

\noindent
\textbf{Notations.}
 Throughout this paper, $C>0$ and $c>0$ denote some constants independent of time, and $C_{T}>0$ is a constant which may depend on time $T$. For $p\in[1,\infty]$ and integer $k\geq 0$, the simplified  Lebesgue, Sobolev and H$\rm{\ddot{o}}$lder norms are given as follows:
\begin{equation}\nonumber
\left\{
\begin{split}
&\|\cdot \|_{L^{p}}=:\|\cdot \|_{L^{p}(\Omega )},\quad \|\cdot \|_{H^{k}}=:\|\cdot\|_{H^{k}(\Omega )},\quad\|\cdot \|_{C^{k}}=:\|\cdot\|_{C^{k}(\Omega )},\\
& \|\cdot \|_{\mathcal{L}^p}=:\|\cdot \|_{L^p(\Omega \times\mathbb{R} )},\quad \|\cdot \|_{\mathcal{H}^{k}}=:\|\cdot \|_{H^{k}(\Omega \times\mathbb{R})},\quad \|\cdot \|_{\mathcal{C}^{k}}=:\|\cdot \|_{C^{k}(\Omega \times\mathbb{R})}.
\end{split}
\right .
\end{equation}
Let $n$ and $w$ denote the average local density and velocity of distribution function $f$ to the Vlasov equation $(\ref{m1})_{3}$ respectively as
\begin{align}
n(x,t)=:\int_{\mathbb{R}} f(x,v,t)dv,\quad nw(x,t)=:\int_{\mathbb{R}}vf(x,v,t)dv.\label{nw}
\end{align}

First, we have the following result for the spatial periodic domain:

\begin{theorem}\label{theorem11}
Let $\Omega=\mathbb{T}$. Assume that the initial data $(\rho_{0},u_{0},f_{0})$ satisfies
\begin{equation}\label{a1}
\left\{
\begin{split}
&\inf_{x\in\mathbb{T} }\rho_{0}(x)>0,\quad \rho_{0}\in W^{1,\infty}(\mathbb{T}),\quad u_{0}\in H^1(\mathbb{T} ),\\
 &0\leq f_{0}\in L^{\infty}(\mathbb{T}\times\mathbb{R}),\quad \text{{\rm{Supp}}}_{v}f_{0}(x,\cdot)\subset \{v\in\mathbb{R}~\big{|}~|v|\leq R_{0}\},\quad  x\in\mathbb{T},
\end{split}
\right .
\end{equation}
 where $R_{0}>0$ is a constant. Then the IVP $(\ref{m1})$-$(\ref{kappa})$ admits a unique global weak solution $(\rho,u,f)$ satisfying for any $T>0$ that
\begin{equation}\label{r1}
\left\{
\begin{split}
& \rho\in C([0,T];H^1(\mathbb{T}))\cap L^{\infty}(0,T;W^{1,\infty}(\mathbb{T})),\\
&u\in C([0,T];H^1(\mathbb{T} ))\cap L^2(0,T;H^2(\mathbb{T} )),\quad u_{t}\in L^2(0,T;L^2(\mathbb{T})),\\
 &f\in C([0,T];L^1(\mathbb{T}\times\mathbb{R}))\cap L^{\infty}(0,T;L^{\infty}(\mathbb{T} \times\mathbb{R})),
\end{split}
\right .
\end{equation}
and
\begin{equation}\label{r11}
\left\{
\begin{split}
&0<\rho_{-}\leq \rho\leq \rho_{+},\quad \sup_{t\in [0,T]}\|\rho(t)\|_{H^1}\leq C_{0},\\
&0\leq f\leq e^{\rho_{+}T}\|f_{0}\|_{\mathcal{L}^{\infty}},\\
&\text{{\rm{Supp}}}_{v}f(x,\cdot, t)\subset \{v\in\mathbb{R}~|~|v|\leq R_{1}\},\quad(x,t)\in \mathbb{T}\times[0,T],
\end{split}
\right.
\end{equation}
where $\rho_{-}>0, \rho_{+}>0, C_{0}>0$ and $R_{1}>0$ are constants independent of time $T>0$.

 In addition to $(\ref{a1})$, if it further holds
\begin{align}
&\rho_{0}\in H^2(\mathbb{T} ),\quad u_{0}\in H^2(\mathbb{T} ),\quad f_{0}\in C^{1}(\mathbb{T} \times\mathbb{R}),\label{a11}
\end{align}
then $(\rho, u, f)$ is indeed a global strong solution satisfying for any $T>0$ that
\begin{equation}
\left\{
\begin{split}
&\rho\in C([0,T];H^2(\mathbb{T} )),~~\rho_{t}\in C([0,T];H^1(\mathbb{T} )),\\
&u\in C([0,T];H^2(\mathbb{T} )),\quad t^{\frac{1}{2}}u\in L^{\infty}(0,T;H^3(\mathbb{T} )),\quad t^{\frac{1}{2}}u_{t}\in L^{\infty}(0,T;H^1(\mathbb{T} )),\\
&f\in C([0,T];C^{1}(\mathbb{T} \times\mathbb{R} )),\quad f_{t}\in C([0,T];C^{0}(\mathbb{T} \times\mathbb{R} )).\label{r12}
\end{split}
\right .
\end{equation}
\end{theorem}
~\par

Next, it is shown that as time evolves, the global solution $(\rho,u,f)$ given by Theorem \ref{theorem11} converges to its equilibrium state exponentially.

\begin{theorem}\label{decay}
Let the assumptions of Theorem \ref{theorem11} hold, and $(\rho,u,f)$ be the global solution to the IVP $(\ref{m1})$-$(\ref{kappa})$ given by Theorem \ref{theorem11}. Then it holds
\begin{equation}\label{decay11}
\left\{
\begin{split}
&\|(\rho-\overline{\rho_{0}})(t)\|_{C^{0}}+\|(u-u_{s})(t)\|_{L^{2}}\leq C_{1}e^{-\lambda t},\\
&\|\big{(}\sqrt{f}(v-u_{s})\big{)}(t)\|_{\mathcal{L}^2}+{\rm{W}}_{1}\big{(}f(t),(n\delta(v-u_{s}))(t)\big{)}\leq C_{1}e^{-\lambda t},\\
&\text{{\rm{Supp}}}_{v}f(x,\cdot, t)\subset \{v\in\mathbb{R}~|~|v-u_{s}|\leq C_{1}e^{-\lambda t}\},\quad (x,t)\in \mathbb{T}\times \mathbb{R}_{+},
\end{split}
\right.
\end{equation}
where ${\rm{W}}_{1}(f,g)$ is the Wasserstein distance given by Definition 1.1, $C_{1}>0$ and $\lambda>0$ are two constants independent of time, and $\overline{\rho_{0}}$ and $u_{s}$ are given by
\begin{equation}\label{rhoinfty}
\left\{
\begin{split}
&\overline{\rho_{0}}=:\int_{\mathbb{T}}\rho_{0}(x)dx,\\
& u_{s}=:\frac{\int_{\mathbb{T}}\rho_{0}u_{0}(x)dx+\int_{\mathbb{T}\times\mathbb{R}}vf_{0}(x,v)dvdx} {\int_{\mathbb{T}}\rho_{0}(x)dx+ \int_{\mathbb{T}\times\mathbb{R}}f_{0}(x,v)dvdx}.
\end{split}
\right.
\end{equation}
\end{theorem}

-\par
Then, we have the following result for the spatial real line:
\begin{theorem}\label{theorem12}
Let $\Omega=\mathbb{R}$. Assume that the initial data $(\rho_{0},u_{0},f_{0})$ satisfies
 \begin{equation}\label{a2}
 \left\{
 \begin{split}
 &\inf_{x\in\mathbb{R}}\rho_{0}(x)>0,\quad \rho_{0}-\widetilde{\rho}\in H^1\cap W^{1,\infty}(\mathbb{R}),\quad u_{0}\in H^1(\mathbb{R}),\\
 &0\leq f_{0}\in L^1\cap L^{\infty}(\mathbb{R}\times\mathbb{R}), \quad \text{{\rm{Supp}}}_{v}f_{0}(x,\cdot)\subset \{v\in\mathbb{R}~\big{|}~|v|\leq \widetilde{R}_{0}\},\quad x\in\mathbb{R},
\end{split}
\right .
\end{equation}
where $\widetilde{\rho}>0$ and $\widetilde{R}_{0}>0$ are two constants. Then the IVP $(\ref{m1})$-$(\ref{kappa})$ admits a unique global weak solution $(\rho,u,f)$ satisfying for any $T>0$ that
\begin{equation}
\left\{
\begin{split}
&\rho-\widetilde{\rho}\in C([0,T];H^1(\mathbb{R}))\cap L^{\infty}(0,T; W^{1,\infty}(\mathbb{R})),\\
&u\in C([0,T];H^1(\mathbb{R}))\cap L^2(0,T;H^2(\mathbb{R})),\quad u_{t}\in L^2(0,T;L^2(\mathbb{R})),\\
 &f\in C([0,T];L^1(\mathbb{R}\times\mathbb{R}))\cap L^{\infty}(0,T;L^{\infty}(\mathbb{R}\times\mathbb{R})),\label{r2}
\end{split}
\right .
\end{equation}
and
\begin{equation}\label{r21}
\left\{
\begin{split}
&0<\rho_{T}^{-1}\leq\rho\leq \rho_{T},\\
&0\leq f\leq e^{\rho_{T}T}\|f_{0}\|_{\mathcal{L}^{\infty}},\\
&{\rm{Supp}}_{v} f(x,\cdot, t)\subset \{v\in\mathbb{R}~\big{|}~|v|\leq \widetilde{R}_{T}\},\quad (x,t)\in\mathbb{R}\times[0,T],
\end{split}
\right.
\end{equation}
where $\rho_{T}>0$ and $\widetilde{R}_{T}>0$ are constants dependent of time $T>0$.

 In addition to $(\ref{a2})$, if it further holds
\begin{align}
&\rho_{0}-\widetilde{\rho}\in H^2(\mathbb{R}),\quad u_{0}\in H^2(\mathbb{R}),\quad f_{0}\in H^{1}\cap C^{1}(\mathbb{R}\times\mathbb{R}),\label{a22}
\end{align}
then the global weak solution $(\rho, u, f)$ is indeed a strong solution satisfying for any $T>0$ that
\begin{equation}
\left\{
\begin{split}
&\rho-\widetilde{\rho}\in C([0,T];H^2(\mathbb{R})),~~\rho_{t}\in C([0,T];H^1(\mathbb{R})),\\
&u\in C([0,T];H^2(\mathbb{R})),\quad t^{\frac{1}{2}}u\in L^{\infty}(0,T;H^3(\mathbb{R})),\quad t^{\frac{1}{2}}u_{t}\in L^{\infty}(0,T;H^1(\mathbb{R})),\\
&f\in C([0,T];H^1\cap C^{1}(\mathbb{R}\times\mathbb{R} )),\quad f_{t}\in C([0,T];L^2\cap C^{0}(\mathbb{R}\times\mathbb{R} )).\label{r22}
\end{split}
\right .
\end{equation}
\end{theorem}
~\par
\indent  We explain the main ideas to prove Theorem \ref{theorem11}, which can be applied to the proof of Theorem \ref{theorem12} with some modifications. Indeed, for any initial data satisfying $(\ref{a1})$ and (\ref{a11}), we can show the local well-posedness of strong solution to the IVP $(\ref{m1})$-$(\ref{kappa})$. Thus, to extend the solution globally in time, we need to overcome the difficulties caused by the drag force term $\rho n(w-u)$ on the right-hand side of the equation $(\ref{m1})_{2}$ and establish the uniformly a-priori estimates of the solution. To this end, we introduce a new effective velocity
\begin{align}
u+\mathcal{I}(n),\nonumber
\end{align}
where the operator $\mathcal{I}:L^1(\mathbb{T})\rightarrow L^{\infty}(\mathbb{T})$ is defined by
\begin{align}
\mathcal{I}(g)(x)=:\int^{x}_{0}g(y)dy-\int_{0}^{1}\int_{0}^{y}g(z) dzdy,\quad \forall g\in L^1(\mathbb{T}),\nonumber
\end{align}
 so that the momentum equation $(\ref{m1})_{2}$ can be re-written as
\begin{align}
&\big{(}\rho (u+\mathcal{I}(n))\big{)}_{t}+\big{(}\rho u(u+\mathcal{I}(n))\big{)}_{x}+(\rho^{\gamma})_{x}=(\mu(\rho)u_{x})_{x}+\rho\int_{0}^{1} nw(y,t)dy.\label{mm}
\end{align}
Making use of basic energy estimates and the equation (\ref{mm}) instead of $(\ref{m1})_{2}$, we are able to obtain the upper bound of the fluid density $\rho$ uniformly in time. In addition, we show that the $L^1$-norm of the pressure $P(\rho)$ is strictly positive for large time, which is crucial to establish the lower bound of $\rho$ uniformly with respect to time.

 Then, we define the second effective velocity
\begin{align}
U=:u+\mathcal{I}(n)+\rho^{-2}\mu(\rho)\rho_{x}.\nonumber
\end{align}
By $(\ref{m1})_{1}$ and (\ref{mm}), we can verify that $U$ satisfies the following equation:
\begin{align}
\rho(U_{t}+uU_{x})+\gamma\rho^{\gamma+1}\mu(\rho)^{-1}U=\gamma\rho^{\gamma+1}\mu(\rho)^{-1}(u+\mathcal{I}(n))+\rho\int_{0}^{1} nw(y,t)dy.\label{Ue}
\end{align}
After having proved that the density $\rho$ has strictly positive upper and lower bounds, which implies that the term $\gamma\rho^{\gamma+1}\mu(\rho)^{-1}U$ acts a damping in (\ref{Ue}), we are able to establish the $L^2$-estimate of $\rho_{x}$ uniformly in time. The higher order regularities of the solution can be showed by standard arguments after some modifications.

The rest of the paper is organized as follows. Section 2 is devoted to the a-priori estimates and global existence of solutions to the IVP $(\ref{m1})$-$(\ref{kappa})$ in both spatial periodic domain and spatial real line. The uniqueness of the weak solution satisfying either $(\ref{r1})$-$(\ref{r11})$ or $(\ref{r2})$-$(\ref{r21})$ is proved in Section 3. The local well-posedness of the IVP $(\ref{m1})$-$(\ref{kappa})$ is shown in Appendix A, and the exponential time-decay rate of the global weak solution is obtained in Appendix B.

\section{The a-priori estimates and global existence}

\subsection{Spatial periodic domain}

In this subsection, we establish the a-priori estimates for any regular solution to the IVP $(\ref{m1})$-$(\ref{kappa})$ in spatial periodic domain. Indeed, we can construct a approximate sequence of global regular solutions, show their convergence to a global weak solution of the IVP $(\ref{m1})$-$(\ref{kappa})$ based on the corresponding uniform estimates, and then justify the expected properties in Theorem \ref{theorem11} for the solution.

We start with

\begin{lemma}\label{lemma21}
Let $T>0$, and $(\rho,u,f)$ be any regular solution to the IVP $(\ref{m1})$-$(\ref{kappa})$ for $t\in (0,T]$. Then, under the assumptions of Theorem \ref{theorem11}, it holds
\begin{equation}
\left\{
\begin{split}
&\rho(x,t)\geq 0,\quad f(x,v,t)\geq0,\quad (x,v,t)\in\mathbb{T}\times\mathbb{R}\times [0,T],\\
&\int_{\mathbb{T}}\rho(x,t) dx=\int_{\mathbb{T}}\rho_{0}(x)dx,\quad t\in[0,T],\\
 &\int_{\mathbb{T}\times\mathbb{R}} f(x,v,t)dvdx=\int_{\mathbb{T}} n(x,t)dx=\int_{\mathbb{T}\times\mathbb{R}}f_{0}(x,v)dx,\quad t\in[0,T],\\
&\sup_{0\leq t\leq T}\bigg{(}\int_{\mathbb{T}}\big{(}\frac{1}{2}\rho|u|^2+\frac{\rho^{\gamma}}{\gamma-1}\big{)}(x,t)dx+\int_{\mathbb{T}\times\mathbb{R}} \frac{1}{2}|v|^2f(x,v,t)dvdx\bigg{)}\\
&\quad\quad+\int_{0}^{T}\int_{\mathbb{T}}(\mu(\rho)|u_{x}|^2)(x,t)dxdt+\int_{0}^{T}\int_{\mathbb{T}\times\mathbb{R}}(\rho|u-v|^2f)(x,v,t)dvdxdt\leq E_{0},\label{basiccnsv}
\end{split}
\right .
\end{equation}
where the initial energy $E_{0}$ is given by
\begin{equation}\label{E0}
\begin{split}
E_{0}=:\int_{\mathbb{T}} \big{(}\frac{1}{2}\rho_{0}|u_{0}|^2+\frac{\rho_{0}^{\gamma}}{\gamma-1}\big{)}(x)dx+\int_{\mathbb{T}\times\mathbb{R}} \frac{1}{2}|v|^2f_{0}(x,v)dvdx.
\end{split}
\end{equation}
\end{lemma}
\begin{proof}
Due to the maximum principle for transport equations,  $(\ref{basiccnsv})_{1}$ holds. In addition, $(\ref{basiccnsv})_{2}$-$(\ref{basiccnsv})_{3}$ can be derived after a direct computation. To show $(\ref{basiccnsv})_{4}$, multiplying $(\ref{m1})_{1}$ and $(\ref{m1})_{2}$ by $\frac{\gamma}{\gamma-1}\rho^{\gamma-1}$ and $u$ respectively, and integrating the resulted equations by parts over $\mathbb{T}$, one gets after adding them together that
\begin{equation}
\begin{split}
&\frac{d}{dt}\int_{\mathbb{T}}\big{(}\frac{1}{2}\rho |u|^2+\frac{\rho^{\gamma}}{\gamma-1}\big{)}(x,t)dx+\int_{\mathbb{T}}(\mu(\rho)|u_{x}|^2)(x,t)dx\\
&\quad=-\int_{\mathbb{T}\times\mathbb{R}}\rho(x,t) (u(x,t)-v)u(x,t)f(x,v,t)dvdx.\label{energyfluid}
\end{split}
\end{equation}
Multiplying $(\ref{m1})_{3}$ by $\frac{1}{2}|v|^2$, and integrating the resulted equation by parts over $\mathbb{T}\times\mathbb{R}$, we obtain
\begin{equation}
\begin{split}
\frac{d}{dt}\int_{\mathbb{T}\times\mathbb{R}} \frac{1}{2}|v|^2f(x,v,t)dvdx=\int_{\mathbb{T}\times\mathbb{R}}\rho(x,t) (u(x,t)-v)vf(x,v,t)dvdx.\label{energyparticle}
\end{split}
\end{equation}
 The combination of (\ref{energyfluid})-(\ref{energyparticle}) gives rise to $(\ref{basiccnsv})_{4}$. The proof of Lemma \ref{lemma21} is completed.
\end{proof}

By Lemma \ref{lemma21}, we can establish the lower and upper bounds of the density $\rho$ below:
\begin{lemma}\label{lemma22}
Let $T>0$, and $(\rho,u,f)$ be any regular solution to the IVP $(\ref{m1})$-$(\ref{kappa})$ for $t\in(0,T]$. Then, under the assumptions of Theorem \ref{theorem11}, it holds
\begin{equation}
\begin{split}
&0<\frac{1}{C_{T}}\leq\rho(x,t)\leq \rho_{+},\quad (x,t)\in\mathbb{T}\times [0,T],\label{rhocnsv}
\end{split}
\end{equation}
where $\rho_{+}>0$ is a constant independent of time $T>0$, and $C_{T}>0$ is a constant dependent of time $T>0$.
\end{lemma}
\begin{proof}
  We introduce a new effective velocity
 $$
 u+\mathcal{I}(n),
  $$
  where the operator $\mathcal{I}:L^1(\mathbb{T})\rightarrow L^{\infty}(\mathbb{T})$ is defined by
  \begin{equation}\label{j}
\begin{split}
\mathcal{I}(g)(x)=:\int^{x}_{0}g(y)dy-\int_{0}^{1}\int_{0}^{y}g(z) dzdy,\quad \forall g\in L^1(\mathbb{T}).
\end{split}
\end{equation}
It can be verified verify for $g\in L^1(\mathbb{T})$ that
\begin{equation}
\begin{split}
 \sup_{x\in\mathbb{T}}\mathcal{I}(g)(x)\leq \|g\|_{L^{1}},\quad \big{(}\mathcal{I}(g)(x)\big{)}_{x}=g,\quad \mathcal{I}\big{(}g-\int_{0}^{1}g(y)dy\big{)}(x)\in W^{1,1}(\mathbb{T}),\label{J}
\end{split}
\end{equation}
and for $g\in W^{1,1}(\mathbb{T})$ that
\begin{equation}\label{J1}
\begin{split}
\mathcal{I}(g_{x})(x)=g(x)-\int_{0}^{1}g(y)dy.
\end{split}
\end{equation}
Integrating the equation $(\ref{m1})_{3}$ over $\mathbb{R}$ with respect to velocity variable $v$, we have
\begin{equation}
\begin{split}
n_{t}+(nw)_{x}=0,\label{nj}
\end{split}
\end{equation}
where $n$ and $w$ are defined by $(\ref{nw})$. Applying the operator $\mathcal{I}$ to the equation $(\ref{nj})$, we deduce by (\ref{J1}) that
\begin{equation}\begin{split}\nonumber
(\mathcal{I}(n))_{t}+nw-\int_{0}^{1}nw(y,t)dy=0,
\end{split}\end{equation}
which together with $(\ref{m1})_{1}$, (\ref{nw}) and (\ref{J}) leads to
\begin{equation}\begin{split}\label{drag}
&(\rho\mathcal{I}(n))_{t}+(\rho u\mathcal{I}(n))_{x}=\rho(\mathcal{I}(n))_{t}+\rho u(\mathcal{I}(n))_{x}\\
&\quad\quad\quad\quad\quad\quad\quad\quad\quad=-\rho nw+\rho u n+\rho\int_{0}^{1}nw(y,t)dy\\
&\quad\quad\quad\quad\quad\quad\quad\quad\quad=\int_{\mathbb{R}}\rho(u-v)fdv+\rho\int_{0}^{1}nw(y,t)dy.
\end{split}\end{equation}
 Thus, by virtue of (\ref{drag}), we can re-write the momentum equation $(\ref{m1})_{2}$ as
\begin{equation}\label{newm1}
\begin{split}
&\big{[}\rho (u+\mathcal{I}(n))\big{]}_{t}+\big{[}\rho u(u+\mathcal{I}(n))\big{]}_{x}+(\rho^{\gamma})_{x}=(\mu(\rho)u_{x})_{x}+\rho \int_{0}^{1}nw(y,t)dy.
\end{split}
\end{equation}
It follows from (\ref{j})-(\ref{J1}) and (\ref{newm1}) that
\begin{equation}\nonumber
\begin{split}
&\big{[}\mathcal{I}\big{(}\rho(u+\mathcal{I}(n))\big{)}\big{]}_{t}\\
&\quad=\int_{0}^{x}\big{[}\rho(u+\mathcal{I}(n))\big{]}_{t}(y,t)dy-\int_{0}^{1}\int_{0}^{y}\big{[}\rho u(u+\mathcal{I}(n))\big{]}_{t}(z,t)dzdy\\
&\quad=-\big{[}\rho u(u+\mathcal{I}(n))+\rho^{\gamma}-\mu(\rho)u_{y}\big{]}|^{y=x}_{y=0}+\int_{0}^{1}nw(y,t)dy\int_{0}^{x}\rho(y,t) dy\\
&\quad\quad+\int_{0}^{1}\big{[}\rho u(u+\mathcal{I}(n))+\rho^{\gamma}-\mu(\rho)u_{z}\big{]}|^{z=y}_{z=0}dy-\int_{0}^{1}nw(y,t)dy\int_{0}^{1}\int_{0}^{y}\rho(z,t) dzdy \\
&\quad=-\rho u(u+\mathcal{I}(n))-\rho^{\gamma}+\mu(\rho)u_{x}\\
&\quad\quad+\int_{0}^{1}\big{[}\rho u(u+\mathcal{I}(n))+\rho^{\gamma}-\mu(\rho)u_{y}\big{]}(y,t)dy+\mathcal{I}(\rho)\int_{0}^{1}nw(y,t)dy,
\end{split}
\end{equation}
which gives rise to
\begin{equation}
\begin{split}
&\big{[}\mathcal{I}\big{(}\rho(u+\mathcal{I}(n))\big{)}\big{]}_{t}+ u\big{[}\mathcal{I}\big{(}\rho(u+\mathcal{I}(n))\big{)}\big{]}_{x}+\rho^{\gamma}\\
&\quad=\mu(\rho)u_{x}+\mathcal{I}(\rho)\int_{0}^{1}nw(y,t)dy+ \int_{0}^{1}\big{[}\rho |u|^2+\rho u\mathcal{I}(n)+\rho^{\gamma}-\mu(\rho)u_{y}\big{]}(y,t)dy.\label{newm2}
\end{split}
\end{equation}
Furthermore, due to the equation $(\ref{m1})_{1}$, it holds
\begin{equation}\label{massrhox}
\begin{split}
\frac{d}{dt}\theta(\rho)+u\theta(\rho)_{x}=\mu(\rho)\frac{\rho_{t}+u\rho_{x}}{\rho}=-\mu(\rho)u_{x},
\end{split}
\end{equation}
where the function $\theta(\rho)$ is given by
\begin{eqnarray}
\label{214} \theta(\rho)=:\int_{1}^{\rho}\frac{\mu(s)}{s}ds=
\begin{cases}
2\log{\rho},
& \mbox{if $\beta=0,$ } \\
\log{\rho}+\frac{\rho^{\beta}-1}{\beta},
& \mbox{if $\beta>0.$}
\end{cases}
\end{eqnarray}
Substituting (\ref{massrhox}) into (\ref{newm2}), and re-writing the resulted equation along the particle path $\mathcal{X}^{x,t}(s)$ for any $(x,t)\in \mathbb{T}\times [0,T]$ defined as
\begin{equation}\label{overlinex}
\left\{ \begin{split}
&\frac{d}{ds}\mathcal{X}^{x,t}(s)=u(\mathcal{X}^{x,t}(s),x),\quad s\in [0,t],\\
&\quad \mathcal{X}^{x,t}(t)=x,
\end{split}
\right.
\end{equation}
 we have
 \begin{equation}\label{newm4}
 \begin{split}
&\frac{d}{ds}\Big{(}\theta(\rho)+\mathcal{I}\big{(}\rho u+\rho\mathcal{I}(n)\big{)}\Big{)}(\mathcal{X}^{x,t}(s),s)+\rho^{\gamma}(\mathcal{X}^{x,t}(s),s)\\
&=F(\mathcal{X}^{x,t}(s),s)-\int_{0}^{1}(\mu(\rho)u_{y})(y,s)dy,
\end{split}
\end{equation}
where $F(x,t)$ is given by
$$
F(x,t)=:\mathcal{I}(\rho)(x,t)\int_{0}^{1}nw(y,t)dy+ \int_{0}^{1}\big{[}\rho |u|^2+\rho u\mathcal{I}(n)+\rho^{\gamma}\big{]}(y,t)dy.
$$
It follows from (\ref{nw}), (\ref{basiccnsv}) and (\ref{J}) for any $(x,t)\in\mathbb{T}\times[0,T]$ that
\begin{equation}\label{F}
\left\{
 \begin{split}
 &|\mathcal{I}\big{(}\rho u+\rho\mathcal{I}(n)\big{)}(x,t)|\leq \|\rho(t)\|_{L^1}^{\frac{1}{2}}\|(\sqrt{\rho}u)(t)\|_{L^2}+\|\rho(t)\|_{L^1}\|n(t)\|_{L^1}\leq C_{2},\\
 &|F(x,t)|\leq  \|\rho(t)\|_{L^1}\|f(t)\|_{\mathcal{L}^{1}}^{\frac{1}{2}}\||v|^2f(t)\|_{\mathcal{L}^1}^{\frac{1}{2}}+\|(\sqrt{\rho}u)(t)\|_{L^2}^2\\
 &\quad\quad\quad~\quad+\|n(t)\|_{L^1}\|\rho(t)\|_{L^{1}}^{\frac{1}{2}}\|(\sqrt{\rho}u)(t)\|_{L^2}+\|\rho^{\gamma}(t)\|_{L^{1}}\leq C_{3},
 \end{split}
 \right.
 \end{equation}
 where $C_{2}>0$ and $C_{3}>0$ are two constants given by
 \begin{equation}\nonumber
\left\{
 \begin{split}
 &C_{2}=:\|\rho_{0}\|_{L^1}^{\frac{1}{2}}(2E_{0})^{\frac{1}{2}}+\|\rho_{0}\|_{L^1}\|f_{0}\|_{\mathcal{L}^1},\\
 &C_{3}=:\big{(}\|\rho_{0}\|_{L^1}\|f_{0}\|_{\mathcal{L}^1}^{\frac{1}{2}}+\|\rho_{0}\|_{L^1}^{\frac{1}{2}}\|f_{0}\|_{\mathcal{L}^1}\big{)}(2E_{0})^{\frac{1}{2}}+(\gamma+1) E_{0}.
 \end{split}
\right.
 \end{equation}
  Since it follows
 \begin{equation}\label{mutheta}
 \left\{
 \begin{split}
 &\mu(\rho)\leq 2,\quad\text{if}\quad0\leq \rho\leq 1,\\
 &\mu(\rho)\leq \beta\theta(\rho)+2,\quad\text{if}\quad\rho\geq 1,
 \end{split}
 \right.
 \end{equation}
we get
  \begin{align}
 &\Big{|}\int_{0}^{1}(\mu(\rho)u_{y})(y,t)dy\Big{|}\nonumber\\
 &\quad\leq \Big{(}\int_{0}^{1}\mu(\rho)(y,t)dy\Big{)}^{\frac{1}{2}}\Big{(}\int_{0}^{1}(\mu(\rho)|u_{y}|^2)(y,t)dy\Big{)}^{\frac{1}{2}}\nonumber\\
&\quad=1+\frac{1}{4} \Big{(}\int_{\{y\in(0,1)|\rho(y,t)\leq 1\}}+\int_{\{y\in(0,1)|\rho(y,t)\geq 1\}}\Big{)}\mu(\rho)(y,t)dy\int_{0}^{1}(\mu(\rho)|u_{y}|^2)(y,t)dy\nonumber\\
&\quad\leq 1+\int_{0}^{1}(\mu(\rho)|u_{y}|^2)(y,t)dy+\frac{\beta}{4}\sup_{x\in\{y\in(0,1)|\rho(y,t)\geq 1\}}\theta(\rho)(x,t)\int_{0}^{1}(\mu(\rho)|u_{y}|^2)(y,t)dy.\label{ux11}
  \end{align}

  Then, the upper bound of $\rho$ uniformly in time can be derived by the arguments due to Zlotnik \cite{zlotnik1}. For any $(x,t)\in\mathbb{T}\times[0,T]$, if there is a time $s\in[0,t]$ such that it holds
 $$
\rho(\mathcal{X}^{x,t}(s),s)> C_{4}=:\max\big{\{}(C_{3}+1)^{\frac{1}{\gamma}},~\sup_{x\in\mathbb{T}}\rho_{0}(x)\big{\}}>1,
$$
 then by the continuity of $\rho(\mathcal{X}^{x,t}(s),s)$ with respect to $s$ and the fact $\rho_{0}(\mathcal{X}^{x,t}(0))\leq C_{4}$, we can find a time $s_{*}\in[0,s)$ to have
 \begin{equation}\label{cutcut}
 \left\{
 \begin{split}
 &\rho(\mathcal{X}^{x,t}(s_{*}),s_{*})=C_{4}, \\
 &\rho(\mathcal{X}^{x,t}(\tau),\tau)>C_{4}\geq (C_{3}+1)^{\frac{1}{\gamma}},\quad \forall \tau\in(s_{*},s].
 \end{split}
 \right.
 \end{equation}
Integrating (\ref{newm4}) over $[s_{*},s]$, we obtain by (\ref{F}) and (\ref{ux11})-(\ref{cutcut}) that
\begin{equation}\nonumber
\begin{split}
&\theta(\rho)(\mathcal{X}^{x,t}(s),s)=\theta(\rho)(\mathcal{X}^{x,t}(s_{*}),s_{*})-\mathcal{I}\big{(}\rho u+\rho\mathcal{I}(n)\big{)}(\mathcal{X}^{x,t}(\tau),\tau)|^{\tau=s}_{\tau=s_{*}}\\
&\quad\quad\quad\quad\quad\quad\quad~~-\int_{s_{*}}^{s}\rho^{\gamma}(\mathcal{X}^{x,t}(\tau),\tau)d\tau+\int_{s_{*}}^{s}F(\mathcal{X}^{x,t}(\tau),\tau)d\tau-\int_{s_{*}}^{s}\int_{0}^{1}(\mu(\rho)u_{y})(y,\tau)dyd\tau\\
&\quad\quad\quad\quad\quad\quad~~\leq \theta(C_{4})+2C_{2}+\int_{s_{*}}^{s}\int_{0}^{1}(\mu(\rho)|u_{y}|^2)(y,\tau)dyd\tau\\
&\quad\quad\quad\quad\quad\quad\quad~~+\frac{\beta}{4}\int_{s_{*}}^{s}\sup_{x\in\{y\in(0,1)|\rho(y,t)\geq 1\}}\theta(\rho)(x,\tau)\int_{0}^{1}(\mu(\rho)|u_{y}|^2)(y,\tau)dyd\tau,
\end{split}
\end{equation}
which together with $(\ref{basiccnsv})_{4}$ and the case $1\leq \rho(\mathcal{X}^{x,t}(s),s)\leq C_{4}$ leads to
\begin{equation}\label{pp}
\begin{split}
&\sup_{x\in\{y\in(0,1)|\rho(y,t)\geq 1\}}\theta(\rho)(x,t)\\
&\quad\quad\leq  \theta(C_{4})+2C_{2}+E_{0}\\
&\quad\quad\quad+\frac{\beta}{4}\int_{0}^{T}\sup_{x\in\{y\in(0,1)|\rho(y,t)\geq 1\}}\theta(\rho)(x,\tau)\int_{0}^{1}(\mu(\rho)|u_{y}|^2)(y,\tau)dyd\tau,\quad t\in[0,T].
\end{split}
\end{equation}
By virtue of $(\ref{basiccnsv})_{4}$, (\ref{pp}) and the Gr${\rm{\ddot{o}}}$nwall's inequality, we have
\begin{equation}\label{upperdd}
\begin{split}
\sup_{x\in\{y\in(0,1)|\rho(y,t)\geq 1\}}\theta(\rho)(x,t)\leq e^{\frac{\beta E_{0}}{4}}\big{(}\theta(C_{4})+2C_{2}+E_{0}\big{)}.
\end{split}
\end{equation}
Thus, we show the upper bound of $\rho$ uniformly in time.

Finally, integrating (\ref{newm4}) over $[0,t]$, one deduces by $(\ref{basiccnsv})_{4}$, the upper bound in (\ref{rhocnsv}), (\ref{overlinex}) and (\ref{F}) for any $(x,t)\in\mathbb{T}\times[0,T]$ that
 \begin{equation}\label{the1234}
 \begin{split}
&\theta(\rho)(x,t)\geq \inf_{x\in\mathbb{T}}\theta(\rho_{0})(x)-\big{|}\mathcal{I}\big{(}\rho u+\rho\mathcal{I}(n)\big{)}(\mathcal{X}^{x,t}(\tau),\tau)|^{\tau=t}_{\tau=0}\big{|}-T\rho_{+}^{\gamma}\\
&\quad\quad\quad\quad\quad-T\sup_{(x,t)\in\mathbb{T}\times[0,T]}|F(x,t)|-\int_{0}^{T}\Big{|}\int_{0}^{1}(\mu(\rho)u_{y})(y,\tau)dy\Big{|}d\tau\\
&\quad\quad\quad\quad\geq \inf_{x\in\mathbb{T}}\theta(\rho_{0})(x)-2C_{2}-(\rho_{+}^{\gamma}+C_{3})T-(\mu(\rho_{+})T)^{\frac{1}{2}}\|\sqrt{\mu(\rho)}u_{x}\|_{L^2(0,T;L^2)}\\
&\quad\quad\quad\quad\geq\inf_{x\in\mathbb{T}}\theta(\rho_{0})(x) -2C_{2} -\frac{\mu(\rho_{+})E_{0}}{4}-(\rho_{+}^{\gamma}+C_{3}+1)T,
 \end{split}\end{equation}
 which together with the fact
 \begin{equation}\label{logrho}
 \left\{
 \begin{split}
 &\theta(\rho)=2\log{\rho},\quad \text{if}~\beta=0,\\
 &\theta(\rho)\leq \log{\rho}+\frac{\rho_{+}^{\beta}}{\beta},\quad\text{if}~\beta>0,
 \end{split}
 \right.
 \end{equation}
yields the time-dependent lower bound of $\rho$. The proof of Lemma \ref{lemma22} is completed.
\end{proof}

\begin{lemma}\label{lemma23}
Let $T>0$, and $(\rho,u,f)$ be any regular solution to the IVP $(\ref{m1})$-$(\ref{kappa})$ for $t\in(0,T]$. Then, under the assumptions of Theorem \ref{theorem11}, it holds
\begin{equation}
\left\{
\begin{split}
&\sup_{t\in[0,T]}\|f(t)\|_{\mathcal{L}^{\infty}}\leq e^{\rho_{+}T}\|f_{0}\|_{\mathcal{L}^{\infty}},\\
&~{\rm{Supp}}_{v}f(x,\cdot,t)\subset \{v\in\mathbb{R}~\big{|}~|v|\leq R_{T}\},\quad (x,t)\in \mathbb{T}\times[0,T],\\
&\sup_{t\in[0,T]}\big{(}\|n(t)\|_{L^{\infty}}+\|\rho_{x}(t)\|_{L^{\infty}}+\|u_{x}(t)\|_{L^2}\big{)}+\|(u_{t},u_{xx})\|_{L^2(0,T;L^2)}\leq C_{T},\label{rhouh1cnsv}
\end{split}
\right .
\end{equation}
where the constant $\rho_{+}>0$ is given by $(\ref{rhocnsv})$, and $R_{T}>0$ and $C_{T}>0$ are two constants.
\end{lemma}
\begin{proof}
For any $(x,v,t)\in \mathbb{T}\times\mathbb{R}\times[0,T]$, define the bi-characteristic curves $(X^{x,v,t}(s),V^{x,v,t}(s))$ by
\begin{equation}\label{XV}
\left\{
\begin{split}
&\frac{d}{ds}X^{x,v,t}(s)=V^{x,v,t}(s),\quad s\in [0,t],\\
&\quad X^{x,v,t}(t)=x,\\
&\frac{d}{ds}V^{x,v,t}(s)=\rho(X^{x,v,t}(s),s) \big{(}u(X^{x,v,t}(s),s)-V^{x,v,t}(s)\big{)},\quad s\in [0,t],\\
&\quad V^{x,v,t}(t)=v.
\end{split}
\right .
\end{equation}
It holds by $(\ref{m1})_{3}$ that
\begin{equation}\label{m33}
\begin{split}
f_{t}+vf_{x}+\rho(u-v)f_{v}-\rho f=0,
\end{split}
\end{equation}
which can be re-written along the bi-characteristic curves $(X^{x,v,t}(s),V^{x,v,t}(s))$ as
$$
\frac{d}{ds}f(X^{x,v,t}(s),V^{x,v,t}(s),s)-\rho(X^{x,v,t}(s),s)f(X^{x,v,t}(s),V^{x,v,t}(s),s)=0.
$$
Thus, we have
\begin{equation}
\begin{split}
f(x,v,t)=e^{\int_{0}^{t}\rho(X^{x,v,t}(s),s)ds}f_{0}(X^{x,v,t}(0),V^{x,v,t}(0)).\label{fformula}
\end{split}
\end{equation}
By (\ref{rhocnsv}) and (\ref{fformula}), we have $(\ref{rhouh1cnsv})_{1}$. Then we solve the equation $(\ref{XV})_{3}$ over $[0,s]$ to get
\begin{equation}\nonumber
\begin{split}
&V^{x,v,t}(s)=e^{-\int_{0}^{s}\rho(X^{x,v,t}(\tau),\tau)d\tau}V^{x,v,t}(0)+\int_{0}^{s}e^{-\int_{\tau}^{s}\rho(X^{x,v,t}(\omega),\omega)d\omega}\rho u(X^{x,v,t}(\tau),\tau)d\tau,
\end{split}
\end{equation}
which shows for $s=t$ that
\begin{equation}\label{xv1}
\begin{split}
&v=e^{-\int_{0}^{t}\rho(X^{x,v,t}(\tau),\tau)d\tau}V^{x,v,t}(0)+\int_{0}^{t}e^{-\int_{\tau}^{t}\rho(X^{x,v,t}(\omega),\omega)d\omega}\rho u(X^{x,v,t}(\tau),\tau)d\tau.
\end{split}
\end{equation}
For any $(x,t)\in\mathbb{T}\times[0,T]$, defining
\begin{equation}\label{Sigmadef}
\begin{split}
 \Sigma(x,t)=:\{v\in\mathbb{R}~\big{|}~f(x,v,t)\neq 0\},
\end{split}
\end{equation}
 we have
\begin{equation}\label{Sigma}
\begin{split}
{\rm{Supp}}_{v} f(x,\cdot,t)\subset \{v\in\mathbb{R}~|~|v|\leq \sup_{v\in\Sigma(x,t)}|v|\}.
\end{split}
\end{equation}
 Due to (\ref{fformula}), it holds
\begin{equation}\label{Sigma1}
\begin{split}
V^{x,v,t}(0)\notin {\rm{Supp}}_{v}f_{0}(X^{x,v,t}(0),\cdot)\subset \{v\in\mathbb{R}~\big{|}~|v|\leq R_{0}\}\Rightarrow f(x,v,t)=0.
\end{split}
\end{equation}
 Therefore, one deduces by $(\ref{a1})_{2}$, (\ref{basiccnsv}), (\ref{rhocnsv}) and (\ref{Sigma1}) for any $(x,t)\in\mathbb{T}\times[0,T]$ that
\begin{equation}\label{compactv1}
\begin{split}
&\sup_{v\in\Sigma(x,t)}|v|\leq R_{0}+\rho_{+}\int_{0}^{t}|u(x,s)|ds\\
&\quad\quad\quad\quad~\leq R_{T}=:R_{0}+\rho_{+}E_{0}^{\frac{1}{2}}\Big{(}T^{\frac{1}{2}}+\frac{2^{\frac{1}{2}}T}{\|\rho_{0}\|_{L^1}^{\frac{1}{2}}}\Big{)},
\end{split}
\end{equation}
where we have used the fact
\begin{equation}\label{uinfty1}
\begin{split}
&u(x,t)=\frac{\int_{\mathbb{T}}\rho(y,t)\int^{x}_{y}u_{z}(z,t)dzdy+\int_{\mathbb{T}}\rho u(y,t)dy}{\int_{\mathbb{T}}\rho(y,t)dy}\\
&\quad\quad~~\leq\|u_{x}(t)\|_{L^2}+\frac{\|(\sqrt{\rho} u)(t)\|_{L^2}}{\|\rho_{0}\|_{L^1}^{\frac{1}{2}}}.
\end{split}
\end{equation}
 The combination of (\ref{Sigma})-(\ref{compactv1}) gives rise to $(\ref{rhouh1cnsv})_{2}$.

Next, we are going to show $(\ref{rhouh1cnsv})_{3}$. Denote the second effective velocity as
\begin{equation}\label{m}
\begin{split}
U=:u+\mathcal{I}(n)+\rho^{-2}\mu(\rho)\rho_{x}.
\end{split}
\end{equation}
It is easy to verify
\begin{equation}\label{bdx}
\begin{split}
&-(\mu(\rho)u_{x})_{x}=(\rho^{-1}\mu(\rho)\rho_{t})_{x}+(\rho^{-1}\mu(\rho)u \rho_{x})_{x}\\
&\quad\quad\quad\quad\quad~=\rho(\rho^{-2}\mu(\rho)\rho_{x})_{t}+\rho u(\rho^{-2}\mu(\rho)\rho_{x})_{x},
\end{split}
\end{equation}
derived from $(\ref{m1})_{1}$, and
\begin{equation}\nonumber
\begin{split}
(\rho^{\gamma})_{x}=\gamma\rho^{\gamma+1}\mu(\rho)^{-1}U-\gamma\rho^{\gamma+1}\mu(\rho)^{-1}(u+\mathcal{I}(n)),
\end{split}
\end{equation}
so that the equation (\ref{newm1}) can be re-written as
 \begin{equation}
\begin{split}
\rho(U_{t}+uU_{x})+\gamma\rho^{\gamma+1}\mu(\rho)^{-1}U=\gamma\rho^{\gamma+1}\mu(\rho)^{-1}(u+\mathcal{I}(n))+\rho\int_{0}^{1}nw(y,t)dy.\label{bd}
\end{split}
\end{equation}
Multiplying (\ref{bd}) by $p|U|^{p-2}U$ for any $p\in[2,\infty)$, and integrating the resulted equation by parts over $\mathbb{T}$, we obtain by  (\ref{basiccnsv}), (\ref{rhocnsv}) and (\ref{J}) that
\begin{align}
&\frac{d}{dt}\|(\rho^{\frac{1}{p}}U)(t)\|_{L^{p}}^{p}+ p\gamma\int_{\mathbb{T}} (\rho^{\gamma+1}\mu(\rho)^{-1}|U|^{p})(x,t)dx\nonumber\\
&= p\gamma\int_{\mathbb{T}}\big{[} \big{(}\gamma\rho^{\gamma+1}\mu(\rho)^{-1}(u+\mathcal{I}(n))+\rho\int_{\mathbb{T}}nw(y,t)dy\big{)}|U|^{p-2}U\big{]}(x,t) dx\nonumber\\
&\leq p\big{(}\gamma\rho_{+}^{\gamma}\|u(t)\|_{L^{\infty}}+\gamma\rho_{+}^{\gamma}\|n(t)\|_{L^1}+\|f(t)\|_{\mathcal{L}^1}^{\frac{1}{2}}\||v|^2f\|_{\mathcal{L}^1}\big{)}\|\rho(t)\|_{L^{1}}^{\frac{1}{p}}\|(\rho^{\frac{1}{p}}U)(t)\|_{L^{p}}^{p-1}\nonumber\\
&\leq p\big{(}\gamma\rho_{+}^{\gamma}\|u(t)\|_{L^{\infty}}+\gamma\rho_{+}^{\gamma}\|f_{0}\|_{\mathcal{L}^1}+(2\|f_{0}\|_{\mathcal{L}^1}E_{0})^{\frac{1}{2}}\big{)}\big{(}\|\rho_{0}\|_{L^1}+1\big{)}^{\frac{1}{2}}\|(\rho^{\frac{1}{p}}U)(t)\|_{L^{p}}^{p-1}\nonumber\\
&\leq pC\big{(}1+\|u(t)\|_{L^{\infty}}\big{)}\|(\rho^{\frac{1}{p}}U)(t)\|_{L^{p}}^{p-1}.\label{the221}
\end{align}
We deduce after dividing (\ref{the221}) by $\big{(}\|(\rho^{\frac{1}{p}}U)(t)\|_{L^{p}}^{p}+\varepsilon^{p}\big{)}^{\frac{p-1}{p}}$ that
\begin{equation}\label{ddtU1}
\begin{split}
&\frac{d}{dt}\big{(}\|(\rho^{\frac{1}{p}}U)(t)\|_{L^{p}}^{p}+\varepsilon^{p}\big{)}^{\frac{1}{p}}\leq C\big{(}1+\|u(t)\|_{L^{\infty}}\big{)}.
\end{split}
\end{equation}
Integrating (\ref{ddtU1}) over $[0,t]$, making use of (\ref{basiccnsv}), (\ref{rhocnsv}), (\ref{J}) and (\ref{uinfty1}), and then taking the limit $\varepsilon\rightarrow 0$, we get
\begin{equation}
\begin{split}
&\sup_{t\in[0,T]}\|(\rho^{\frac{1}{p}}U)(t)\|_{L^{p}}\\
&\quad\leq \|\rho_{0}^{\frac{1}{p}}\big{(}u_{0}+\mathcal{I}(\int_{\mathbb{R}}f_{0}(\cdot,v)dv)+\rho_{0}^{-2}\mu(\rho_{0})(\rho_{0})_{x}\big{)}\|_{L^{p}}+C\big{(}1+\|u\|_{L^1(0,T;L^{\infty})}\big{)}\\
&\quad\leq \big{(}1+\|\rho_{0}\|_{L^1}\big{)}^{\frac{1}{2}}\big{(}\|u_{0}\|_{L^{\infty}}+\|f_{0}\|_{\mathcal{L}^1}+\|\rho_{0}^{-2}\mu(\rho_{0})(\rho_{0})_{x}\|_{L^{\infty}}\big{)}\\
&\quad\quad+C\big{(}1+T^{\frac{1}{2}}\|u_{x}\|_{L^2(0,T;L^2)}+T\|\sqrt{\rho}u\|_{L^{\infty}(0,T;L^2)}\big{)}\\
&\quad\leq C_{5}(1+T),\label{Ulpf}
\end{split}
\end{equation}
where $C_{5}>0$ is a sufficiently large constant independent of time $T>0$ and $p\in [2,\infty)$.
By (\ref{rhocnsv}) and (\ref{Ulpf}), we have
\begin{equation}\label{Ulp}
\begin{split}
&\sup_{t\in[0,T]}\|U(t)\|_{L^{p}}\leq C_{T}^{\frac{1}{p}}\sup_{t\in[0,T]}\|(\rho^{\frac{1}{p}}U)(t)\|_{L^{p}}\leq (C_{T}+1)^{\frac{1}{2}}C_{5}(1+T).
\end{split}
\end{equation}
Since the constant on the right-hand side of (\ref{Ulp}) is independent of $p\in[2,\infty)$, one can take the limit $p\rightarrow \infty$ in (\ref{Ulp}) to obtain
\begin{equation}
\begin{split}
\sup_{t\in[0,T]}\|U(t)\|_{L^{\infty}}\leq (C_{T}+1)^{\frac{1}{2}}C_{5}(1+T).\label{rhow1p}
\end{split}
\end{equation}
Choosing $p=2$ in (\ref{Ulp}), we deduce by (\ref{basiccnsv}), (\ref{rhocnsv}) and (\ref{m}) that
\begin{equation}\label{rhoh1time}
\begin{split}
&\sup_{t\in[0,T]}\|\rho_{x}(t)\|_{L^2} \leq \rho_{+}^{2}\sup_{t\in[0,T]}\|\big{(}U-u-\mathcal{I}(n)\big{)}(t)\|_{L^2}\\
&\quad\quad\quad\quad\quad\quad\quad\leq \rho_{+}^{2} \sup_{t\in[0,T]}\big{(}\|U(t)\|_{L^2}+\|u(t)\|_{L^2}+\|n(t)\|_{L^1}\big{)}\leq C_{T}.
\end{split}
\end{equation}

Then, we divide the equation $(\ref{m1})_{2}$ by $\rho$ to get
\begin{equation}
\begin{split}
 u_{t}-\mu(\rho)\rho^{-1} u_{xx}=G=:nw-nu-\gamma \rho^{\gamma-2}\rho_{x}-uu_{x}+\beta\rho^{\beta-2}\rho_{x}u_{x}.\label{unorho}
\end{split}
\end{equation}
By $(\ref{nw})$ and $(\ref{rhouh1cnsv})_{1}$-$(\ref{rhouh1cnsv})_{2}$, one has
\begin{equation}\label{nl2}
\begin{split}
&\sup_{t\in[0,T]}\|n(t)\|_{L^{\infty}}\leq 2R_{T}\sup_{t\in[0,T]}\|f(t)\|_{\mathcal{L}^{\infty}}\leq 2R_{T}e^{\rho_{+}T}\|f_{0}\|_{\mathcal{L}^{\infty}}.
\end{split}
\end{equation}
It follows from (\ref{basiccnsv}), (\ref{rhocnsv}), (\ref{rhoh1time}), (\ref{nl2}) and the Gagliardo-Nirenberg's inequality that
\begin{equation}\label{GL2}
\begin{split}
&\|G\|_{L^2(0,T;L^2)}\leq \|n\|_{L^{\infty}(0,T;L^{\infty})}^{\frac{1}{2}}\|\sqrt{f}(u-v)\|_{L^2(0,T;\mathcal{L}^2)}+\gamma\|\rho^{\gamma-2}\rho_{x}\|_{L^{\infty}(0,T;L^{2})}\\
&\quad\quad\quad\quad\quad\quad\quad+\big{(}\|u\|_{L^{\infty}(0,T;L^{2})}+\beta\|\rho^{\beta-2}\rho_{x}\|_{L^{\infty}(0,T;L^{2})}\big{)}\|u_{x}\|_{L^2(0,T;L^{\infty})}\\
&\quad\quad\quad\quad\quad\quad\leq C_{T}\big{(}1+\|u_{x}\|_{L^2(0,T;L^{2})}^{\frac{1}{2}}\|u_{xx}\|_{L^2(0,T;L^{2})}^{\frac{1}{2}}\big{)}\\
&\quad\quad\quad\quad\quad\quad\leq C_{T}(1+\|u_{xx}\|_{L^2(0,T;L^{2})}^{\frac{1}{2}}).
\end{split}
\end{equation}
Multiplying (\ref{unorho}) by $-u_{xx}$, and integrating the resulted equation by parts over $\mathbb{T}\times[0,t]$, one concludes from (\ref{GL2}) and the Young's inequality that
\begin{equation}\nonumber
\begin{split}
&\|u_{x}(t)\|_{L^2}^2+\|u_{xx}\|_{L^2(0,t;L^2)}^2\\
&\quad\leq C\big{(}\|(u_{0})_{x}\|_{L^2}^2+\|u_{xx}\|_{L^2(0,t;L^2)}\|G\|_{L^2(0,t;L^2)}\big{)}\\
&\quad\leq C+C_{T}\big{(}\|u_{xx}\|_{L^2(0,t;L^2)}+\|u_{xx}\|_{L^2(0,t;L^2)}^{\frac{3}{2}}\big{)}\\
&\quad\leq C_{T}+\frac{1}{2}\|u_{xx}\|_{L^2(0,t;L^2)}^2,
\end{split}
\end{equation}
from which we infer
\begin{equation}\label{uh1}
\begin{split}
\sup_{t\in[0,T]}\|u_{x}(t)\|_{L^2}+\|u_{xx}\|_{L^2(0,T;L^2)}\leq C_{T}.
\end{split}
\end{equation}
By (\ref{rhocnsv}) and (\ref{unorho})-(\ref{uh1}), it also holds
\begin{equation}\label{utl2}
\begin{split}
\|u_{t}\|_{L^2(0,T;L^2)}=\|\mu(\rho)\rho^{-1} u_{xx}+G\|_{L^2(0,T;L^2)}\leq C_{T}.
\end{split}
\end{equation}
One may get from $(\ref{basiccnsv})_{3}$, (\ref{J}), (\ref{m}), (\ref{rhow1p}), (\ref{uh1}) and the Sobolev embedding $H^1(\mathbb{T})\hookrightarrow L^{\infty}(\mathbb{T})$ that
\begin{equation}\label{rhow1infty}
\begin{split}
&\sup_{t\in[0,T]}\|\rho_{x}(t)\|_{L^{\infty}}\leq \rho_{+}^{2} \sup_{t\in[0,T]}\|\big{(}U-u-\mathcal{I}(n)\big{)}(t)\|_{L^{\infty}}\\
&\quad\quad\quad\quad\quad\quad\quad~\leq \rho_{+}^{2} \sup_{t\in[0,T]}\big{(}\|U(t)\|_{L^{\infty}}+\|u(t)\|_{L^{\infty}}+\|n(t)\|_{L^1}\big{)}\leq C_{T}.
\end{split}
\end{equation}
The combination of (\ref{nl2}) and (\ref{uh1})-(\ref{rhow1infty}) leads to $(\ref{rhouh1cnsv})_{3}$. The proof of Lemma \ref{lemma23} is completed.
\end{proof}

 We are ready to establish the higher order estimates of $(\rho,u,f)$.
\begin{lemma}\label{lemma24}
Let $T>0$, and $(\rho,u,f)$ be any regular solution to the IVP $(\ref{m1})$-$(\ref{kappa})$ for $t\in(0,T]$. Then, under the assumptions of Theorem \ref{theorem11}, it holds
\begin{equation}
\begin{split}
&\sup_{t\in[0,T]}\Big{(}\|(\rho_{xx},u_{xx})(t)\|_{L^2}+\|\rho_{t}(t)\|_{H^1}+t^{\frac{1}{2}}\|(u_{xxx},u_{tx})(t)\|_{L^2}\\
&\quad\quad\quad\quad+\|f(t)\|_{\mathcal{C}^{1}}+\|f_{t}(t)\|_{\mathcal{C}^{0}}\Big{)}+\|(u_{xxx},u_{xt})\|_{L^2(0,T;L^2)}\leq C_{T},\label{rhoufh2cnsv}
\end{split}
\end{equation}
where $C_{T}>0$ is a constant.
\end{lemma}
\begin{proof}
Differentiating $(\ref{m1})_{3}$ with respect to $x$ and $v$ respectively, we have
\begin{equation}\label{fxfv}
\left\{
\begin{split}
&(f_{v})_{t}+v(f_{v})_{x}+\rho(u-v)(f_{v})_{v}-2\rho f_{v}=-f_{x},\\
&(f_{x})_{t}+v(f_{x})_{x}+\rho(u-v)(f_{x})_{v}-\rho f_{x}=\rho_{x}vf_{v}-(\rho u)_{x}f_{v}+\rho_{x}f.
\end{split}
\right.
\end{equation}
Thus, $f_{x}$ and $f_{v}$ can be re-written along the bi-characteristic curves $(X^{x,v,t}(s),V^{x,v,t}(s))$ defined by (\ref{XV}) for any $(x,v,t)\in\mathbb{T}\times\mathbb{R}\times[0,T]$ as
\begin{equation}\label{fxfv1}
\left\{
\begin{split}
&f_{v}(x,v,t)=e^{2\int_{0}^{t}\rho(X^{x,v,t}(s),s)ds}(f_{0})_{v}(X^{x,v,t}(0),V^{x,v,t}(0))\\
&\quad\quad\quad\quad\quad-\int_{0}^{t}e^{2\int_{s}^{t}\rho(X^{x,v,t}(\tau),\tau)d\tau}f_{x}(X^{x,v,t}(s),V^{x,v,t}(s),s)ds,\\
&f_{x}(x,v,t)=e^{\int_{0}^{t}\rho(X^{x,v,t}(s),s)ds}(f_{0})_{x}(X^{x,v,t}(0),V^{x,v,t}(0))\\
&\quad\quad\quad\quad\quad+\int_{0}^{t}e^{\int_{s}^{t}\rho(X^{x,v,t}(\tau),\tau)d\tau}\big{[}\rho_{x}vf_{v}-(\rho u)_{x}f_{v}+\rho_{x}f\big{]}(X^{x,v,t}(s),V^{x,v,t}(s),s)ds,
\end{split}
\right .
\end{equation}
which together with (\ref{rhocnsv}) and (\ref{rhouh1cnsv}) leads to
\begin{equation}\label{fvinfty}
\begin{split}
&\sup_{t\in[0,T]}\|f_{v}(t)\|_{\mathcal{L}^{\infty}}\leq e^{2\rho_{+}T}\Big{(}\|(f_{0})_{v}\|_{\mathcal{L}^{\infty}}+\int_{0}^{T}\|f_{x}(t)\|_{\mathcal{L}^{\infty}}dt\Big{)},
\end{split}
\end{equation}
and
\begin{equation}\label{fxinfty}
\begin{split}
&\sup_{t\in[0,T]}\|f_{x}(t)\|_{\mathcal{L}^{\infty}}\leq e^{\rho_{+}T}\Big{(}\|(f_{0})_{x}\|_{\mathcal{L}^{\infty}}+\int_{0}^{T}\|\big{[}\rho_{x}vf_{v}-(\rho u)_{x}f_{v}+\rho_{x}f\big{]}(t)\|_{\mathcal{L}^{\infty}}dt\Big{)}\\
&\quad\quad\quad\quad\quad\quad\quad~\leq e^{\rho_{+}T}\Big{(}\|(f_{0})_{x}\|_{\mathcal{L}^{\infty}}+\int_{0}^{T}\big{(}\|\rho_{x}(t)\|_{L^{\infty}}R_{T}\|f_{v}(t)\|_{\mathcal{L}^{\infty}}\\
&\quad\quad\quad\quad\quad\quad\quad\quad~+\|(\rho u)_{x}(t)\|_{L^{\infty}}\|f_{v}(t)\|_{\mathcal{L}^{\infty}}+\|\rho_{x}(t)\|_{L^{\infty}}\|f(t)\|_{\mathcal{L}^{\infty}}\big{)}dt\Big{)}\\
&\quad\quad\quad\quad\quad\quad\quad~\leq C_{T}\Big{(}\|(f_{0})_{x}\|_{\mathcal{L}^{\infty}}+1+\int_{0}^{T}\big{(}1+\|u(t)\|_{H^2}\big{)}\|f_{v}(t)\|_{\mathcal{L}^{\infty}}dt\Big{)},
\end{split}
\end{equation}
where the constants $\rho_{+}>0$ and $R_{T}>0$ are given by (\ref{rhocnsv}) and $(\ref{rhouh1cnsv})_{2}$ respectively. Substituting (\ref{fxinfty}) into (\ref{fvinfty}), and employing a Gr${\rm{\ddot{o}}}$nwall type argument, we obtain
\begin{equation}\label{fw1inftycnsv}
\begin{split}
\sup_{t\in[0,T]}\|(f_{x},f_{v})(t)\|_{\mathcal{L}^{\infty}}\leq C_{T}e^{\|u\|_{L^1(0,T;H^2)}}\big{(}\|[(f_{0})_{x},(f_{0})_{v}]\|_{\mathcal{L}^{\infty}}+1\big{)}\leq C_{T}.
\end{split}
\end{equation}
By (\ref{rhocnsv}), (\ref{rhouh1cnsv}),  $(\ref{m33})$ and (\ref{fw1inftycnsv}), it holds
\begin{equation}
\begin{split}
&\sup_{t\in[0,T]}\|f_{t}(t)\|_{\mathcal{L}^{\infty}}\\
&\quad=\sup_{t\in[0,T]}\|\big{[}vf_{x}+\rho(u-v)f_{v}-\rho f\big{]}(t)\|_{\mathcal{L}^{\infty}}\\
&\quad\leq \sup_{t\in[0,T]}\Big{(}R_{T}\|f_{x}(t)\|_{\mathcal{L}^{\infty}}+\rho_{+}\big{(}\|u(t)\|_{L^{\infty}}+R_{T}\big{)}\|f_{v}(t)\|_{\mathcal{L}^{\infty}}+\rho_{+}\|f(t)\|_{\mathcal{L}^{\infty}}\Big{)}\leq C_{T}.\label{fw1inftycnsv2}
\end{split}
\end{equation}
Then one deduces by (\ref{rhouh1cnsv}) and (\ref{fw1inftycnsv})-(\ref{fw1inftycnsv2}) that
\begin{equation}\label{nxt}
\left\{
\begin{split}
&\sup_{t\in[0,T]}\|(n_{x},n_{t})(t)\|_{L^{\infty}}\leq 2R_{T}\sup_{t\in[0,T]}\|(f_{x},f_{t})(t)\|_{\mathcal{L}^{\infty}}\leq C_{T},\\
&\sup_{t\in[0,T]}\|[(nw)_{x},(nw)_{t}](t)\|_{ L^{\infty}}\leq 2R_{T}^2\sup_{t\in[0,T]}\|(f_{x},f_{t})(t)\|_{\mathcal{L}^{\infty}}\leq C_{T}.
\end{split}
\right.
\end{equation}

We turn to establish the $H^2$-norm estimate of $\rho$. Dividing (\ref{bd}) by $\rho$, and differentiating the resulted equation with respect to $x$, we have
\begin{equation}\label{Ux}
\begin{split}
&\frac{d}{dt}U_{x}+\gamma\rho^{\gamma}\mu(\rho)^{-1}U_{x}=-(uU_{x})_{x}-\gamma(\rho^{\gamma}\mu(\rho)^{-1})_{x}\rho^{-2}\mu(\rho)\rho_{x}+\gamma\rho^{\gamma}\mu(\rho)^{-1}(u_{x}+n).
\end{split}
\end{equation}
Multiplying (\ref{Ux}) by $U_{x}$, and integrating the resulted equation by parts over $\mathbb{T}$, one deduces by (\ref{rhocnsv}) and (\ref{rhouh1cnsv}) that
\begin{equation}\nonumber
\begin{split}
&\frac{1}{2}\frac{d}{dt}\|U_{x}(t)\|_{L^2}^2+\gamma\int_{\mathbb{T}}(\rho^{\gamma}\mu(\rho)^{-1}|U_{x}|^2)(x,t)dx\\
&~~\leq \frac{1}{2}\|u_{x}(t)\|_{L^{\infty}}\|U_{x}(t)\|_{L^2}^2+C_{T}\big{(}\|\rho_{x}(t)\|_{L^2}\|\rho_{x}(t)\|_{L^{\infty}}+\|u_{x}(t)\|_{L^{\infty}}+\|n(t)\|_{L^2}\big{)}\|U_{x}(t)\|_{L^2}\\
&~~\leq C_{T}\big{(}\|u_{x}(t)\|_{H^1}+1\big{)}\|U_{x}(t)\|_{L^2}^2+C_{T}\|u_{x}(t)\|_{H^1}^2+C_{T},
\end{split}
\end{equation}
which implies
\begin{equation}
\begin{split}
&\sup_{t\in[0,T]}\|U_{x}(t)\|_{L^2}^2\leq e^{C_{T}(\|u_{x}\|_{L^1(0,T;H^1)}+1)}\big{(}\|U_{x}(0)\|_{L^2}^2+C_{T}\|u_{x}\|_{L^2(0,T;H^1)}^2+C_{T}\big{)}.\label{rhoh22}
\end{split}
\end{equation}
Thereby, it follows from  (\ref{rhocnsv}), (\ref{rhouh1cnsv}), (\ref{m}) and (\ref{rhoh22}) that
\begin{equation}\label{rhoh2}
\begin{split}
&\sup_{t\in[0,T]}\|\rho_{xx}(t)\|_{L^2}\leq \rho_{+}^2\sup_{t\in[0,T]}\|\big{(}U_{x}-u_{x}-n\big{)}(t)\|_{L^2}\\
&\quad\quad\quad\quad\quad\quad\quad~\leq \rho_{+}^2\sup_{t\in[0,T]}\big{(}\|U_{x}(t)\|_{L^2}+\|u_{x}(t)\|_{L^2}+\|n(t)\|_{L^2}\big{)}\leq C_{T}.
\end{split}
\end{equation}

To establish the $H^2$-estimate of $u$, we differentiate the equation (\ref{unorho}) with respect to variable $x$ to have
\begin{equation}
\begin{split}
u_{xt}-\mu(\rho)\rho^{-1}u_{xxx}=G_{x}+(\mu(\rho)\rho^{-1})_{x}u_{xx}.\label{ux}
\end{split}
\end{equation}
Multiplying (\ref{ux}) by $-u_{xxx}$, integrating the resulted equation by parts over $\mathbb{T}\times[0,t]$, and making use of (\ref{basiccnsv}), (\ref{rhocnsv}), (\ref{rhouh1cnsv}), (\ref{nxt}) and (\ref{rhoh2}), one may get after a complicated computation  that
\begin{equation}\label{rhouh2cnsv}
\begin{split}
&\sup_{t\in[0,T]}\|u_{xx}(t)\|_{L^2}^2+\|u_{xxx}\|_{L^2(0,T;L^2)}^2\\
&\quad\leq C\big{(}\|(u_{0})_{x}\|_{H^1}^2+\|G_{x}+(\mu(\rho)\rho^{-1})_{x}u_{xx}\|_{L^{2}(0,T;L^2)}^2\big{)}\\
&\quad\leq C_{T}\big{(}1+\|(n,nw)\|_{L^{\infty}(0,T;H^1)}^2+\|\rho\|_{L^{\infty}(0,T;H^2)}^2+\|\rho_{x}\|_{L^{\infty}(0,T;L^{\infty})}\|u_{xx}\|_{L^2(0,T;L^2)}\\
&\quad\quad+\|\rho_{xx}\|_{L^{\infty}(0,T;L^2)}\|u_{x}\|_{L^2(0,T;L^{\infty})}\big{)}\leq C_{T}.
\end{split}
\end{equation}
By $(\ref{m1})_{1}$, (\ref{unorho}) and (\ref{rhoh2})-(\ref{rhouh2cnsv}), it also holds
 \begin{equation}\label{rhouh2cnsv1}
 \left\{
 \begin{split}
&\sup_{t\in {0,T}}\|\rho_{t}(t)\|_{H^1}= \sup_{t\in [0,T]}\|(\rho u)_{x}(t)\|_{H^1}\leq C\sup_{t\in [0,T]}\|\rho(t)\|_{H^2}\|u(t)\|_{H^2}\leq C_{T},\\
&~~\|u_{xt}\|_{L^2(0,T;L^2)}=\|\rho^{-1}\mu(\rho)u_{xxx}+G_{x}+(\mu(\rho)\rho^{-1})_{x}u_{xx}\|_{L^2(0,T;L^2)}\leq C_{T}.
\end{split}
\right.
\end{equation}

We are ready to establish the time-weighted estimates of $u$. We differentiate  (\ref{unorho}) with respect to $t$ to get
 \begin{equation}\label{ut}
 \begin{split}
 u_{tt}-\mu(\rho)\rho^{-1}u_{txx}=G_{t}+(\mu(\rho)\rho^{-1})_{t}u_{xx}.
\end{split}
\end{equation}
Multiplying (\ref{ut}) by $-tu_{txx}$, integrating the resulted equation by parts over $\mathbb{T}$, and making use of the estimates (\ref{rhocnsv}), (\ref{nxt}), (\ref{rhoh2}) and (\ref{rhouh2cnsv})-(\ref{rhouh2cnsv1}), we have
 \begin{equation}\label{buxiang}
 \begin{split}
 &\frac{1}{2}\frac{d}{dt}\Big{(}t\|u_{xt}(t)\|_{L^2}^2\Big{)}+\frac{t}{\rho_{+}}\|u_{xxt}(t)\|_{L^2}^2\\
 &~~\leq\frac{1}{2}\|u_{xt}(t)\|_{L^2}^2+t\|u_{xxt}(t)\|_{L^2}\|[G_{t}+(\mu(\rho)\rho^{-1})_{t}u_{xx}](t)\|_{L^2}\\
  &~~\leq \frac{1}{2}\|u_{xt}(t)\|_{L^2}^2+C_{T}t\|u_{xxt}(t)\|_{L^2} \big{(} 1+\|u_{xt}(t)\|_{L^2} \big{)}\\
  &~~\leq\frac{t}{2\rho_{+}}\|u_{xxt}(t)\|_{L^2}^2+ C_{T}t\|u_{xt}(t)\|_{L^2}^2+\frac{1}{2}\|u_{xt}(t)\|_{L^2}^2+C_{T}.
 \end{split}
 \end{equation}
Thus, one deduces by (\ref{rhouh2cnsv1}), (\ref{buxiang}), the Gr$\rm{\ddot{o}}$nwall's inequality and the fact $tu_{xt}|_{t=0}=0$ that
 \begin{equation}\label{tut}
 \begin{split}
 \sup_{t\in[0,T]}t^{\frac{1}{2}}\|u_{xt}(t)\|_{L^2}\leq C_{T}.
 \end{split}
 \end{equation}
By (\ref{rhocnsv}), (\ref{rhoh2})-(\ref{rhouh2cnsv1}) and (\ref{tut}), it also holds
  \begin{equation}\label{tux}
 \begin{split}
 &\sup_{t\in[0,T]}t^{\frac{1}{2}}\|u_{xxx}(t)\|_{L^2}\\
 &\quad\leq \rho_{+} \sup_{t\in[0,T]}t^{\frac{1}{2}}\|\rho^{-1}\mu(\rho)u_{xxx}(t)\|_{L^2}\\
 &\quad= \rho_{+}\sup_{t\in[0,T]}t^{\frac{1}{2}}\|\big{[}u_{xt}-G_{x}-(\mu(\rho)\rho^{-1})_{x}u_{xx}\big{]}(t)\|_{L^2}\leq C_{T}.
 \end{split}
 \end{equation}
 Since it follows by (\ref{rhoh2}) and (\ref{rhouh2cnsv}) that
 $$
 \sup_{t\in[0,T]}\|(\rho,u)(t)\|_{C^{1}}\leq C \sup_{t\in[0,T]}\|(\rho,u)(t)\|_{H^{2}}\leq C_{T},
 $$
 similarly to (\ref{fxfv})-(\ref{fw1inftycnsv2}),  one can show
  \begin{equation}\label{fc1t}
 \begin{split}
 &\sup_{0\leq t\leq T}\big{(}\|f(t)\|_{\mathcal{C}^1}+\|f_{t}(t)\|_{\mathcal{C}^{0}}\big{)}\leq C_{T}.
  \end{split}
 \end{equation}
The combination of (\ref{fw1inftycnsv})-(\ref{fw1inftycnsv2}), (\ref{rhoh2}), (\ref{rhouh2cnsv})-(\ref{rhouh2cnsv1}) and (\ref{tut})-(\ref{fc1t}) gives rise to (\ref{rhoufh2cnsv}). The proof of Lemma \ref{lemma24} is completed.
\end{proof}

Inspired by the arguments as in \cite{straskraba1}, we turn to show that the $L^1$-norm of the pressure $P(\rho)$ is strictly positive for large time, which is essential to establish the lower bound of $\rho$ uniformly in time.
\begin{lemma}\label{lemma25}
Let $(\rho,u,f)$ be any regular solution to the IVP $(\ref{m1})$-$(\ref{kappa})$. Then, under the assumptions of Theorem \ref{theorem11}, there is a sufficiently large time $T_{0}>0$ such that it holds
\begin{equation}\begin{split}
& \int_{\mathbb{T}} \rho^{\gamma}(x,t)dx\geq P_{0}>0,\quad t\geq T_{0},\label{pressure}
\end{split}\end{equation}
where $P_{0}>0$ is a constant independent of time $t$.
\begin{proof}
We claim that it holds
\begin{equation}\begin{split}
\lim_{t \to +\infty}\|(\rho-\overline{\rho_{0}})(t)\|_{L^2}=0,\label{rhobehaviorlp}
\end{split}\end{equation}
where the constant $\overline{\rho_{0}}>0$ is given by $(\ref{rhoinfty})_{1}$. By virtue of (\ref{rhobehaviorlp}), one can prove
\begin{equation}\begin{split}\nonumber
&\lim_{t\rightarrow \infty}\Big{|}\int_{\mathbb{T}}\rho^{\gamma}(x,t)dx-\overline{\rho_{0}}^{\gamma}\Big{|}\\
&\quad\leq \lim_{t\rightarrow \infty}\int_{\mathbb{T}}|\rho^{\gamma}(x,t)-\overline{\rho_{0}}^{\gamma}|dx\\
&\quad\leq \gamma(\rho_{+}+\overline{\rho_{0}})^{\gamma-1} \lim_{t\rightarrow \infty}\|(\rho-\overline{\rho_{0}})(t)\|_{L^2}=0,
\end{split}\end{equation}
which gives rise to (\ref{pressure}).

Indeed, the long time behavior (\ref{rhobehaviorlp}) can be shown by Lemmas \ref{lemma21}-\ref{lemma22} and relative entropy estimates for the compressible Navier-Stokes equations $(\ref{m1})_{1}$-$(\ref{m1})_{2}$. Define
\begin{equation}\nonumber
\left\{
\begin{split}
&\mathcal{E}_{NS}^{\eta}(t)=:\int_{\mathbb{T}}\big{(}\frac{1}{2}\rho|u-m_{1}|^2+\Pi_{\gamma}(\rho|\overline{\rho_{0}})\big{)}(x,t)dx-\eta\int_{\mathbb{T}}[\rho(u-m_{1})\mathcal{I}(\rho-\overline{\rho_{0}})](x,t)dx,\\
 &\mathcal{D}_{NS}^{\eta}(t)=:\int_{\mathbb{T}}(\mu(\rho)|u_{x}|^2)(x,t)dx+\eta\int_{\mathbb{T}}[(\rho^{\gamma}-\overline{\rho_{0}}^{\gamma})(\rho-\overline{\rho_{0}})](x,t)dx\\
&\quad\quad\quad\quad-\eta\int_{\mathbb{T}}\big{[}\overline{\rho_{0}}\rho|u-m_{1}|^2+\mu(\rho)u_{x}(\rho-\overline{\rho_{0}})\big{]}(x,t)dx\\
&\quad\quad\quad\quad+\eta\int_{\mathbb{T}}\big{[}\mathcal{I}(\rho-\overline{\rho_{0}})\big{(}-\rho n(u-w)+\frac{\rho}{\overline{\rho_{0}}}\int_{\mathbb{T}} [\rho n (u-w)](y,t)dy\big{)}\big{]}(x,t)dx,
\end{split}
\right.
\end{equation}
where the operator $\mathcal{I}$ is denoted as (\ref{j}), and $m_{1}(t)$ and $\Pi_{\gamma}(\rho|\overline{\rho_{0}})$ are given by
\begin{equation}\label{m1m2}
\left\{
\begin{split}
&m_{1}(t)=:\frac{\int_{\mathbb{T}}\rho u(x,t)dx}{\int_{\mathbb{T}}\rho_{0}(x)dx},\\
&\Pi_{\gamma}(\rho|\overline{\rho_{0}})=:\frac{\rho^{\gamma}}{\gamma-1}-\frac{\overline{\rho_{0}}^{\gamma}}{\gamma-1}-\frac{\gamma\overline{\rho_{0}}^{\gamma-1}}{\gamma-1}(\rho-\overline{\rho_{0}}).
\end{split}
\right.
\end{equation}
It is easy to verify
\begin{equation}\label{entropyfluid22}
\begin{split}
&\frac{d}{dt}\mathcal{E}_{NS}^{\eta}(t)+\mathcal{D}_{NS}^{\eta}(t)=-\int_{\mathbb{T}}[\rho n(u-w)(u-m_{1})](x,t)dx.
\end{split}
\end{equation}
One can choose a suitably small constant $\eta>0$ to have
\begin{equation}\label{ENSsim}
\left\{
\begin{split}
&\mathcal{E}_{NS}^{\eta}(t)\leq C_{6}(\|(\rho-\overline{\rho_{0}})(t)\|_{L^2}^2+\|\big{(}\sqrt{\rho}(u-m_{1})\big{)}(t)\|_{L^2}^2),\\
&\mathcal{E}_{NS}^{\eta}(t)\geq \frac{1}{C_{6}}(\|(\rho-\overline{\rho_{0}})(t)\|_{L^2}^2+\|\big{(}\sqrt{\rho}(u-m_{1})\big{)}(t)\|_{L^2}^2),\\
&\mathcal{D}^{\eta}_{NS}(t) \geq \frac{1}{2}\|u_{x}(t)\|_{L^2}^2+ C_{7}\mathcal{E}_{NS}^{\eta}(t)-C_{8}\|\big{(}\sqrt{\rho f}(u-v)\big{)}(t)\|_{\mathcal{L}^2}^2,
\end{split}
\right.
\end{equation}
where $C_{6}>1$, $C_{i}>0, i=7,8,$ are constants independent of time $t$, the interested reader can refer to \cite{lhl1} for the details.

In addition, one can show
\begin{equation}\label{pp3}
\begin{split}
&\Big{|}\int_{\mathbb{T}}[\rho n(u-w)(u-m_{1})](x,t)dx\Big{|}\\
&\quad\leq \rho_{+}^{\frac{1}{2}}\|(u-m_{1})(t)\|_{L^{\infty}}\|f(t)\|_{\mathcal{L}^1}^{\frac{1}{2}}\|\big{(}\sqrt{\rho f}(u-v)\big{)}(t)\|_{\mathcal{L}^2}\\
&\quad\leq\frac{1}{2}\|u_{x}(t)\|_{L^2}^2+\frac{\rho_{+}\|f_{0}\|_{\mathcal{L}^1}}{2}\|\big{(}\sqrt{\rho f}(u-v)\big{)}(t)\|_{\mathcal{L}^2}^2,
\end{split}
\end{equation}
where we have used $(\ref{basiccnsv})_{2}$, (\ref{rhocnsv}) and the fact
\begin{equation}\label{um1infty}
\begin{split}
&(u-m_{1})(x,t)=\frac{\int_{\mathbb{T}} \rho(y,t)\int^{x}_{y}u_{z}(z,t)dzdy}{\int_{\mathbb{T}} \rho(y,t)dy}\leq \|u_{x}(t)\|_{L^2},\quad x\in\mathbb{T},\quad t>0.
\end{split}
\end{equation}
Thus, it follows from (\ref{entropyfluid22})-(\ref{pp3}) that
\begin{equation}\label{ddtens}
\begin{split}
&\frac{d}{dt}\mathcal{E}_{NS}^{\eta}(t)+C_{7}\mathcal{E}_{NS}^{\eta}(t)\leq C_{9}\|\big{(}\sqrt{\rho f}(u-v)\big{)}(t)\|_{\mathcal{L}^2}^2,\quad C_{9}=: \frac{\rho_{+}\|f_{0}\|_{\mathcal{L}^1}+2C_{8}}{2}.
\end{split}
\end{equation}
Applying the the Gr${\rm{\ddot{o}}}$nwall's inequality to (\ref{ddtens}), we get
\begin{equation}\label{entropyfluid222}
\begin{split}
&\mathcal{E}_{NS}^{\eta}(t)\leq e^{-C_{7} t}\mathcal{E}_{NS}^{\eta}(0)+C_{9}e^{-\frac{C_{7}t}{2}}\int_{0}^{\frac{t}{2}}\|\big{(}\sqrt{\rho f}(u-v)\big{)}(s)\|_{\mathcal{L}^2}^2ds\\
&\quad\quad\quad\quad+C_{9}\int_{\frac{t}{2}}^{t}\|\big{(}\sqrt{\rho f}(u-v)\big{)}(s)\|_{\mathcal{L}^2}^2ds.
\end{split}
\end{equation}
Since it follows $\|\big{(}\sqrt{\rho f}(u-v)\big{)}(s)\|_{\mathcal{L}^2}^2\in L^1(\mathbb{R}_{+})$ by $(\ref{basiccnsv})_{4}$, the each term on the right-hand side of (\ref{entropyfluid222}) tends to $0$ as $t\rightarrow \infty$, and therefore it holds
\begin{equation}\nonumber
\begin{split}
\lim_{t\rightarrow\infty}\mathcal{E}_{NS}^{\eta}(t)=0,
\end{split}
\end{equation}
which together with $(\ref{ENSsim})_{2}$ leads to  (\ref{rhobehaviorlp}).
\end{proof}
\end{lemma}

\begin{lemma}\label{lemma26}
Let $T>0$, and $(\rho,u,f)$ be any regular solution to the IVP $(\ref{m1})$-$(\ref{kappa})$ for $t\in(0,T]$. Then, under the assumptions of Theorem \ref{theorem11}, it holds
\begin{align}
&\rho(x,t)\geq \rho_{-}>0,\quad (x,t)\in \mathbb{T}\times[0,T],\label{lowerboundrho}
\end{align}
where $\rho_{-}>0$ is a constant independent of time $T>0$.
\end{lemma}
\begin{proof}
By (\ref{rhocnsv}), we have
\begin{equation}\label{shorttime1}
\begin{split}
\rho(x,t)\geq \frac{1}{C_{T_{0}}}>0,\quad (x,t)\in \mathbb{T}\times [0,T_{0}],
\end{split}
\end{equation}
where the time $T_{0}>0$ is given by Lemma \ref{lemma25}. 

We prove that the fluid density $\rho$ is uniformly bounded from below for $t\in(T_{0},T]$ in the case $\beta=0$, since the case $\beta>0$ can be dealt with in a similar way. It is easy to show
\begin{equation}\label{momentumfluid}
\begin{split}
&\frac{d}{dt}m_{1}(t)=\frac{1}{\|\rho_{0}\|_{L^1}}\frac{d}{dt}\int_{\mathbb{T}}\rho u(x,t) dx\\
&\quad\quad\quad~~=-\frac{1}{\|\rho_{0}\|_{L^1}}\int_{\mathbb{T}}[\rho n(u-w)](x,t)dx,
\end{split}
\end{equation}
where $m_{1}(t)$ is given by $(\ref{m1m2})_{1}$. By $(\ref{nw})$ and (\ref{momentumfluid}), the equation $(\ref{m1})_{2}$ for $\beta=0$ can be re-written as
\begin{equation}
\begin{split}
&(\rho(u-m_{1}))_{t}+(\rho u(u-m_{1}))_{x}+(\rho^{\gamma})_{x}\\
&\quad=2u_{xx}-\rho n(u-w)+\frac{\rho}{\|\rho_{0}\|_{L^1}}\int_{\mathbb{T}}[\rho n (u-w)](y,t)dy.\label{m21111}
\end{split}
\end{equation}
 Applying the operator $\mathcal{I}$ defined by (\ref{j}) to (\ref{m21111}), and re-writing the resulted equation along the particle path $\mathcal{X}^{x,t}(s)$ given by (\ref{overlinex}) for any $(x,t)\in\mathbb{T}\times[T_{0},T]$ and $s\in[T_{0},t]$, we obtain by (\ref{J})-(\ref{J1}) and (\ref{massrhox}) that
\begin{equation}\label{aaaaa}
\begin{split}
&\frac{d}{ds}\Big{(}\big{[}2\log{\rho}+\mathcal{I}(\rho(u-m_{1}))\big{]}(\mathcal{X}^{x,t}(s),s)+\int_{T_{0}}^{s} A(\mathcal{X}^{x,t}(\tau),\tau)d\tau\Big{)}\\
&\quad=-\rho^{\gamma}(\mathcal{X}^{x,t}(s),s),
\end{split}
\end{equation}
where $A(x,t)$ is given by
 \begin{equation}\label{a}
\begin{split}
 &A(x,t)=:-\int_{0}^{1}\big{[}\rho^{\gamma}+\rho u(u-m_{1})-\mu(\rho)u_{y}\big{]}(y,t)dy\\
 &\quad\quad\quad\quad~+\mathcal{I}\big{(}\rho n(u-w)\big{)}(x,t)-\frac{\mathcal{I}(\rho)(x,t)}{\|\rho_{0}\|_{L^1}}\int_{\mathbb{T}}[\rho n(u-w)](y,t)dy.
\end{split}
\end{equation}
Due to (\ref{basiccnsv}), (\ref{rhocnsv}), (\ref{J}) and (\ref{m1m2}), for any $(x,t)\in\mathbb{T}\times[T_{0},T]$, it holds
\begin{equation}\label{um1}
\begin{split}
&|\mathcal{I}(\rho(u-m_{1}))(x,t)|\\
&\quad\leq \sup_{t\in[0,T]}\|\big{(}\rho (u-m_{1})\big{)}(t)\|_{L^1}\\
&\quad\leq\sup_{t\in[0,T]}\big{(}\|\rho(t)\|_{L^1}^{\frac{1}{2}}\|(\sqrt{\rho}u)(t)\|_{L^2}+\frac{1}{\|\rho_{0}\|_{L^1}}\|\rho(t)\|_{L^1}^{\frac{3}{2}}\|(\sqrt{\rho}u)(t)\|_{L^2}\big{)}\\
&\quad\leq C_{10}=:2(2\|\rho_{0}\|_{L^1}E_{0})^{\frac{1}{2}}.
\end{split}
\end{equation}
By (\ref{basiccnsv}), (\ref{rhocnsv}), (\ref{J}),  (\ref{pressure}), (\ref{um1infty}) and (\ref{a}), for any $t\in[T_{0},T]$ and $s\in [T_{0},t]$, we have
\begin{align}
&\int_{s}^{t}A(\mathcal{X}^{x,t}(\tau),\tau)d\tau\nonumber\\
&\quad\leq -\int_{s}^{t}\int_{\mathbb{T}} \rho^{\gamma}(y,\tau)dyd\tau+\int_{s}^{t}\int_{\mathbb{T}}(\rho |u||u-m_{1}|+\mu(\rho)|u_{y}|)(y,\tau)dyd\tau\nonumber\\
&\quad\quad+2\int_{s}^{t}\int_{\mathbb{T}\times\mathbb{R}}(\rho|u-v|f)(y,v,\tau)dvdyd\tau\nonumber\\
&\quad\leq-P_{0}(t-s)+(t-s)^{\frac{1}{2}}\Big{(}\sup_{\tau\in [T_{0},T]}\big{(} \|\rho(\tau)\|_{L^1}^{\frac{1}{2}}\|(\sqrt{\rho} u)(\tau)\|_{L^2}\big{)}\|u-m_{1}\|_{L^2(T_{0},T;L^{\infty})}\nonumber\\
&\quad\quad+(\mu(\rho_{+}))^{\frac{1}{2}}\|\sqrt{\mu(\rho)}u_{x}\|_{L^2(T_{0},T;L^{2})}+2\rho_{+}^{\frac{1}{2}}\sup_{\tau\in [T_{0},T]}\|f(\tau)\|_{\mathcal{L}^1}^{\frac{1}{2}}\|\sqrt{\rho f}(u-v)\|_{L^2(T_{0},T;\mathcal{L}^2)}\Big{)}\nonumber\\
&\quad\leq -P_{0}(t-s)+(t-s)^{\frac{1}{2}}E_{0}^{\frac{1}{2}}\Big{(}(2\|\rho_{0}\|_{L^1}E_{0})^{\frac{1}{2}}+(\mu(\rho_{+}))^{\frac{1}{2}}+2(\rho_{+}\|f_{0}\|_{\mathcal{L}^1})^{\frac{1}{2}}\Big{)}\nonumber\\
&\quad\leq -\frac{P_{0}}{2}(t-s)+C_{11},\label{aa}
\end{align}
where the constant $C_{11}>0$ is given by
$$
C_{11}=:\frac{E_{0}}{2P_{0}}\Big{(}(2\|\rho_{0}\|_{L^1}E_{0})^{\frac{1}{2}}+(\mu(\rho_{+}))^{\frac{1}{2}}+2(\rho_{+}\|f_{0}\|_{\mathcal{L}^1})^{\frac{1}{2}}\Big{)}^2.
$$
Multiplying (\ref{aaaaa}) by
 $$
-\frac{1}{2}e^{-\frac{1}{2}[2\log{\rho}+\mathcal{I}(\rho(u-m_{1}))](\mathcal{X}^{x,t}(s),s)-\frac{1}{2}\int_{T_{0}}^{s}A(\mathcal{X}^{x,t}(\tau),\tau)d\tau},
$$
we obtain
\begin{equation}\begin{split}\nonumber
&\frac{d}{ds}\Big{(}\rho^{-1}(\mathcal{X}^{x,t}(s),s) e^{-\frac{1}{2}\mathcal{I}(\rho(u-m_{1}))(\mathcal{X}^{x,t}(s),s)-\frac{1}{2}\int_{T_{0}}^{s}A(\mathcal{X}^{x,t}(\tau),\tau)d\tau}\Big{)}\\
&=\frac{1}{2}\rho^{\gamma-1}(\mathcal{X}^{x,t}(s),s)e^{-\frac{1}{2}\mathcal{I}(\rho(u-m_{1}))(\mathcal{X}^{x,t}(s),s)-\frac{1}{2}\int_{T_{0}}^{s}A(\mathcal{X}^{x,t}(\tau),\tau)d\tau},
\end{split}\end{equation}
which together with (\ref{overlinex}) and (\ref{um1})-(\ref{aa}) implies for any $(x,t)\in\mathbb{T}\times [T_{0},T]$ that
\begin{equation}\begin{split}
&\rho^{-1}(x,t)\\
&\quad=\rho^{-1}(\mathcal{X}^{x,t}(T_{0}),T_{0})e^{\frac{1}{2}\mathcal{I}(\rho(u-m_{1}))(x,t)-\frac{1}{2}\mathcal{I}(\rho(u-m_{1}))(\mathcal{X}^{x,t}(T_{0}),T_{0})+\frac{1}{2}\int_{T_{0}}^{t}A(\mathcal{X}^{x,t}(\tau),\tau)d\tau} \\
&\quad\quad+\frac{1}{2}\int_{T_{0}}^{t}\rho^{\gamma-1}(\mathcal{X}^{x,t}(s),s)e^{\frac{1}{2}\mathcal{I}(\rho(u-m_{1}))(x,t)-\frac{1}{2}\mathcal{I}(\rho(u-m_{1}))(\mathcal{X}^{x,t}(s),s)+\frac{1}{2}\int_{s}^{t}A(\mathcal{X}^{x,t}(\tau),\tau)d\tau}ds\\
&\quad\leq \frac{1}{C_{T_{0}}}e^{C_{12}-\frac{P_{0}}{4}(t-T_{0})}+\frac{1}{2}\rho_{+}^{\gamma-1}e^{C_{12}}\int_{T_{0}}^{t}e^{-\frac{P_{0}}{4}(t-s)}ds,\quad C_{12}=:C_{10}+\frac{C_{11}}{2}. \label{longtime}
\end{split}\end{equation}
The combination of (\ref{shorttime1}) and (\ref{longtime}) leads to (\ref{lowerboundrho}). The proof of Lemma \ref{lemma26} is completed.
\end{proof}

\begin{lemma}\label{lemma27}
Let $T>0$, and $(\rho,u,f)$ be any regular solution to the IVP $(\ref{m1})$-$(\ref{kappa})$ for $t\in(0,T]$. Then, under the assumptions of Theorem \ref{theorem11}, it holds
\begin{equation}
\left\{
\begin{split}
&\sup_{t\in[0,T]}\|\rho_{x}(t)\|_{L^2}\leq C_{13},\\
&~{\rm{Supp}}_{v}f(x,\cdot,t)\subset \{v\in\mathbb{R}~\big{|}~|v|\leq R_{1}\},\quad (x,t)\in \mathbb{T}\times[0,T],\label{rhoh1uniform}
\end{split}
\right.
\end{equation}
where $C_{13}>0$ is a constant independent of time $T>0$, and the constant $R_{1}>0$ is denoted by
\begin{equation}\label{Rt}
\begin{split}
R_{1}=:R_{0}+\rho_{+}E_{0}^{\frac{1}{2}}\Big{(}\frac{1}{(2\rho_{-})^{\frac{1}{2}}}+\frac{2^{\frac{1}{2}}}{\rho_{-}\|\rho_{0}\|_{L^1}^{\frac{1}{2}}}\Big{)}.
\end{split}
\end{equation}
\end{lemma}
\begin{proof}
Multiplying (\ref{bd}) by $U$ given by (\ref{m}), integrating the resulted equation by parts over $\mathbb{T}$, and using (\ref{basiccnsv}), (\ref{rhocnsv}) and (\ref{lowerboundrho}), we have
\begin{equation}\label{the22}
\begin{split}
&\frac{1}{2}\frac{d}{dt}\|(\sqrt{\rho}U)(t)\|_{L^2}^2+\gamma \rho_{-}^{\gamma}\mu^{-1}(\rho_{+})\|(\sqrt{\rho}U)(t)\|_{L^2}^2\\
&~~\leq\big{(}\gamma\|\rho(t)\|_{L^{\infty}}^{\gamma}\|(\sqrt{\rho}u)(t)\|_{L^2}+\gamma\|\rho(t)\|_{L^{\infty}}^{\gamma}\|\rho(t)\|_{L^1}^{\frac{1}{2}}\|n(t)\|_{L^1}\\
&\quad~~+\|\rho(t)\|_{L^1}^{\frac{1}{2}}\|f(t)\|_{\mathcal{L}^1}^{\frac{1}{2}}\||v|^2f(t)\|_{\mathcal{L}^1}^{\frac{1}{2}}\big{)}\|(\sqrt{\rho}U)(t)\|_{L^2}\\
&~~\leq C_{14}\|(\sqrt{\rho}U)(t)\|_{L^2},
\end{split}
\end{equation}
where the constant $C_{14}>0$ is given by
$$
C_{14}=:\gamma\rho_{+}^{\gamma}(2E_{0})^{\frac{1}{2}}+\gamma\rho_{+}^{\gamma}\|\rho_{0}\|_{L^1}^{\frac{1}{2}}\|f_{0}\|_{\mathcal{L}^1}+(2\|\rho_{0}\|_{L^1}\|f_{0}\|_{\mathcal{L}^1}E_{0})^{\frac{1}{2}}.
$$
We divide (\ref{the22}) by $(\|(\sqrt{\rho}U)(t)\|_{L^2}^2+\varepsilon^2)^{\frac{1}{2}}$ for $\varepsilon>0$, employ the Gr$\rm{\ddot{o}}$nwall's inequality and then take the limit $\varepsilon\rightarrow 0$ to obtain
\begin{equation}\label{U2unifrom}
\begin{split}
&\sup_{t\in[0,T]}\|(\sqrt{\rho}U)(t)\|_{L^2}\leq e^{-\gamma \rho_{-}^{\gamma}\mu^{-1}(\rho_{+})t}\|(\sqrt{\rho}U)(0)\|_{L^2}+C_{14}\int_{0}^{T}e^{-\gamma \rho_{-}^{\gamma}\mu^{-1}(\rho_{+})s}ds.
\end{split}
\end{equation}
Due to (\ref{basiccnsv}), (\ref{rhocnsv}) and (\ref{m}), it holds
\begin{equation}\label{rhoxuniform}
\begin{split}
&\sup_{t\in[0,T]}\|\rho_{x}(t)\|_{L^2} \leq \rho_{+}^{\frac{3}{2}} \sup_{t\in[0,T]} \|\big{(}\sqrt{\rho}(U-u-\mathcal{I}(n))\big{)}(t)\|_{L^2}\\
&\quad\quad\quad\quad\quad\quad~~\leq \rho_{+}^{\frac{3}{2}} \sup_{t\in[0,T]}\big{(}\|(\sqrt{\rho}U)(t)\|_{L^2}+\|(\sqrt{\rho}u)(t)\|_{L^2}+\|\rho(t)\|_{L^1}^{\frac{1}{2}}\|n(t)\|_{L^1}\big{)}.
\end{split}
\end{equation}
By (\ref{basiccnsv}) and (\ref{U2unifrom})-(\ref{rhoxuniform}), we prove $(\ref{rhoh1uniform})_{1}$.

Moreover, one deduces by $(\ref{a1})_{2}$, (\ref{basiccnsv}), (\ref{rhocnsv}), (\ref{Sigma1}), (\ref{uinfty1}) and (\ref{lowerboundrho}) that
\begin{equation}
\begin{split}
&\sup_{v\in\Sigma(x,t)}|v|\leq e^{-\rho_{-}t}R_{0}+\rho_{+}\int_{0}^{t}e^{-\rho_{-}s}|u(x,s)|ds\\
&\quad\quad\quad\quad~\leq R_{0}+\rho_{+}\|u_{x}\|_{L^2(0,t;L^2)}\Big{(}\int_{0}^{t}e^{-2\rho_{-}s}ds\Big{)}^{\frac{1}{2}}\\
&\quad\quad\quad\quad~\quad+\frac{\rho_{+}}{\|\rho_{0}\|_{L^1}^{\frac{1}{2}}}\sup_{s\in[0,t]}\|(\sqrt{\rho}u)(s)\|_{L^2}\int_{0}^{t}e^{-\rho_{-}s}ds\\
&\quad\quad\quad\quad~\leq R_{1}=:R_{0}+\rho_{+}E_{0}^{\frac{1}{2}}\Big{(}\frac{1}{(2\rho_{-})^{\frac{1}{2}}}+\frac{2^{\frac{1}{2}}}{\rho_{-}\|\rho_{0}\|_{L^1}^{\frac{1}{2}}}\Big{)},\quad (x,t)\in\mathbb{T}\times[0,T],\label{compactv}
\end{split}
\end{equation}
where the initial energy $E_{0}$ is given by (\ref{E0}). By (\ref{Sigma}) and (\ref{compactv}), $(\ref{rhoh1uniform})_{2}$ holds. The proof of Lemma \ref{lemma27} is completed.
\end{proof}

\bigskip

\underline{\it\textbf{Proof of global existence in spatial periodic domain:}} \emph{Step 1: Construction of global strong solution.} Let the initial data $(\rho_{0},u_{0},f_{0})$ satisfy $(\ref{a1})$ and (\ref{a11}). By Lemma \ref{lemma41} in Appendix, we can obtain the local existence of a unique strong solution to the IVP $(\ref{m1})$-$(\ref{kappa})$.

In the terms of Lemmas \ref{lemma21}-\ref{lemma24} about the uniformly a-priori estimates, we can extend the local solution globally in time and prove that this solution satisfies $(\ref{r1})$-(\ref{r11}) and (\ref{r12}).

\emph{Step 2: Global existence of weak solution.} Assume that the initial data $(\rho_{0},u_{0},f_{0})$ satisfies $(\ref{a1})$. One can regularize the initial data as follows:
\begin{equation}\label{appd}
\begin{split}
(\rho_{0}^{\varepsilon}(x),u_{0}^{\varepsilon}(x),f_{0}^{\varepsilon}(x,v))=( J_{1}^{\varepsilon}\ast\rho_{0}(x),J_{1}^{\varepsilon}\ast u_{0}(x) ,J_{1}^{\varepsilon}\ast J_{2}^{\varepsilon}\ast f_{0}(x,v)),\quad \varepsilon\in(0,1),
\end{split}
\end{equation}
where $J_{1}^{\varepsilon}(x)$ and $J_{2}^{\varepsilon}(v)$ are the Friedrich's mollifiers with respect to variables $x$ and $v$ respectively. Then it is easy to verify
\begin{equation}\label{initialt1}
\left\{
\begin{split}
&\inf_{x\in\mathbb{T}}\rho_{0}^{\varepsilon}(x)\geq \inf_{x\in\mathbb{T}}\rho(x)>0,\quad \|\rho_{0}^{\varepsilon}\|_{W^{1,\infty}}\leq \|\rho_{0}\|_{W^{1,\infty}}, \quad \|u_{0}^{\varepsilon}\|_{H^1}\leq \|u_{0}\|_{H^1},\\
&f_{0}^{\varepsilon}\geq 0,\quad \|f_{0}^{\varepsilon}\|_{\mathcal{L}^{\infty}}\leq \|f_{0}\|_{\mathcal{L}^{\infty}},\quad\text{{\rm{Supp}}}_{v}f^{\varepsilon}_{0}(x,\cdot)\subset \{v\in\mathbb{R}~\big{|}~|v|\leq R_{0}+1\},~  x\in\mathbb{T}.
\end{split}
\right.
\end{equation}
In addition, as $\varepsilon\rightarrow 0$, it holds up to a subsequence (still denoted by $(\rho_{0}^{\varepsilon},u_{0}^{\varepsilon},f_{0}^{\varepsilon})$) that 
\begin{equation}\label{initialt}
\left\{
\begin{split}
&(\rho_{0}^{\varepsilon}, u_{0}^{\varepsilon}, f_{0}^{\varepsilon})\rightarrow (\rho_{0},u_{0},f_{0})\quad \text{in}\quad W^{1,p}(\mathbb{T})\times H^1(\mathbb{T})\times L^{p}(\mathbb{T}\times\mathbb{R}), \quad p\in[1,\infty),\\
&(\rho_{0}^{\varepsilon},f_{0}^{\varepsilon}) \overset{\ast}{\rightharpoonup} (\rho_{0},f_{0})\quad \text{in}\quad W^{1,\infty}(\mathbb{T})\times L^{\infty}(\mathbb{T}\times\mathbb{R}).
\end{split}
\right.
\end{equation}

Let $T>0$ be any given time. It follows from Step 1 that for the initial data $(\rho^{\varepsilon}_{0},u^{\varepsilon}_{0},f^{\varepsilon}_{0})$, the IVP (\ref{m1})-(\ref{kappa}) admits a unique strong solution $(\rho^{\varepsilon},u^{\varepsilon},f^{\varepsilon})$ on $\mathbb{T} \times\mathbb{R}\times[0,T]$.

 Then, by the a-priori estimates established in Lemmas \ref{lemma21}-\ref{lemma23} uniformly with respect to $\varepsilon\in(0,1)$ and the Aubin-Lions's lemma, there is a limit $(\rho,u,f)$ so that as $\varepsilon\rightarrow 0$, it holds up to a subsequence (still denoted by $(\rho^{\varepsilon},u^{\varepsilon}, f^{\varepsilon})$) that
\begin{equation}\label{weakc}
\left\{
\begin{split}
&(\rho^{\varepsilon}, f^{\varepsilon})\overset{\ast}{\rightharpoonup}(\rho,f) \quad\text{in}\quad L^{\infty}(0,T;W^{1,\infty}(\mathbb{T}))\times L^{\infty}(0,T; L^{\infty}(\mathbb{T}\times\mathbb{R})),\\
&u^{\varepsilon}\rightharpoonup u\quad\text{in}\quad L^2(0,T;H^2(\mathbb{T}))\cap H^1(0,T;L^2(\mathbb{T})),\\
&(\rho^{\varepsilon},u^{\varepsilon})\rightarrow (\rho,u)\quad\text{in}~C([0,T];C^{0}(\mathbb{T}))\times C([0,T];C^{0}(\mathbb{T})).
\end{split}
\right .
\end{equation}
By $(\ref{rhouh1cnsv})_{2}$, we have ${\rm{Supp}}_{v}f^{\varepsilon}(x,\cdot,t)\subset \{v\in\mathbb{R}~\big{|}~|v|\leq R_{T}+1\}$  for $(x,t)\in \mathbb{T}\times[0,T]$, and similarly to (\ref{fformula})-(\ref{compactv1}), one can show ${\rm{Supp}}_{v}f(x,\cdot,t)\subset \{v\in\mathbb{R}~\big{|}~|v|\leq R_{T}\}$ for $(x,t)\in \mathbb{T}\times[0,T]$. Therefore, it follows by $(\ref{weakc})_{1}$ that
\begin{equation}\label{weakcc}
\begin{split}
&\lim_{\varepsilon\rightarrow 0}\int_{0}^{T}\int_{\mathbb{T}}\varphi(x,t)\Big{(}\int_{\mathbb{R}} \phi(v) f^{\varepsilon}(x,v,t)dv-\int_{\mathbb{R}} \phi(v)f(x,v,t)dv\Big{)}dxdt\\
&\quad =\lim_{\varepsilon\rightarrow 0}\int_{0}^{T}\int_{\mathbb{T}\times\mathbb{R}} \varphi(x,t)\phi(v)\chi (v)(f^{\varepsilon}-f)(x,v,t)dvdxdt\\
&\quad=0,\quad\quad  \forall \varphi\in C^{\infty}_{0}(\mathbb{T}\times(0,T)),\quad\phi(v)=1, v,
\end{split}
\end{equation}
where $\chi(v)\in C_{0}^{\infty}(\mathbb{R})$ satisfies $\chi(v)=1$ for $|v|\leq R_{T}+1$ and $\chi(v)=0$ for $|v|>R_{T}+2$.

 Since both $\int_{\mathbb{R}} f^{\varepsilon}dv$ and $\int_{\mathbb{R}} vf^{\varepsilon}dv$ are uniformly bounded in $L^{\infty}(0,T;L^{\infty}(\mathbb{T}))$, as $\varepsilon\rightarrow 0$, it holds by (\ref{weakcc}) that
 \begin{equation}\label{ne}
 \left\{
 \begin{split}
& \int_{\mathbb{R}} f^{\varepsilon}dv\overset{\ast}{\rightharpoonup} n=:\int_{\mathbb{R}} fdv \quad\text{in}~L^{\infty}(0,T;L^{\infty}(\mathbb{T})),\\
 &\int_{\mathbb{R}} vf^{\varepsilon}dv\overset{\ast}{\rightharpoonup} nw=:\int_{\mathbb{R}} vfdv \quad\text{in}~L^{\infty}(0,T;L^{\infty}(\mathbb{T})).
 \end{split}
 \right.
 \end{equation}
In addition, it can be verified by (\ref{weakc}) and (\ref{ne}) that $(\rho,u,f)$ solves the equations $(\ref{m1})$ in the sense of distributions. Due to the Sobolev embedding $
L^2(0,T;H^2(\mathbb{T}))\cap H^1(0,T;L^2(\mathbb{T}))\hookrightarrow C([0,T];H^1(\mathbb{T}))$,
it holds
\begin{equation}\label{ch1}
\begin{split}
u\in C([0,T];H^1(\mathbb{T})).
\end{split}
\end{equation}
Employing the theory of renormalized solutions in \cite{diperna1}, we get
\begin{equation}\label{ch2}
\begin{split}
\rho\in C([0,T];H^1(\mathbb{T})),\quad f\in C([0,T];L^1(\mathbb{T}\times\mathbb{R})).
\end{split}
\end{equation}
Finally, it follows from Lemmas \ref{lemma21}-\ref{lemma23} and (\ref{initialt})-(\ref{ch2}) that the weak solution $(\rho,u,f)$ satisfies (\ref{basiccnsv}) and (\ref{rhocnsv}). Then repeating the same arguments as in Lemmas \ref{lemma25}-\ref{lemma27}, we can prove that the expected properties (\ref{r11}) hold for $(\rho,u,f)$.

\subsection{Spatial real line}

In this subsection, we establish the global existence for compressible Navier-Stokes-Vlasov equations (\ref{m1}) in spatial real line. $\Pi_{\gamma}(\rho|\widetilde{\rho})$ is the relative entropy associated to $\frac{\rho^{\gamma}}{\gamma-1}$ defined by
\begin{equation}\label{Pi}
\begin{split}
&\Pi_{\gamma}(\rho|\widetilde{\rho})=:\frac{\rho^{\gamma}}{\gamma-1}-\frac{\widetilde{\rho}^{\gamma}}{\gamma-1}-\frac{\gamma\widetilde{\rho}^{\gamma-1}}{\gamma-1}(\rho-\widetilde{\rho}).
\end{split}
\end{equation}
If it holds $0\leq \rho\leq c_{+}$ for some constant $c_{+}>0$, then as in  {\rm{\cite{lhl1}}}, one can show
\begin{equation}\label{rho2}
\begin{split}
\frac{1}{C_{\widetilde{\rho},c_{+}}}|\rho-\widetilde{\rho}|^2\leq \Pi_{\gamma}(\rho|\widetilde{\rho})\leq C_{\widetilde{\rho},c_{+}}|\rho-\widetilde{\rho}|^2,
\end{split}
\end{equation}
where the constant $C_{\widetilde{\rho},c_{+}}>1$ depends only on $\gamma, \widetilde{\rho}$ and $c_{+}$.

First, we have

\begin{lemma}\label{lemma28}
Let $T>0$, and $(\rho,u,f)$ be any regular solution to the IVP $(\ref{m1})$-$(\ref{kappa})$ for $t\in(0,T]$. Then, under the assumptions of Theorem \ref{theorem12}, it holds
\begin{equation}\label{basicR}
\left\{
\begin{split}
&\rho(x,t)\geq 0,\quad f(x,v,t)\geq 0, \quad (x,v,t)\in\mathbb{R}\times\mathbb{R}\times [0,T],\\
&\int_{\mathbb{R}\times\mathbb{R}} f(x,v,t)dvdx=\int_{\mathbb{R}} n(x,t) dx=\int_{\mathbb{R}\times\mathbb{R}}f_{0}(x,v)dvdx,\quad t\in[0,T],\\
&\sup_{t\in[0,T]}\Big{(}\int_{\mathbb{R}}\big{(}\frac{1}{2}\rho |u|^2+\Pi_{\gamma}(\rho|\widetilde{\rho})\big{)}(x,t)dx+\int_{\mathbb{R}\times\mathbb{R}}\frac{1}{2}|v|^2f(x,v,t)dvdx\Big{)}\\
&\quad+\int_{0}^{T}\int_{\mathbb{R}}(\mu(\rho)|u_{x}|^2)(x,t)dxdt+\int_{0}^{T}\int_{\mathbb{R}\times\mathbb{R}} (\rho|u-v|^2f)(x,v,t)dvdxdt\leq \widetilde{E}_{0},\\
&\|u\|_{L^1(0,T;L^{p})}\leq C_{15}(T+T^{\frac{1}{2}}),\quad p\in[2,\infty],
\end{split}
\right .
\end{equation}
where $\Pi_{\gamma}(\rho|\widetilde{\rho})$ is defined by $(\ref{Pi})$, the initial energy $\widetilde{E}_{0}$ is given by
$$
\widetilde{E}_{0}=:\int_{\mathbb{R}}\big{(}\frac{1}{2}\rho_{0} |u_{0}|^2+\Pi_{\gamma}(\rho_{0}|\widetilde{\rho})\big{)}(x)dx+\int_{\mathbb{R}\times\mathbb{R}}\frac{1}{2} |v|^2f_{0}(x,v)dvdx,
$$
and $C_{15}>0$ is a constant independent of $p$ and $T$.
\end{lemma}
\begin{proof}
$(\ref{basicR})_{1}$-$(\ref{basicR})_{2}$ can be derived directly by $(\ref{m1})_{1}$ and $(\ref{m1})_{3}$. Similarly to (\ref{energyfluid})-(\ref{energyparticle}), we can obtain $(\ref{basicR})_{3}$. Then it is easy to verify
 \begin{equation}\label{411}
\begin{split}
&\widetilde{\rho}\int_{\mathbb{R}} |u(x,t)|^2dx\\
&\quad= \int_{\{\rho\leq \widetilde{\rho}\}}[(\widetilde{\rho}-\rho)|u|^2](x,t)dx+\int_{\{\rho\geq \widetilde{\rho}\}}[(\widetilde{\rho}-\rho)|u|^2](x,t)dx+\int_{\mathbb{R}} (\rho |u|^2)(x,t)dx.
\end{split}
\end{equation}
One deduces by (\ref{rho2}), the Gagliardo-Nirenberg's and Young's inequalities that
\begin{equation}\label{412}
\begin{split}
&\int_{\{\rho\leq \widetilde{\rho}\}}\big{[}(\widetilde{\rho}-\rho)|u|^2\big{]}(x,t)dx\\
&\quad\leq \Big{(}\int_{\{\rho\leq \widetilde{\rho}\}}|(\widetilde{\rho}-\rho)(x,t)|^2dx\Big{)}^{\frac{1}{2}}\|u(t)\|_{L^{4}}^2\\
&\quad\leq C\Big{(}\int_{\mathbb{R}} \Pi_{\gamma}(\rho|\widetilde{\rho})(x,t)dx\Big{)}^{\frac{1}{2}}\|u(t)\|_{L^2}^{\frac{3}{2}}\|u_{x}(t)\|_{L^2}^{\frac{1}{2}}\\
&\quad\leq \frac{\widetilde{\rho}}{2}\int_{\mathbb{R}}|u(x,t)|^2dx+C\Big{(}\int_{\mathbb{R}} \Pi_{\gamma}(\rho|\widetilde{\rho})(x,t)dx\Big{)}^{2}\int_{\mathbb{R}}|u_{x}(x,t)|^2dx.
\end{split}
\end{equation}
For the second term on the right-hand side of (\ref{411}), it holds
\begin{equation}\label{413}
\begin{split}
&\int_{\{\rho\geq \widetilde{\rho}\}}[(\widetilde{\rho}-\rho)|u|^2](x,t)dx\leq \int_{\{\rho\geq \widetilde{\rho}\}}(\widetilde{\rho}|u|^2)(x,t)dx\leq \int_{\mathbb{R}} (\rho|u|^2)(x,t)dx.
\end{split}
\end{equation}
Substituting (\ref{412})-(\ref{413}) into (\ref{411}), we obtain
$$
\|u(t)\|_{L^2}\leq C\big{(}\|\Pi_{\gamma}(\rho|\widetilde{\rho})(t)\|_{L^1}\|u_{x}(t)\|_{L^2}+\|(\sqrt{\rho}u)(t)\|_{L^2}\big{)},
$$
which together with $(\ref{basicR})_{3}$ and the fact
\begin{equation}\nonumber
\begin{split}
&\|u(t)\|_{L^{p}}\leq C\big{(}\|u(t)\|_{L^2}+\|u_{x}(t)\|_{L^2}\big{)},\quad p\in[2,\infty],
\end{split}
\end{equation}
 gives rise to $(\ref{basicR})_{4}$. The proof of Lemma \ref{lemma28} is completed.
 \end{proof}

\begin{lemma}\label{lemma29}
Let $T>0$, and $(\rho,u,f)$ be any regular solution to the IVP $(\ref{m1})$-$(\ref{kappa})$ for $t\in(0,T]$. Then, under the assumptions of Theorem \ref{theorem12}, it holds
\begin{equation}\label{upper1}
\left\{
\begin{split}
&\quad 0<\rho_{T}^{-1}\leq \rho(x,t)\leq \rho_{T},\quad (x,t)\in\mathbb{R}\times [0,T],\\
&\sup_{t\in[0,T]}\|(\rho-\overline{\rho})(t)\|_{L^2}\leq C_{T},
\end{split}
\right.
\end{equation}
where $\rho_{T}>0$ and $C_{T}>0$ are two constants dependent of time $T>0$.
\end{lemma}
\begin{proof}
For any $x_{0}\in \mathbb{R}$, we claim that it holds
\begin{equation}
\begin{split}
\rho(x,t)\leq \rho_{T},\quad (x,t)\in  (x_{0},x_{0}+1)\times[0,T],\label{claimrho}
\end{split}
\end{equation}
where $\rho_{T}>0$ is a constant independent of the choice of $x_{0}$, so the upper bound in $(\ref{upper1})_{1}$ for any $(x,t)\in \mathbb{R}\times [0,T]$ can be obtained immediately. Indeed, integrating the equation $(\ref{m1})_{3}$ over $(-\infty, x]\times\mathbb{R}$ for any $x\in \mathbb{R}$, we have
\begin{equation}\nonumber
\begin{split}
(\int_{-\infty}^{x}n(y,t)dy)_{t}+nw=0,
\end{split}
\end{equation}
which together with $(\ref{m1})_{1}$ leads to
\begin{equation}\label{m111}
\begin{split}
&(\rho\int_{-\infty}^{x}n(y,t)dy)_{t}+(\rho u\int_{-\infty}^{x}n(y,t)dy)_{x}\\
&\quad=\rho(\int_{-\infty}^{x}n(y,t)dy)_{t}+\rho u (\int_{-\infty}^{x}n(y,t)dy)_{x}\\
&\quad=-\rho nw+\rho u n.
\end{split}
\end{equation}
Substituting (\ref{m111}) into momentum equation $(\ref{m1})_{2}$,  we get
\begin{equation}
\begin{split}
&\big{(}\rho( u+\int_{-\infty}^{x}n(y,t)dy)\big{)}_{t}+\big{(}\rho u( u+\int_{-\infty}^{x}n(y,t)dy)\big{)}_{x}+(\rho^{\gamma})_{x}=(\mu(\rho)u_{x})_{x}.\label{wholem}
\end{split}
\end{equation}
 Similarly to (\ref{j})-(\ref{newm4}), introduce the operator $\mathcal{I}_{x_{0}}: L^{1}(x_{0},x_{0}+1)\rightarrow W^{1,1}(x_{0},x_{0}+1)$ as
\begin{equation}\nonumber
\begin{split}
\mathcal{I}_{x_{0}}: \mathcal{I}_{x_{0}}(g)(x)=\int^{x}_{x_{0}}g(y)dy-\int_{x_{0}}^{x_{0}+1}\int_{x_{0}}^{y}g(z)dzdy,\quad\forall g\in L^1(x_{0},x_{0}+1),
\end{split}
\end{equation}
which satisfies for $g\in L^1(x_{0},x_{0}+1)$ that
\begin{align}
\sup_{x\in (x_{0},x_{0}+1)}\mathcal{I}_{x_{0}}(g)(x)\leq \|g\|_{L^1(x_{0},x_{0}+1)},\quad (\mathcal{I}_{x_{0}}(g)(x))_{x}=g(x),\label{Jx0}
\end{align}
and for $g\in W^{1,1}(x_{0},x_{0}+1)$ that
\begin{align}
\mathcal{I}_{x_{0}}(g_{x})(x)=g(x)-\int_{x_{0}}^{x_{0}+1}g(y)dy.\label{Jx01}
\end{align}
Applying $\mathcal{I}_{x_{0}}$ to the equation (\ref{wholem}), one can show by (\ref{Jx0})-(\ref{Jx01}) that
\begin{equation}\label{pppp}
\begin{split}
&\frac{d}{ds}\theta(\rho)(\mathcal{X}^{x,t}(s),s)+\rho^{\gamma}(\mathcal{X}^{x,t}(s),s)\\
&=-\frac{d}{ds}\mathcal{I}_{x_{0}}\big{(}\rho u+\rho\int_{-\infty}^{x}n(y,t)dy\big{)}(\mathcal{X}^{x,t}(s),s)+H(s)-\int_{x_{0}}^{x_{0}+1}(\mu(\rho)u_{y})(y,s)dx,
\end{split}
\end{equation}
where $\theta(\rho)$ and $\mathcal{X}^{x,t}(s)$ for any $(x,t)\in (x_{0},x_{0}+1)\times[0,T]$ are defined by (\ref{214}) and (\ref{overlinex}) respectively, and $H(s)$ is given by
$$
H(s)=:\int_{x_{0}}^{x_{0}+1}\Big{(}(\rho |u|^2)(y,s)+\rho(y,s)\int_{-\infty}^{y}n(z,s)dz+\rho^{\gamma}(y,s)\Big{)}dy.
$$
For $(x,t)\in (x_{0},x_{0}+1)\times[0,T]$ and $s\in[0,t]$, by virtue of (\ref{overlinex}) and $(\ref{basicR})_{4}$, we have
\begin{align}
&|\mathcal{X}^{x,t}(s)-x|\leq \int_{s}^{t}\|u(\tau)\|_{L^{\infty}}d\tau\leq C_{15}(T+T^{\frac{1}{2}}),\nonumber
\end{align}
which leads to
 \begin{align}
 &|\mathcal{X}^{x,t}(s)-x_{0}|\leq |\mathcal{X}^{x,t}(s)-x|+|x-x_{0}|\leq C_{15}(T+T^{\frac{1}{2}})+1.\label{Xx}
 \end{align}
In addition, it is easy to verify
  \begin{align}
\rho^{\gamma}+\rho\leq C_{16}\big{(}\Pi_{\gamma}(\rho|\widetilde{\rho})+1\big{)},\label{factx}
 \end{align}
 where $C_{16}>0$ is a constant depending only on $\gamma$ and $\widetilde{\rho}$. With the help of (\ref{basicR}), (\ref{Jx0}) and (\ref{Xx})-(\ref{factx}), we derive for any $s\in [0,t]$ that
  \begin{align}
 &\big{|}\mathcal{I}_{x_{0}}\big{(}\rho u+\rho \int_{-\infty}^{x}n(y,t)dy\big{)}(\mathcal{X}^{x,t}(s),s)\big{|}\nonumber\\
 &\quad\leq \Big{(}\int_{x_{0}}^{\mathcal{X}^{x,t}(s)}\rho(y,s)dy\Big{)}^{\frac{1}{2}}\Big{(}\int_{\mathbb{R}}(\rho |u|^2)(y,s)dy\Big{)}^{\frac{1}{2}}+\int_{x_{0}}^{\mathcal{X}^{x,t}(s)}\rho(y,s)dy\int_{\mathbb{R}} n(y,s)dy\nonumber\\
 &\quad\leq (2C_{16}E_{0})^{\frac{1}{2}}\Big{(}\int_{\mathbb{R}}\Pi(\rho|\widetilde{\rho})(y,s)dy+|\mathcal{X}^{x,t}(s)-X_{0}|\Big{)}^{\frac{1}{2}}\nonumber\\
 &\quad\quad+C_{16}\|f_{0}\|_{\mathcal{L}^1}\Big{(} \int_{\mathbb{R}}\Pi(\rho|\widetilde{\rho})(y,s)dy+|\mathcal{X}^{x,t}(s)-X_{0}| \Big{)}\nonumber\\
 &\quad\leq C_{17}(1+T),\label{Jx011}
\end{align}
and
  \begin{equation}\label{H}
 \begin{split}
 &|H(s)|\leq \int_{\mathbb{R}} (\rho |u|^2)(y,s)dy+C_{16}\big{(}\|n(s)\|_{L^1}+1\big{)}\Big{(}\int_{\mathbb{R}} \Pi_{\gamma}(\rho|\widetilde{\rho})(y,s)ds+1\Big{)}\leq C_{18},
\end{split}
\end{equation}
where $C_{17}>0$ and $C_{18}>0$ are two constants independent of time $T>0$. By (\ref{mutheta}), one has
\begin{align}
 &\Big{|}\int_{x_{0}}^{x_{0}+1}(\mu(\rho)u_{y})(y,t)dy\Big{|}\nonumber\\
&\quad\leq 1+ \frac{1}{4}\Big{(}{\footnotesize{\int_{\{y\in (x_{0},x_{0}+1)|\rho\leq 1\}}+\int_{\{y\in (x_{0},x_{0}+1)|\rho\geq 1\}}}}\Big{)}\mu(\rho)(y,s)dy\int_{\mathbb{R}}(\mu(\rho)|u_{y}|^2)(y,t)dy\nonumber\\
&\quad\leq 1+\int_{\mathbb{R}}(\mu(\rho)|u_{y}|^2)(y,t)dy+\frac{\beta}{4}{\footnotesize{\sup_{x\in\{y\in (x_{0},x_{0}+1)|\rho\geq 1\}}\theta(\rho)(x,t)\int_{\mathbb{R}}(\mu(\rho)|u_{y}|^2)(y,t)dy}}.\label{ux1}
\end{align}

Then, integrating the equality (\ref{pppp}) over $[0,t]$, we obtain by (\ref{basicR}) and (\ref{Jx011})-(\ref{ux1}) for any $(x,t)\in(x_{0},x_{0}+1)\times[0,T]$ that
\begin{equation}\nonumber
\begin{split}
&\theta(\rho)(x,t)\leq \theta(\rho_{0})(\mathcal{X}^{x,t}(0))+\big{|} \mathcal{I}_{x_{0}}\big{(}\rho u+\rho\int_{-\infty}^{x}n(y,t)dy\big{)}(\mathcal{X}^{x,t}(s),s)|^{s=t}_{s=0}\big{|}\\
&\quad\quad\quad\quad\quad+\int_{0}^{t}|H(s)|ds+\Big{|}\int_{0}^{t}\int_{x_{0}}^{x_{0}+1}(\mu(\rho)u_{y})(y,s)dyds\Big{|}\\
&\quad\quad\quad\quad\leq \sup_{x\in \mathbb{R}}\theta(\rho_{0})(x)+2C_{17}(1+T)+C_{18}T+T+\widetilde{E}_{0}\\
&\quad\quad\quad\quad\quad+\frac{\beta}{4}\int_{0}^{T}\sup_{x\in\{y\in (x_{0},x_{0}+1)|\rho\geq 1\}}\theta(\rho)(x,s)\int_{\mathbb{R}}(\mu(\rho)|u_{y}|^2)(y,s)dyds,
\end{split}
\end{equation}
which together with $(\ref{basicR})_{3}$ and the Gr${\rm{\ddot{o}}}$nwall's inequality yields (\ref{claimrho}). By (\ref{rho2}), $(\ref{basicR})_{3}$ and (\ref{claimrho}), we have
$$
\sup_{t\in[0,T]}\int_{\mathbb{R}} |\rho-\widetilde{\rho}|^2(x,t)dx\leq C_{T}\sup_{t\in[0,T]}\int_{\mathbb{R}} \Pi_{\gamma}(\rho|\widetilde{\rho})(x,t)dx\leq C_{T}.
$$

Similarly, integrating (\ref{pppp}) over $[0,t]$, one deduce by (\ref{basicR}), (\ref{claimrho}) and (\ref{Jx011})-(\ref{ux1}) for any $(x,t)\in (x_{0},x_{0}+1)\times[0,T]$ that
\begin{equation}
\begin{split}
&\theta(\rho)(x,t)\geq  \inf_{x\in\mathbb{R}}\theta(\rho_{0})(x)-T\rho_{T}^{\gamma}-\big{|} \big{(}\mathcal{I}_{x_{0}}\big{(}\rho u+\rho\int_{-\infty}^{x}n(y,t)dy\big{)}\big{)}(\mathcal{X}^{x,t}(s),s)|^{s=t}_{s=0} \big{|}\\
&\quad\quad\quad\quad\quad-\int_{0}^{T}|H(s)|ds-(\mu(\rho_{T})T)^{\frac{1}{2}}\|\sqrt{\mu(\rho)}u_{x}\|_{L^2(0,T;L^2)} \\
&\quad\quad\quad\quad\geq \inf_{x\in \mathbb{R}}\theta(\rho_{0})(x)-T\rho_{T}^{\gamma}-2C_{17}(1+T)-C_{18}T-(\mu(\rho_{T})T\widetilde{E}_{0})^{\frac{1}{2}}.\label{lowppp}
\end{split}
\end{equation}
  The combination of (\ref{logrho}) and (\ref{lowppp}) shows the lower bound of $\rho$. The proof of Lemma \ref{lemma29} is completed.
\end{proof}

\begin{lemma}\label{lemma210}
Let $T>0$, and $(\rho,u,f)$ be any regular solution to the IVP $(\ref{m1})$-$(\ref{kappa})$ for $t\in(0,T]$. Then, under the assumptions of Theorem \ref{theorem12}, it holds
\begin{equation}\label{H1R}
\left\{
\begin{split}
&\sup_{t\in[0,T]}\|f(t)\|_{\mathcal{L}^{\infty}}\leq e^{\rho_{T}T}\|f_{0}\|_{\mathcal{L}^{\infty}},\\
&~{\rm{Supp}}_{v}f(x,\cdot,t)\subset \{v\in\mathbb{R}~\big{|}~|v|\leq \widetilde{R}_{T}\},\\
&\sup_{t\in[0,T]}\big{(}\|n(t)\|_{L^{\infty}\cap L^2}+\|\rho_{x}(t)\|_{L^2\cap L^{\infty}}+\|u_{x}(t)\|_{L^2}\big{)}+\|(u_{xx},u_{t})\|_{L^2(0,T;L^2)}\leq C_{T},
\end{split}
\right .
\end{equation}
where the constant $\rho_{T}>0$ is given by $(\ref{upper1})_{1}$, and $\widetilde{R}_{T}>0$ and $C_{T}>0$ are two constants.
\end{lemma}
\begin{proof}
In accordance to (\ref{fformula}) and  $(\ref{upper1})_{1}$, $(\ref{H1R})_{1}$ holds. Let $\Sigma(x,t)$ be given by $(\ref{Sigmadef})$ for $(x,t)\in\mathbb{R}\times[0,T]$. Similarly to the proof of Lemma \ref{lemma23}, one deduces by $(\ref{a2})_{2}$, (\ref{fformula}), $(\ref{basicR})_{4}$ and $ (\ref{upper1})_{1}$ that
\begin{equation}\nonumber
\begin{split}
&\sup_{v\in\Sigma(x,t)}|v|\leq \widetilde{R}_{0}+\rho_{T}\int_{0}^{t}\|u(s)\|_{L^{\infty}}ds\\
&\quad\quad\quad\quad~\leq \widetilde{R}_{T}=:\widetilde{R}_{0}+C_{15}\rho_{T}(T+T^{\frac{1}{2}}),
\end{split}
\end{equation}
which implies $(\ref{H1R})_{2}$. By (\ref{nw}) and $(\ref{H1R})_{1}$-$(\ref{H1R})_{2}$, we have
\begin{equation}\label{H1R21}
\left\{
\begin{split}
&\sup_{t\in[0,T]}\|n(t)\|_{L^{\infty}}\leq 2\widetilde{R}_{T} \|f(t)\|_{\mathcal{L}^{\infty}}\leq 2\widetilde{R}_{T}e^{\rho_{T}T}\|f_{0}\|_{\mathcal{L}^{\infty}},\\
&\sup_{t\in[0,T]}\|nw(t)\|_{L^{\infty}}\leq 2\widetilde{R}^2_{T} \|f(t)\|_{\mathcal{L}^{\infty}}\leq 2\widetilde{R}_{T}^2e^{\rho_{T}T}\|f_{0}\|_{\mathcal{L}^{\infty}}.
\end{split}
\right.
\end{equation}
In addition, it follows from (\ref{basicR}), (\ref{H1R21}) and the Young's inequality for any $p\in (1,\infty)$ that
\begin{equation}\label{H1R2}
\begin{split}
&\sup_{t\in[0,T]}\|n(t)\|_{L^{p}}\leq  \sup_{t\in[0,T]}\big{(}\|n(t)\|_{L^{1}}+\|n(t)\|_{L^{\infty}}\big{)}\\
&\quad\quad\quad\quad\quad~~~~\leq\|f_{0}\|_{\mathcal{L}^1}+2\widetilde{R}_{T}e^{\rho_{T}T}\|f_{0}\|_{\mathcal{L}^{\infty}},
\end{split}
\end{equation}
and similarly,
\begin{equation}\label{H1R211}
\begin{split}
&\sup_{t\in[0,T]}\|nw(t)\|_{L^{p}}\leq \sup_{t\in[0,T]}\big{(}\|f(t)\|_{\mathcal{L}^{1}}^{\frac{1}{2}}\||v|^2f(t)\|_{\mathcal{L}^1}+\|nw(t)\|_{L^{\infty}}\big{)}\\
&\quad\quad\quad\quad\quad\quad\quad\leq \|f_{0}\|_{\mathcal{L}^1}^{\frac{1}{2}}(2\widetilde{E}_{0})^{\frac{1}{2}}+2\widetilde{R}_{T}e^{\rho_{T}T}\|f_{0}\|_{\mathcal{L}^{\infty}}.
\end{split}
\end{equation}

Then, we turn to show $(\ref{H1R})_{3}$. Define the following effective velocity of Bresch-Desjardins type \cite{bresch1,bresch2,haspot1}:
\begin{align}
&W=:u+\rho^{-2}\mu(\rho)\rho_{x}.\label{WBD}
 \end{align}
 Making use of (\ref{bdx}) and (\ref{WBD}), we can re-write the equation $(\ref{m1})_{2}$ as
\begin{equation}\label{M}
\begin{split}
&\rho (W_{t}+ uW_{x})+\gamma\rho^{\gamma+1}\mu(\rho)^{-1}W=\gamma\rho^{\gamma+1}\mu(\rho)^{-1}u+\rho(nw-nu).
\end{split}
\end{equation}
Multiplying (\ref{M}) by $p|W|^{p-2}W$ for any $p\in[2,\infty)$, and integrating the resulted equation by parts over $\mathbb{R}$, we derive by $(\ref{upper1})_{1}$ and (\ref{H1R2})-(\ref{H1R211}) that
\begin{align}
&\frac{d}{dt}\|(\rho^{\frac{1}{p}} W)(t)\|_{L^{p}}^{p}+p\gamma\int_{\mathbb{R}}\big{(}\rho^{\gamma+1}\mu(\rho)^{-1}|W|^{p}\big{)}(x,t)dx\nonumber\\
&=p\int_{\mathbb{R}} \big{[}\big{(}\gamma\rho^{\gamma+1}\mu(\rho)^{-1}u+\rho(nw-nu)\big{)}|W|^{p-2}W\big{]}(x,t)dx\nonumber\\
&\leq p\rho_{T}^{\frac{1}{p}}\big{(}\gamma\rho_{T}^{\gamma}\|u(t)\|_{L^{p}}+\|nw(t)\|_{L^{p}}+\|n(t)\|_{L^{p}}\|u(t)\|_{L^{\infty}}\big{)}\|(\rho^{\frac{1}{p}} W)(t)\|_{L^{p}}^{p-1}\nonumber\\
&\leq pC_{T}(1+\|u(t)\|_{L^{\infty}\cap L^{p}})\|(\rho^{\frac{1}{p}} W)(t)\|_{L^{p}}^{p-1},\label{www}
\end{align}
where $C_{T}>0$ is a constant independent of $p\in[2,\infty)$. Dividing (\ref{www}) by $(\|(\rho^{\frac{1}{p}}W)(t)\|_{L^{p}}^{p}+\varepsilon^{p})^{\frac{p-1}{p}}$, integrating the resulted inequality over $[0,t]$, and then taking the limit $\varepsilon\rightarrow0$, we have
\begin{equation}\label{Mlp}
\begin{split}
&\sup_{t\in[0,T]}\|(\rho^{\frac{1}{p}}W)(t)\|_{L^{p}}\\
&\quad\leq\|\rho_{0}^{\frac{1}{p}}(u_{0}+\rho_{0}^{-2}\mu(\rho_{0})(\rho_{0})_{x})\|_{L^{p}}+ C_{T}+C_{T}\int_{0}^{T}\|u(t)\|_{L^{\infty}\cap L^{p}}dt.
\end{split}
\end{equation}
Due to $(\ref{basicR})_{4}$, $(\ref{upper1})_{1}$, (\ref{Mlp}) and the fact
\begin{equation}\nonumber
\begin{split}
&\|\rho_{0}^{\frac{1}{p}}(u_{0}+\rho_{0}^{-2}\mu(\rho_{0})(\rho_{0})_{x})\|_{L^{p}}\\
&\quad\leq \|\rho_{0}\|_{L^{\infty}}^{\frac{1}{p}}\big{(}\|u_{0}\|_{L^2}^{\frac{2}{p}}\|u_{0}\|_{L^{\infty}}^{1-\frac{2}{p}}+\|\rho_{0}^{-2}\mu(\rho_{0})(\rho_{0})_{x}\|_{L^2}^{\frac{2}{p}}\|\rho_{0}^{-2}\mu(\rho_{0})(\rho_{0})_{x}\|_{L^{\infty}}^{1-\frac{2}{p}}\big{)}\\
&\quad\leq \big{(}\|\rho_{0}\|_{L^{\infty}}+1\big{)}^{\frac{1}{2}}\big{(}\|u_{0}\|_{L^2}+\|u_{0}\|_{L^{\infty}}+\|\rho_{0}^{-2}\mu(\rho_{0})(\rho_{0})_{x}\|_{L^2}+\|\rho_{0}^{-2}\mu(\rho_{0})(\rho_{0})_{x}\|_{L^{\infty}}\big{)},
\end{split}
\end{equation}
it follows
\begin{equation}\label{Mlp1}
\begin{split}
&\sup_{t\in[0,T]}\|W(t)\|_{L^{p}}\leq \widetilde{C}_{T},
\end{split}
\end{equation}
where $\widetilde{C}_{T}>0$ is a sufficiently large constant independent of $p\in[2,\infty)$. One gets after taking the limit $p\rightarrow\infty$ in (\ref{Mlp1}) that
\begin{equation}\label{Mlinfty}
\begin{split}
&\sup_{t\in[0,T]}\|W(t)\|_{L^{\infty}}\leq \widetilde{C}_{T}.
\end{split}
\end{equation}
Choosing $p=2$ in (\ref{Mlp1}), we have by $(\ref{upper1})_{1}$ and (\ref{WBD}) that
\begin{equation}\nonumber
\begin{split}
&\sup_{t\in[0,T]}\|\rho_{x}(t)\|_{L^{2}}\leq \rho_{T}^{2}\sup_{t\in[0,T]}\|(W-u)(t)\|_{L^{2}}\\
&\quad\quad\quad\quad\quad\quad~~\leq \rho_{T}^{2}\sup_{t\in[0,T]}\big{(}\|W(t)\|_{L^2}+\|u(t)\|_{L^2}\big{)}\leq C_{T}.
\end{split}
\end{equation}
Similarly to (\ref{uh1})-(\ref{utl2}), one can obtain
 \begin{equation}\nonumber
 \begin{split}
 \sup_{t\in[0,T]}\|u_{x}(t)\|_{L^2}+\|(u_{xx},u_{t})\|_{L^2(0,T;L^2)}\leq C_{T},
 \end{split}
 \end{equation}
 which together with $(\ref{upper1})_{1}$, (\ref{WBD}) and (\ref{Mlinfty}) leads to
  \begin{equation}\nonumber
 \begin{split}
 &\sup_{t\in[0,T]}\|\rho_{x}(t)\|_{L^{\infty}}\leq \rho_{T}^2 \sup_{t\in[0,T]}\|(W-u)(t)\|_{L^{\infty}}\\
 &\quad\quad\quad\quad\quad\quad\quad\leq \rho_{T}^2 \sup_{t\in[0,T]}\big{(}\|W(t)\|_{L^{\infty}}+\|u(t)\|_{L^{\infty}}\big{)}\leq C_{T}.
 \end{split}
 \end{equation}
 The proof of Lemma \ref{lemma210} is completed.
\end{proof}

\begin{lemma}\label{lemma211}
Let $T>0$, and $(\rho,u,f)$ be any regular solution to the IVP $(\ref{m1})$-$(\ref{kappa})$ for $t\in(0,T]$. Then, under the assumptions of Theorem \ref{theorem12}, it holds
\begin{equation}\label{H2R}
\begin{split}
&\sup_{t\in[0,T]}\Big{(}\|(\rho_{xx},u_{xx})(t)\|_{L^2}+\|\rho_{t}(t)\|_{H^1}+t^{\frac{1}{2}}\|(u_{xxx},u_{tx})(t)\|_{L^2}\\
&\quad\quad\quad\quad+\|f(t)\|_{\mathcal{C}^{1}\cap \mathcal{H}^{1}}+\|f_{t}(t)\|_{\mathcal{C}^{0}\cap\mathcal{L}^{2}}\Big{)}+\|(u_{xxx},u_{xt})\|_{L^2(0,T;L^2)}\leq C_{T},
\end{split}
\end{equation}
where $C_{T}>0$ is a constant.
\end{lemma}
\begin{proof}
We establish the $\mathcal{H}^1$-estimate of $f$ so as to obtain the $H^1$-estimate of the term $ nw-un$ on the right-hand side of the equation (\ref{unorho}). Since the rest of the proof follows similarly to $(\ref{fxfv})$-(\ref{fc1t}) in Lemma \ref{lemma24}, we omit the details here. Multiplying $(\ref{fxfv})_{1}$ by $f_{x}$, and integrating the resulted equation by parts over $\mathbb{R}\times\mathbb{R}$, one deduces by $(\ref{upper1})_{1}$ and (\ref{H1R}) that
\begin{equation}\label{fxddt}
\begin{split}
&\frac{1}{2}\frac{d}{dt}\|f_{x}(t)\|_{\mathcal{L}^2}^2\leq \frac{1}{2}\rho_{T}\|f_{x}(t)\|_{\mathcal{L}^2}^2+\|\rho_{x}(t)\|_{L^{\infty}}\int_{\mathbb{R}\times\mathbb{R}} (f|f_{x}|)(x,v,t)dvdx\\
&\quad\quad\quad\quad\quad\quad\quad+\big{(}\|\rho_{x}(t)\|_{L^{\infty}}\widetilde{R}_{T}+\|(\rho u)_{x}(t)\|_{L^{\infty}}\big{)}\int_{\mathbb{R}\times\mathbb{R}}(|f_{v}||f_{x}|)(x,v,t)dvdx\\
&\quad\quad\quad\quad\quad\quad\leq C_{T}\big{(}1+\|u(t)\|_{H^2}\big{)}\|(f_{x},f_{v})(t)\|_{\mathcal{L}^2}^2+C_{T}.
\end{split}
\end{equation}
Similarly, one has
\begin{equation}\nonumber
\begin{split}
&\frac{1}{2}\frac{d}{dt}\|f_{v}(t)\|_{\mathcal{L}^2}^2\leq C_{T}\|(f_{x},f_{v})(t)\|_{\mathcal{L}^2}^2+C_{T},
\end{split}
\end{equation}
which together with (\ref{H1R}), (\ref{fxddt}) and the Gr${\rm{\ddot{o}}}$nwall's inequality leads to
\begin{equation}\label{fxfvH1}
\begin{split}
\sup_{t\in[0,T]}\|(f_{x},f_{v})(t)\|_{\mathcal{L}^2}\leq C_{T}e^{C_{T}\|u\|_{L^1(0,T;H^2)}}\big{(}\|[(f_{0})_{x},(f_{0})_{v}]\|_{\mathcal{L}^2}+1\big{)}\leq C_{T}.
\end{split}
\end{equation}
By $(\ref{m33})$, $(\ref{upper1})_{1}$, (\ref{H1R}) and (\ref{fxfvH1}), it also holds
\begin{equation}\label{ft}
\begin{split}
\sup_{t\in[0,T]}\|f_{t}(t)\|_{\mathcal{L}^2}=\sup_{t\in[0,T]}\|\big{[}v f_{x}+\rho(u-v)f_{v}-\rho f\big{]}(t)\|_{\mathcal{L}^2}\leq C_{T}.
\end{split}
\end{equation}
Thus, we obtain by $(\ref{nw})$, (\ref{H1R}) and (\ref{fxfvH1})-(\ref{ft}) that
\begin{equation}\label{nwx2}
\left\{
\begin{split}
&\sup_{t\in[0,T]}\|(n_{x},n_{t})(t)\|_{L^{2}}\leq 2\widetilde{R}_{T} \sup_{t\in[0,T]}\|(f_{x},f_{t})(t)\|_{\mathcal{L}^{2}}\leq C_{T},\\
&\sup_{t\in[0,T]}\|[(nw)_{x},(nw)_{t}](t)\|_{L^{2}}\leq 2\widetilde{R}_{T}^2 \sup_{t\in[0,T]}\|(f_{x},f_{t})(t)\|_{\mathcal{L}^{2}}\leq C_{T}.
\end{split}
\right.
\end{equation}
\end{proof}

\underline{\it\textbf{Proof of global existence in spatial real line:~}}  Let $T>0$ be any given time. With the help of Lemmas \ref{lemma28}-\ref{lemma211} and Lemma \ref{lemma42} below, we can obtain a sequence of strong solutions $(\rho^{\varepsilon},u^{\varepsilon},f^{\varepsilon})$ for $\varepsilon\in(0,1)$ on $\mathbb{R}\times\mathbb{R}\times [0,T]$. 

Then, one can show under the assumptions (\ref{a2}) of Theorem \ref{theorem12} that $(\rho^{\varepsilon},u^{\varepsilon},f^{\varepsilon})$ satisfies the estimates established in Lemmas \ref{lemma28}-\ref{lemma210} uniformly with respect to $\varepsilon\in(0,1)$. Since the locally compact embedding $H^2(\mathbb{R}) \hookrightarrow\hookrightarrow L^2_{loc}(\mathbb{R})$  holds, by virtue of the Aubin-Lions's lemma and cantor's diagonal argument, up to a subsequence (still denoted by $(\rho^{\varepsilon},u^{\varepsilon},f^{\varepsilon})$), we have the strong convergence
\begin{equation}\nonumber
\begin{split}
&(\rho^{\varepsilon},u^{\varepsilon})\rightarrow (\rho,u)\quad\text{strongly in}~C([0,T];C^{0}(K))\times C([0,T];C^{0}(K)),\quad\text{as}\quad \varepsilon\rightarrow 0,
\end{split}
\end{equation}
for any compact subset $K$ of $\mathbb{R}$. Similarly to the proof of global existence to the IVP $(\ref{m1})$-$(\ref{kappa})$ in spatial periodic domain (i.e., Theorem \ref{theorem11}), one can prove that the approximate sequence $(\rho^{\varepsilon},u^{\varepsilon},f^{\varepsilon})$ converges to the corresponding weak solution $(\rho,u,f)$ in the sense of distributions.

\section{Uniqueness}

In this section, we prove the uniqueness of the weak solution to the IVP $(\ref{m1})$-$(\ref{kappa})$ satisfying $(\ref{r1})$-(\ref{r11}) in spatial periodic domain, and similar arguments can be applied to the case of spatial real line.

\begin{prop}\label{prop31}
Let $\Omega=\mathbb{T}$ and $T>0$. Assume that the initial data $(\rho_{0},u_{0},f_{0})$ satisfies $(\ref{a1})$. If $(\rho_{1},u_{1},f_{1})$ and $(\rho_{2},u_{2},f_{2})$ are two solutions to the IVP $(\ref{m1})$-$(\ref{kappa})$ satisfying $(\ref{r1})$-$(\ref{r11})$ on $\mathbb{T}\times\mathbb{R}\times[0,T]$ with the same initial data $(\rho_{0},u_{0},f_{0})$, then it holds $(\rho_{1},u_{1},f_{1})=(\rho_{2},u_{2},f_{2})$ a.e. in $\mathbb{T}\times\mathbb{R}\times (0,T)$.
\end{prop}
\begin{proof}
Define the bi-characteristic curves $(X_{i}(t),V_{i}(t))$, $i=1,2$, for any $(x,v,t)\in \mathbb{T}\times\mathbb{R}\times[0,T]$ as
\begin{equation}\label{XiVi}
\left\{
\begin{split}
&\frac{d}{dt}X_{i}(t)=V_{i}(t),\\
&~~X_{i}(0)=x,\\
&\frac{d}{dt}V_{i}(t)=\rho_{i}(X_{i}(t),t)\big{(}u_{i}(X_{i}(t),t)-V_{i}(t)\big{)},\\
&~~V_{i}(0)=v.
\end{split}
\right .
\end{equation}
Since it holds for $i=1,2$ that
\begin{equation}\label{f1f2}
\begin{split}
&(f_{i})_{t}+\dive_{x,v} \bigg{(}f_{i} {\scriptsize{\left(           
  \begin{matrix}   
    v  \\
    \rho_{i}(u_{i}-v)  \\ 
  \end{matrix}
\right)}}\bigg{)}
=(f_{i})_{t}+v(f_{i})_{x}+\rho_{i}(u_{i}-v)(f_{i})_{v}-\rho_{i}f_{i}=0,
\end{split}
\end{equation}
we can express $f_{i}$ along $(X_{i}(t),V_{i}(t))$ for any $(x,v,t)\in \mathbb{T}\times\mathbb{R}\times[0,T]$ as
\begin{equation}
\begin{split}
f_{i}\big{(}X_{i}(t),V_{i}(t),t\big{)}=e^{\int_{0}^{t}\rho_{i}(X_{i}(s),s)ds}f_{0}(x,v),\quad i=1,2.\label{fformula12}
\end{split}
\end{equation}
Due to (\ref{f1f2}), Proposition 1.2 in \cite{majda1} and the fact
\begin{equation}\nonumber
\begin{split}
\dive_{x,v}
 {\scriptsize{\left(
  \begin{matrix}
    v  \\
    \rho_{i}\big{(}u_{i}-v\big{)}  \\
  \end{matrix}
\right)}}=-\rho_{i},
\end{split}
\end{equation}
 we have
\begin{equation}\nonumber
\partial_{t}J_{i}=-J_{i} \rho_{i} (X_{i}(t),t),\quad J_{i}=:
\bigg{|}{\rm{det}}
 {\scriptsize{\left(
  \begin{matrix}
    \partial_{x}X (t) & \partial_{x}V (t)  \\
    \partial_{v}X (t) &\partial_{v}V (t)  \\
  \end{matrix}
\right)}}
\bigg{|},
\end{equation}
which implies
\begin{equation}\label{det}
\begin{split}
J_{i}=e^{-\int_{0}^{t}\rho_{i}(s,X_{i}(s))ds},\quad i=1,2.
\end{split}
\end{equation}
We combine (\ref{fformula12})-(\ref{det}) together to show the following property about the coordinates transform $(x,v)\mapsto (X_{i}(t),V_{i}(t))$:
\begin{equation}
\begin{split}
&\int_{\mathbb{T}\times\mathbb{R}} \phi(X_{i}(t),V_{i}(t))f_{0}(x,v)dvdx\\
&\quad=\int_{\mathbb{T}\times\mathbb{R}} \phi(X_{i}(t),V_{i}(t))f_{i}(X_{i}(t),V_{i}(t),t)dX_{i}(t)dV_{i}(t)\\
&\quad=\int_{\mathbb{T}\times\mathbb{R}} \phi(x,v) f_{i}(x,v,t)dvdx,\quad\forall \phi\in C^{0}(\mathbb{T}\times\mathbb{R}),\quad i=1,2.\label{pullback}
\end{split}
\end{equation}

Inspired by \cite{loeper1,han1}, we estimate the quantity
 \begin{equation}\label{Xi}
 \begin{split}
Q(t)=:\frac{1}{2}\int_{\Omega\times\mathbb{R}} f_{0}(x,v)\big{(}|X_{1}(t)-X_{2}(t)|^2+|V_{1}(t)-V_{2}(t)|^2\big{)}dvdx.
 \end{split}
 \end{equation}
 For $\Omega=\mathbb{T}$, making use of (\ref{XiVi}), we can obtain after a complicated computation that
\begin{equation}\label{unf}
\begin{split}
&\frac{d}{dt}Q(t)+\int_{\mathbb{T}\times\mathbb{R}} \rho_{1}(X_{1}(t),t)f_{0}(x,v)|V_{1}(t)-V_{2}(t)|^2dvdx\\
&=\int_{\mathbb{T}\times\mathbb{R}} f_{0}(x,v)\big{(}X_{1}(t)-X_{2}(t)\big{)}\big{(}V_{1}(t)-V_{2}(t)\big{)}dvdx\\
&\quad+\int_{\mathbb{T}\times\mathbb{R}} f_{0}(x,v)\big{(}V_{1}(t)-V_{2}(t)\big{)}(\rho_{1}u_{1}-\rho_{2}u_{2})(X_{1}(t),t)dvdx\\
&\quad+\int_{\mathbb{T}\times\mathbb{R}} f_{0}(x,v)\big{(}V_{1}(t)-V_{2}(t)\big{)}\big{[}(\rho_{2}u_{2})(X_{1}(t),t)-(\rho_{2}u_{2})(X_{2}(t),t)\big{]}dvdx\\
&\quad+\int_{\mathbb{T}\times\mathbb{R}} f_{0}(x,v)\big{(}V_{1}(t)-V_{2}(t)\big{)}\big{(}\rho_{2}(X_{2}(t),t)-\rho_{2}(X_{1}(t),t)\big{)}V_{2}(t)dvdx\\
&\quad+\int_{\mathbb{T}\times\mathbb{R}} f_{0}(x,v)\big{(}V_{1}(t)-V_{2}(t)\big{)}(\rho_{2}-\rho_{1})(X_{1}(t),t)V_{2}(t)dvdx\\
&\quad=:\sum_{i=1}^{5}I_{i}^{1}.
\end{split}
\end{equation}
The right-hand side terms of (\ref{unf}) can be estimated as follows. First, one has
\begin{equation}\nonumber
\begin{split}
I_{1}^{1}\leq Q(t).
\end{split}
\end{equation}
By (\ref{pullback}), it holds
\begin{equation}\nonumber
\begin{split}
&I_{2}^{1}=\int_{\mathbb{T}\times\mathbb{R}} f_{0}(x,v)\big{(}V_{1}(t)-V_{2}(t)\big{)}\big{[}\rho_{1}(u_{1}-u_{2})+(\rho_{1}-\rho_{2})u_{2}\big{]}(X_{1}(t),t)dvdx\\
&\quad \leq \frac{1}{2}\int_{\mathbb{T}\times\mathbb{R}} f_{0}(x,v)|V_{1}(t)-V_{2}(t)|^2dvdx\\
&\quad\quad+\frac{1}{2}\int_{\mathbb{T}\times\mathbb{R}} \big{(}\rho_{1}^2|u_{1}-u_{2}|^2+|\rho_{1}-\rho_{2}|^2|u_{2}|^2\big{)}(X_{1}(t),t)f_{0}(x,v)dvdx\\
&\quad\leq  Q(t)+\frac{1}{2}\int_{\mathbb{T}} \big{(}\rho_{1}^2|u_{1}-u_{2}|^2+|\rho_{1}-\rho_{2}|^2|u_{2}|^2\big{)}(x,t)\big{(}\int_{\mathbb{R}} f_{1}(x,v,t)dv\big{)}dx\\
&\quad\leq Q(t)+\frac{1}{2}\|\int_{\mathbb{R}} f_{1}(\cdot,v,t)dv\|_{L^{\infty}} {\footnotesize{ \Big{(}\|\rho_{1}(t)\|_{L^{\infty}}\|\big{(}\sqrt{\rho_{1}}(u_{1}-u_{2})\big{)}(t)\|_{L^2}^2}}\\
&\quad\quad {\footnotesize{ +\|u_{2}(t)\|_{L^{\infty}}^2\|(\rho_{1}-\rho_{2})(t)\|_{L^2}^2\Big{)}}}.
\end{split}
\end{equation}
For the term $I_{3}^1$, we have
\begin{equation}\nonumber
\begin{split}
&I_{3}^1\leq \|(\rho_{2}u_{2})_{x}(t)\|_{L^{\infty}}\int_{\mathbb{T}\times\mathbb{R}} f_{0}(x,v)|V_{1}(t)-V_{2}(t)||X_{1}(t)-X_{2}(t)|dvdx\\
&\quad\leq  {\footnotesize{ \Big{(}\|\rho_{2}(t)\|_{L^{\infty}}\|(u_{2})_{x}(t)\|_{L^{\infty}}+\|(\rho_{2})_{x}(t)\|_{L^{\infty}}\|u_{2}(t)\|_{L^{\infty}}\Big{)}}}Q(t).
\end{split}
\end{equation}
We estimate the term $I_{4}^1$ as
\begin{equation}\nonumber
\begin{split}
&I_{4}^{1}\leq \|(\rho_{2})_{x}(t)\|_{L^{\infty}}\int_{\mathbb{T}\times\mathbb{R}} f_{0}(x,v)|V_{1}(t)-V_{2}(t)||X_{1}(t)-X_{2}(t)||V_{2}(t)|dvdx\\
&\quad\leq R_{1}\|(\rho_{2})_{x}(t)\|_{L^{\infty}}Q(t),
\end{split}
\end{equation}
where in the last inequality we have used the following fact for $i=1,2$:
\begin{equation}\label{V1V2v}
\begin{split}
&f_{0}(x,v)|V_{i}(t)|\leq  {\footnotesize{ \Big{(}e^{-\rho_{-}t}|v|+\rho_{+}\int_{0}^{t}e^{-\rho_{-}s}|u_{i}(x,t)|ds\Big{)}}}f_{0}(x,v)\leq R_{1}f_{0}(x,v),
\end{split}
\end{equation}
derived from $(\ref{a1})_{2}$, (\ref{r11}), (\ref{uinfty1}), (\ref{Rt}) and (\ref{XiVi}). By (\ref{pullback}) and (\ref{V1V2v}), it also holds
\begin{equation}\nonumber
\begin{split}
&I_{5}^1\leq \frac{1}{2}\int_{\mathbb{T}\times\mathbb{R}} f_{0}(x,v)|V_{1}(t)-V_{2}(t)|^2|V_{2}(t)|dvdx\\
&\quad\quad+\frac{1}{2}\int_{\mathbb{T}\times\mathbb{R}}|(\rho_{1}-\rho_{2})(X_{1}(t),t)|^2|V_{2}(t)|f_{0}(x,v)dvdx\\
&\quad\leq R_{1} Q(t)+\frac{R_{1}}{2}\|\int_{\mathbb{R}} f_{1}(\cdot,v,t)dv\|_{L^{\infty}}\|(\rho_{1}-\rho_{2})(t)\|_{L^2}^2.
\end{split}
\end{equation}
Substituting the above estimates of $I^{1}_{i}$, $1,...,5$, into (\ref{unf}), and employing (\ref{r1})-(\ref{r11}), we have
\begin{equation}
\begin{split}
&\frac{d}{dt}Q(t)\leq C_{T}\big{(}1+\|(u_{2})_{x}(t)\|_{L^{\infty}}\big{)}Q(t)\\
&\quad\quad\quad\quad+C_{T} {\footnotesize{ \Big{(}\|\big{(}\sqrt{\rho_{1}}(u_{1}-u_{2})\big{)}(t)\|_{L^2}^2+\|(\rho_{1}-\rho_{2})(t)\|_{L^2}^2\Big{)}}}.\label{widetildef1}
\end{split}
\end{equation}

Next, we turn to estimate $(\rho_{1}-\rho_{2},u_{1}-u_{2})$. With loss of generalization, we prove the case $\beta=0$. It is easy to verify
\begin{align}
&(\rho_{1}-\rho_{2})_{t}+(\rho_{1})_{x}(u_{1}-u_{2})+\rho_{1}(u_{1}-u_{2})_{x}+(u_{2}(\rho_{1}-\rho_{2}))_{x}=0,\label{widetilderho}\\
&\rho_{1}(u_{1}-u_{2})_{t}+\rho_{1}u_{1}(u_{1}-u_{2})_{x}-2(u_{1}-u_{2})_{xx}+\rho_{1}(u_{1}-u_{2})\int_{\mathbb{R}} f_{1}dv\nonumber\\
&\quad=-(\rho_{1}^{\gamma}-\rho_{2}^{\gamma})_{x}-\rho_{1}(u_{1}-u_{2})(u_{2})_{x}-(\rho_{1}-\rho_{2})((u_{2})_{t}+u_{2}(u_{2})_{x})\nonumber\\
&\quad\quad+\int_{\mathbb{R}}(\rho_{1}-\rho_{2})(v-u_{2})f_{1}dv+\int_{\mathbb{R}}\rho_{2}(v-u_{2})(f_{1}-f_{2})dv.\label{widetildeu}
\end{align}
Multiplying (\ref{widetilderho}) by $(\rho_{1}-\rho_{2})$, integrating the resulted equation by parts over $\mathbb{T}$, and using (\ref{r1})-(\ref{r11}), we have
\begin{equation}
\begin{split}
&\frac{1}{2}\frac{d}{dt}\|(\rho_{1}-\rho_{2})(t)\|_{L^2}^2\\
&~~\leq C_{T}\big{(}1+\|(u_{2})_{x}(t)\|_{L^{\infty}}\big{)}  {\footnotesize{ \Big{(}\|(\rho_{1}-\rho_{2})(t)\|_{L^2}^2+\|\big{(}\sqrt{\rho_{1}}(u_{1}-u_{2})\big{)}(t)\|_{L^{2}}\Big{)}}}\\
&~~\quad+\frac{1}{100}\|(u_{1}-u_{2})_{x}(t)\|_{L^2}^2.\label{widetilderho1}
\end{split}
\end{equation}
In addition, we multiply (\ref{widetildeu}) by $(u_{1}-u_{2})$ and integrate the resulted equation by parts over $\mathbb{T}$ to obtain
\begin{equation}\label{widetildeu1}
\begin{split}
&\frac{1}{2}\frac{d}{dt}\|\big{(}\sqrt{\rho_{1}}(u_{1}-u_{2})\big{)}(t)\|_{L^2}^2+2\|(u_{1}-u_{2})_{x}(t)\|_{L^2}^2+\|\big{(}\sqrt{\rho_{1}f_{1}}(u_{1}-u_{2})\big{)}(t)\|_{L^2}^2\\
&~~\leq\|(u_{2})_{x}(t)\|_{L^{\infty}}\|\sqrt{\rho_{1}}(u_{1}-u_{2})(t)\|_{L^2}^2+\|(\rho_{1}^{\gamma}-\rho_{2}^{\gamma})(t)\|_{L^2}\|(u_{1}-u_{2})_{x}(t)\|_{L^2}\\
&\quad~~+\|((u_{2})_{t}+u_{2}(u_{2})_{x})(t)\|_{L^2}\|(\rho_{1}-\rho_{2})(t)\|_{L^2}\|(u_{1}-u_{2})(t)\|_{L^{\infty}}\\
&\quad~~+\|(\rho_{1}-\rho_{2})(t)\|_{L^2}\|(u_{1}-u_{2})(t)\|_{L^{\infty}}\|\int_{\mathbb{R}} (v-u_{2}(\cdot,t))f_{1}(\cdot,v,t)dv\|_{L^2}\\
&\quad~~+ {\footnotesize{ \int_{\mathbb{T}\times\mathbb{R}}\rho_{2}(x,t)\big{(}v-u_{2}(x,t)\big{)}(u_{1}-u_{2})(x,t)(f_{1}-f_{2})(x,v,t)dvdx}}\\
&~~=:\|(u_{2})_{x}(t)\|_{L^{\infty}}\|\sqrt{\rho_{1}}(u_{1}-u_{2})(t)\|_{L^2}^2+\sum_{i=1}^{4}I_{i}^{2}.
\end{split}
\end{equation}
One can show
\begin{equation}\nonumber
\begin{split}
&I_{1}^{2}\leq C\|(\rho_{1}-\rho_{2})(t)\|_{L^2}^2+\frac{1}{100}\|(u_{1}-u_{2})_{x}(t)\|_{L^2}^2.
\end{split}
\end{equation}
For $I_{2}^2$, we have
\begin{equation}\nonumber
\begin{split}
&I_{2}^{2}\leq C\big{(}1+\|(u_{2})_{t}(t)\|_{L^2}^2\big{)} {\footnotesize{  \Big{(}\|\big{(}\sqrt{\rho_{1}}(u_{1}-u_{2})\big{)}(t)\|_{L^2}^2+\|(\rho_{1}-\rho_{2})(t)\|_{L^2}^2\Big{)}}}\\
&\quad\quad+\frac{1}{100}\|(u_{1}-u_{2})_{x}(t)\|_{L^2}^2,
\end{split}
\end{equation}
where we have used (\ref{r1})-(\ref{r11}), the Young's inequality and the fact
\begin{equation}\label{widetildeuinfty}
\begin{split}
\|(u_{1}-u_{2})(t)\|_{L^{\infty}}\leq C\rho_{-}^{-\frac{1}{4}}\|\big{(}\sqrt{\rho_{1}}(u_{1}-u_{2})\big{)}(t)\|_{L^2}^{\frac{1}{2}}\|(u_{1}-u_{2})_{x}(t)\|_{L^2}^{\frac{1}{2}}.
\end{split}
\end{equation}
 Similarly, by (\ref{r1}) and (\ref{widetildeuinfty}), one has
\begin{equation}\nonumber
\begin{split}
&I_{3}^{2}\leq C {\footnotesize{ \Big{(}\|\big{(}\sqrt{\rho_{1}}(u_{1}-u_{2})\big{)}(t)\|_{L^2}^2+\|(\rho_{1}-\rho_{2})(t)\|_{L^2}^2\Big{)}}}+\frac{1}{100}\|(u_{1}-u_{2})_{x}(t)\|_{L^2}^2.
\end{split}
\end{equation}
With the help of (\ref{pullback}), we obtain
 \begin{equation}\label{change1}
 \begin{split}
 &\int_{\mathbb{T}\times\mathbb{R}}\rho_{2}(x,t)\big{(} v-u_{2}(x,t)\big{)}(u_{1}-u_{2})(x,t)f_{1}(x,v,t)dvdx\\
 &\quad=\int_{\mathbb{T}\times\mathbb{R}}\rho_{2}(X_{1}(t),t)\big{(}V_{1}(t)-u_{2}(X_{1}(t),t)\big{)}(u_{1}-u_{2})(X_{1}(t),t)f_{0}(x,v)dvdx,
 \end{split}
 \end{equation}
and similarly,
  \begin{equation}\label{change2}
 \begin{split}
 &\int_{\mathbb{T}\times\mathbb{R}}\rho_{2}(x,t)\big{(} v-u_{2}(x,t)\big{)}(u_{1}-u_{2})(x,t)f_{2}(x,v,t)dvdx\\
 &\quad=\int_{\mathbb{T}\times\mathbb{R}}\rho_{2}(X_{2}(t),t)\big{(}V_{2}(t)-u_{2}(X_{2}(t),t)\big{)}(u_{1}-u_{2})(X_{2}(t),t)f_{0}(x,v)dvdx.
 \end{split}
 \end{equation}
By (\ref{change1})-(\ref{change2}), the last term $I_{4}^2$ on the right-hand side of (\ref{widetildeu1}) can be estimated as follows:
  \begin{equation}\label{widetildeu2}
 \begin{split}
&I_{4}^2=\int_{\mathbb{T}\times\mathbb{R}} \big{[}(\rho_{2}u_{2})(X_{2}(t),t)-(\rho_{2}u_{2})(X_{1}(t),t)\big{]}(u_{1}-u_{2})(X_{1}(t),t)f_{0}(x,v)dvdx\\
&\quad\quad+\int_{\mathbb{T}\times\mathbb{R}} \big{(}\rho_{2}(X_{1}(t),t)-\rho_{2}(X_{2}(t),t)\big{)}V_{2}(t)(u_{1}-u_{2})(X_{1}(t),t)f_{0}(x,v)dvdx\\
&\quad\quad+\int_{\mathbb{T}\times\mathbb{R}} \rho_{2}(X_{1}(t),t)\big{(}V_{1}(t)-V_{2}(t)\big{)}(u_{1}-u_{2})(X_{1}(t),t)f_{0}(x,v)dvdx\\
&\quad\quad+\int_{\mathbb{T}\times\mathbb{R}}\rho_{2}(X_{2}(t),t)\big{(}V_{2}(t)-u_{2}(X_{2}(t),t)\big{)}\big{[}(u_{1}-u_{2})(X_{1}(t),t)\\
&\quad\quad\quad-(u_{1}-u_{2})(X_{2}(t),t)\big{]}f_{0}(x,v)dvdx\\
&\quad=:\sum_{i=1}^{4}I^3_{i}.
 \end{split}
 \end{equation}
To estimate $I_{1}^3$, we obtain by (\ref{r1})-(\ref{r11}) and (\ref{widetildeuinfty}) that
$$
\begin{aligned}
&I_{1}^3\leq  \|(\rho_{2}u_{2})_{x}(t)\|_{L^{\infty}}\|(u_{1}-u_{2})(t)\|_{L^{\infty}}\int_{\mathbb{T}\times\mathbb{R}} |X_{1}(t)-X_{2}(t)|f_{0}(x,v)dvdx\\
&\quad\leq  C_{T}\big{(}1+\|(u_{2})_{x}(t)\|_{L^{\infty}}^2\big{)} {\footnotesize{ \Big{(}Q(t)+\|\big{(}\sqrt{\rho_{1}}(u_{1}-u_{2})\big{)}(t)\|_{L^2}^2\Big{)}}}+\frac{1}{100}\|(u_{1}-u_{2})_{x}(t)\|_{L^2}^2.
\end{aligned}
$$
Similarly, one can derive by (\ref{r1})-(\ref{r11}), (\ref{V1V2v}) and (\ref{widetildeuinfty}) that
$$
\begin{aligned}
&I_{2}^3\leq  \|(\rho_{2})_{x}(t)\|_{L^{\infty}}\|(u_{1}-u_{2})(t)\|_{L^{\infty}}\int_{\mathbb{T}\times\mathbb{R}} |X_{1}(t)-X_{2}(t)||V_{2}(t)|f_{0}(x,v)dvdx\\
&\quad\leq Q(t)+C_{T}\|\big{(}\sqrt{\rho_{1}}(u_{1}-u_{2})\big{)}(t)\|_{L^2}^2+\frac{1}{100}\|(u_{1}-u_{2})_{x}(t)\|_{L^2}^2,
\end{aligned}
$$
and
$$
\begin{aligned}
&I_{3}^3\leq \|\rho_{2}(t)\|_{L^{\infty}}\|(u_{1}-u_{2})(t)\|_{L^{\infty}}\int_{\mathbb{T}\times\mathbb{R}}|V_{1}(t)-V_{2}(t)| f_{0}(x,v)dvdx\\
 &\quad\leq Q(t)+C_{T}\|\big{(}\sqrt{\rho_{1}}(u_{1}-u_{2})\big{)}(t)\|_{L^2}^2+\frac{1}{100}\|(u_{1}-u_{2})_{x}(t)\|_{L^2}^2.
\end{aligned}
$$
It follows from (\ref{r1})-(\ref{r11}), (\ref{V1V2v}) and (\ref{2102}) in Lemma \ref{lemma32} below that
\begin{equation}
\begin{split}
&I_{4}^{3}\leq C \int_{\mathbb{T}\times\mathbb{R}}  {\footnotesize{ \sup_{r\in(0,1)}\frac{1}{2r}\Big{(}\int^{X_{1}(t)+r}_{X_{1}(t)-r}+\int^{X_{2}(t)+r}_{X_{2}(t)-r}\Big{)}|(u_{1}-u_{2})_{z}(z,t)|}}dz \\
&\quad\quad \quad {\footnotesize{ \times\big{(}|V_{2}(t)+|u_{2}(X_{2}(t),t)|\big{)} |X_{1}(t)-X_{2}(t)|}}f_{0}(x,v)dvdx\\
&\quad\leq C_{T}Q^{\frac{1}{2}}(t)\Big{(}\int_{\mathbb{T}\times\mathbb{R}}  {\footnotesize{ \Big{(}\sup_{r\in(0,1)}\frac{1}{2r}\int^{X_{1}(t)+r}_{X_{1}(t)-r}|(u_{1}-u_{2})_{z}(z,t)|dz\Big{)}^2}}f_{0}(x,v)dvdx\Big{)}^{\frac{1}{2}}\\
&\quad\quad+ C_{T}Q^{\frac{1}{2}}(t)\Big{(}\int_{\mathbb{T}\times\mathbb{R}}  {\footnotesize{ \Big{(}\sup_{r\in(0,1)}\frac{1}{2r}\int^{X_{2}(t)+r}_{X_{2}(t)-r}|(u_{1}-u_{2})_{z}(z,t)|dz\Big{)}^2}}f_{0}(x,v)dvdx\Big{)}^{\frac{1}{2}}.\label{asdd}
\end{split}
\end{equation}
Making use of $(\ref{r1})$, (\ref{pullback}) and (\ref{2101}) in Lemma \ref{lemma31} below, we can obtain
\begin{equation}\label{asdd1}
\begin{split}
&\int_{\mathbb{T}\times\mathbb{R}} {\footnotesize{ \Big{(}\sup_{r\in(0,1)}\frac{1}{2r}\int^{X_{1}(t)+r}_{X_{1}(t)-r}|(u_{1}-u_{2})_{z}(z,t)|dz\Big{)}^2}}f_{0}(x,v)dvdx\\
&\quad=\int_{\mathbb{T}\times\mathbb{R}}  {\footnotesize{ \Big{(}\sup_{r\in(0,1)}\frac{1}{2r}\int^{x+r}_{x-r}|(u_{1}-u_{2})_{z}(z,t)|dz\Big{)}^2}} f_{1}(x,v,t)dvdx\\
&\quad\leq C\|\int_{\mathbb{R}} f_{1}(\cdot,v,t)dv\|_{L^{\infty}}\Big{\|} {\footnotesize{ \sup_{r\in(0,1)}\frac{1}{2r}\int^{x+r}_{x-r}|(u_{1}-u_{2})_{z}(z,t)|dz}}\Big{\|}_{L^2}^2\\
&\quad\leq C_{T}\|(u_{1}-u_{2})_{x}(t)\|_{L^2}^2,
\end{split}
\end{equation}
and similarly,
\begin{equation}\label{asdd2}
\begin{split}
&\int_{\mathbb{T}\times\mathbb{R}} {\footnotesize{\Big{(}\sup_{r\in(0,1)}\frac{1}{2r}\int^{X_{2}(t)+r}_{X_{2}(t)-r}|(u_{1}-u_{2})_{z}(z,t)|dz\Big{)}^2}}f_{0}(x,v)dvdx\\
&\quad\leq C_{T}\|(u_{1}-u_{2})_{x}(t)\|_{L^2}^2.
\end{split}
\end{equation}
The combination of (\ref{asdd})-(\ref{asdd2}) leads to
\begin{equation}\nonumber
\begin{split}
&I_{4}^{3}\leq C_{T}Q^{\frac{1}{2}}(t)\|(u_{1}-u_{2})_{x}(t)\|_{L^2}\leq C_{T}Q(t)+\frac{1}{100}\|(u_{1}-u_{2})_{x}(t)\|_{L^2}^2.
\end{split}
\end{equation}
Substituting the estimates of $I_{i}^3$, $i=1,...,4$, into (\ref{widetildeu2}), we obtain
\begin{equation}\nonumber
\begin{split}
&I_{4}^2\leq C_{T}\big{(}1+\|(u_{2})_{x}(t)\|_{L^{\infty}}^2\big{)} {\footnotesize{ \Big{(}Q(t)+\|\big{(}\sqrt{\rho_{1}}(u_{1}-u_{2})\big{)}(t)\|_{L^2}^2\Big{)}}}+\frac{1}{100}\|(u_{1}-u_{2})_{x}(t)\|_{L^2}^2.
\end{split}
\end{equation}
Substituting the estimates of $I_{i}^2$, $i=1,...,4$, into (\ref{widetildeu1}), and adding the resulted inequality, (\ref{widetildef1}) and (\ref{widetilderho1}) together, we have
\begin{equation}\nonumber
\begin{split}
&\frac{d}{dt} {\footnotesize{ \Big{(}Q(t)+\|\big{(}\sqrt{\rho_{1}}(u_{1}-u_{2})\big{)}(t)\|_{L^2}^2+\|(\rho_{1}-\rho_{2})(t)\|_{L^2}^2\Big{)}}}+\|(u_{1}-u_{2})_{x}(t)\|_{L^2}^2\\
&\leq C_{T}\big{(}1+\|(u_{2})_{x}(t)\|_{L^{\infty}}^2+\|(u_{2})_{t}(t)\|_{L^2}^2\big{)} {\footnotesize{ \Big{(}Q(t)+\|\big{(}\sqrt{\rho_{1}}(u_{1}-u_{2})\big{)}(t)\|_{L^2}^2+\|(\rho_{1}-\rho_{2})(t)\|_{L^2}^2\Big{)}}}.
\end{split}
\end{equation}
Thus, one deduces by the Gr$\rm{\ddot{o}}$nwall's inequality, (\ref{r1}), the fact $Q(0)=0$ and Lemma \ref{lemma32} below that
$$
(\rho_{1},u_{1},f_{1})=(\rho_{2},u_{2},f_{2}),\quad\text{a.e. in}~ \mathbb{T} \times\mathbb{R}\times (0,T).
$$
The proof of Proposition \ref{prop31} is completed.
\end{proof}

\begin{lemma}[\!\!\cite{acerbi1,han1,stein1}] \label{lemma31} Let $\Omega =\mathbb{T} $ or $\mathbb{R}$. For any $g\in L^2(\Omega)$, it holds
\begin{equation}\label{2101}
\begin{split}
\Big{\|} {\footnotesize{ \sup_{r\in(0,1)}\frac{1}{2r}\int^{x+r}_{x-r}|g(z)|dz}}\Big{\|}_{L^2}\leq C\|g\|_{L^2},
\end{split}
\end{equation}
If assume further $g_{x}\in L^2(\Omega),$ then for a.e. $x,y\in \Omega$, we have
\begin{equation}\label{2102}
\begin{split}
|g(x)-g(y)|\leq C|x-y|  \Big{(} {\footnotesize{ \sup_{r\in(0,1)}\frac{1}{2r}\int^{x+r}_{x-r}|g_{z}(z)|dz+\sup_{r\in(0,1)}\frac{1}{2r}\int^{y+r}_{y-r}|g_{z}(z)|dz }}\Big{)}.
\end{split}
\end{equation}
\end{lemma}

\begin{lemma}[\!\!\cite{loeper1}] \label{lemma32} Let $\Omega =\mathbb{T} $ or $\mathbb{R}$, and $Q(t)$ be defined through $(\ref{Xi})$. Then $Q(t)=0$ implies $f_{1}=f_{2}$ a.e. in $\Omega\times\mathbb{R}\times (0,t)$.
\end{lemma}

Similarly, we can prove
\begin{prop}\label{prop41}
Let $\Omega=\mathbb{R}$ and $T>0$. Assume that the initial data $(\rho_{0},u_{0},f_{0})$ satisfies $(\ref{a2})$. If $(\rho_{1},u_{1},f_{1})$ and $(\rho_{2},u_{2},f_{2})$ are two solutions to the IVP $(\ref{m1})$-$(\ref{kappa})$ satisfying $(\ref{r2})$-$(\ref{r21})$ on $\mathbb{R}\times\mathbb{R}\times[0,T]$ with the same initial data $(\rho_{0},u_{0},f_{0})$, then it holds $(\rho_{1},u_{1},f_{1})=(\rho_{2},u_{2},f_{2})$ a.e. in $\mathbb{R}\times\mathbb{R}\times (0,T)$.
\end{prop}

\section{Appendix A: Local well-posedness}

In this Appendix, we prove the local well-posednes of strong solution to the IVP (\ref{m1})-(\ref{kappa}) for spatial periodic domain, and similar arguments can be applied to the case of spatial real line.
\begin{lemma} \label{lemma41}
Let $\Omega=\mathbb{T}$. Suppose that the initial data $(\rho_{0},u_{0},f_{0})$ satisfies
\begin{equation}\label{alocal}
\left\{
\begin{split}
&\inf_{x\in\mathbb{T}}\rho_{0}(x)>0,\quad \rho_{0}\in H^2(\mathbb{T} ),\quad u_{0}\in H^2(\mathbb{T} ),\\
&0\leq f_{0}\in C^{1}(\mathbb{T}\times\mathbb{R}),\quad \text{{\rm{Supp}}}_{v}f_{0}(x,\cdot)\subset \{v\in\mathbb{R}~\big{|}~|v|\leq  R_{0}\},\quad x\in\mathbb{T},
\end{split}
\right.
\end{equation}
where $R_{0}>0$ is a constant. Then there is a time $T_{*}>0$ such that the IVP $(\ref{m1})$-$(\ref{kappa})$ admits a unique strong solution satisfying
\begin{equation}\label{rlocal}
\left\{
\begin{split}
&\inf_{(x,t)\in \mathbb{T}\times[0,T_{*}]}\rho(x,t)>0,\quad \rho\in C([0,T_{*}];H^2(\mathbb{T} )),\\
&~~u\in C([0,T_{*}];H^2(\mathbb{T} ))\cap L^2(0,T_{*};H^3(\mathbb{T})),\\
&~~0\leq f\in C([0,T_{*}];C^{1}(\mathbb{T} \times\mathbb{R} )),\\
&~~\text{{\rm{Supp}}}_{v}f(x,\cdot,t)\subset \{v\in\mathbb{R}~\big{|}~|v|\leq  R_{0}+1\},\quad (x,t)\in\mathbb{T}\times[0,T_{*}].
\end{split}
\right .
\end{equation}
\end{lemma}
\begin{proof} \emph{Step 1: Construction of approximate sequence.} The iteration scheme for the approximate solutions $(\rho^{n+1},u^{n+1},f^{n+1})$, $n\geq 0$, is defined by solving the following equations:
 \begin{equation}\label{m1n}
\left\{
 \begin{split}
 &(\rho^{n+1})_{t}+(u^{n}\rho^{n+1})_{x}=0,\\
 & (u^{n+1})_{t}-\frac{1}{\rho^{n}}(\mu(\rho^{n})(u^{n+1})_{x})_{x}=-u^{n}(u^{n})_{x}-\gamma(\rho^{n})^{\gamma-2}(\rho^{n})_{x}-\int_{\mathbb{R}}(u^{n}-v)f^{n}dv,\\
& (f^{n+1})_{t}+v(f^{n+1})_{x}+(\rho^{n}(u^{n}-v)f^{n+1})_{v}=0,\quad (x,v)\in \mathbb{T}\times\mathbb{R},~t>0,
\end{split}
\right .
\end{equation}
with the initial data
\begin{align}
&(\rho^{n+1}(x,0),u^{n+1}(x,0),f^{n+1}(x,v,0))=(\rho_{0}(x),u_{0}(x),f_{0}(x,v)), \quad (x,v)\in\mathbb{T}\times\mathbb{R}.\label{dn}
\end{align}
Define the work space
\begin{equation}\nonumber
\begin{split}
&\mathcal{W}(T)=:\Big{\{}(\rho,u,f)~\big{|}~(\rho,u) \in C([0,T];H^2(\mathbb{T} )),\quad  0\leq f\in C([0,T];C^{1}(\mathbb{T} \times\mathbb{R} )),\\
 &\quad \quad\quad \quad \quad \quad\quad \quad\quad~\frac{1}{2}\rho_{0-}\leq \rho(x,t)\leq 2\rho_{0+},\quad (x,t)\in\mathbb{T}\times[0,T],\\
 &\quad \quad\quad \quad \quad \quad\quad \quad\quad~\text{{\rm{Supp}}}_{v}f^{n}(x,\cdot,t)\subset \{v\in\mathbb{R}~\big{|}~|v|\leq  R_{0}+1\},\quad (x,t)\in\mathbb{T}\times[0,T],\\
 &\quad \quad\quad \quad \quad \quad\quad \quad \quad\sup_{t\in[0,T]}\big{(}\|(\rho,u)(t)\|_{H^2}^2+\|f(t)\|_{\mathcal{C}^1}\big{)}+\frac{1}{2\rho_{0+}}\|u_{x}\|_{L^2(0,T;H^2)}^2\leq 2E_{0}+1.\Big{\}},
\end{split}
\end{equation}
where $\rho_{0+}>0, \rho_{0-}>0$ and $E_{0}>0$ are constants given by
 \begin{equation}\nonumber
 \begin{split}
 &\rho_{0-}=:\inf_{x\in\mathbb{T}}\rho_{0}(x),\quad \rho_{0+}=:\sup_{x\in\mathbb{T}}\rho_{0}(x),\quad E_{0}=:\|(\rho_{0},u_{0})\|_{H^2}^2+\|f_{0}\|_{\mathcal{C}^1}.
 \end{split}
 \end{equation}
Set $(\rho^{0},u^{0},f^{0})=: (\rho_{0},u_{0},f_{0})$. Suppose
\begin{align}
(\rho^{n},u^{n},f^{n})\in\mathcal{W}(T_{0}),\label{n}
\end{align}
for some time $T_{0}>0$. By the standard theory of linear O.D.E. systems \cite{hsi1}, for given $u^{n}\in C([0,T_{0}];H^2(\mathbb{T}))\hookrightarrow C([0,T_{0}];C^{1}(\mathbb{T}))$, there is a unique particle path $\mathcal{X}_{n}^{x,t}(s)$ for any $(x,t)\in \mathbb{T}\times[0,T_{0}]$ defined by
 \begin{equation}\nonumber
 \left\{
 \begin{split}
 &\frac{d}{ds}\mathcal{X}_{n}^{x,t}(s)=u^{n}(\mathcal{X}_{n}^{x,t}(s),s),\quad s\in[0,t],\\
 &~\mathcal{X}_{n}^{x,t}(t)=x.
 \end{split}
 \right.
 \end{equation}
Thus, the linear transport equation $(\ref{m1n})_{1}$ can be solved by
 \begin{equation}\label{rhon1}
 \begin{split}
 \rho^{n+1}(x,t)=e^{-\int_{0}^{t} (u^{n})_{x}(\mathcal{X}_{n}^{x,t}(s),s)ds}\rho_{0}(\mathcal{X}_{n}^{x,t}(0)).
 \end{split}
 \end{equation}
Similarly, for given $(\rho^{n},u^{n})\in C([0,T_{0}];H^2(\mathbb{T}))\hookrightarrow C([0,T_{0}];C^{1}(\mathbb{T}))$, we can define uniquely the bi-characteristic curves
\begin{equation}\label{XVn}
\left\{
\begin{split}
&\frac{d}{ds}X_{n}^{x,v,t}(s)=V_{n}^{x,v,t}(s),\quad s\in [0,t],\\
&\quad X_{n}^{x,v,t}(t)=x,\\
&\frac{d}{ds}V_{n}^{x,v,t}(s)=\rho^{n}(X_{n}^{x,v,t}(s),s) \big{(}u^{n}(X_{n}^{x,v,t}(s),s)-V_{n}^{x,v,t}(s)\big{)},\quad s\in [0,t],\\
&\quad V_{n}^{x,v,t}(t)=v,
\end{split}
\right .
\end{equation}
for any $(x,v,t)\in \mathbb{T}\times\mathbb{R}\times[0,T_{0}]$, so that the solution $f^{n+1}$ for the linear transport equation $(\ref{m1n})_{3}$ can be represented by
\begin{align}
f^{n+1}(x,v,t)=e^{\int_{0}^{t}\rho^{n}(X_{n}^{x,v,t}(s),s)ds}f_{0}(X_{n}^{x,v,t}(0),V_{n}^{x,v,t}(0))\geq 0.\label{fn1}
\end{align}
Finally, by the standard theorem of linear parabolic equations \cite{jllions1,valli1}, $(\ref{m1n})_{2}$ admits a unique solution $u^{n+1}\in C([0,T_{0}];H^2(\mathbb{T}))\cap L^2(0,T_{0};H^3(\mathbb{T}))$.

 \emph{Step~2: Uniform estimates.} We claim that there is a time $T_{0}>0$ such that for each $n\geq 0$, if $(\rho^{n},u^{n},f^{n})$ satisfies (\ref{n}), then it is also true for $(\rho^{n+1},u^{n+1},f^{n+1})$. Indeed, by (\ref{n})-(\ref{rhon1}), one has
  \begin{equation}\label{upperrhon}
 \begin{split}
& 0<e^{-NE_{0}^{\frac{1}{2}}T_{0}}\rho_{0-}\leq \rho^{n+1}(x,t)\leq e^{NE_{0}^{\frac{1}{2}}T_{0}}\rho_{0+},\quad (x,t)\in \mathbb{T}\times[0,T_{0}],
 \end{split}
 \end{equation}
where $N>0$ denotes a a suitably large constant depending only on $\rho_{0+},\rho_{0-}$ and $R_{0}$. In addition, the equation $(\ref{m1n})_{1}$ implies
    \begin{equation}\label{nablarhon}
 \begin{split}
 &\partial_{t}\sum_{i=0}^{2}\partial_{x}^{i}\rho^{n+1}+u^{n}\sum_{i=0}^{2}\partial_{x}^{i+1}\rho^{n+1}\\
 &\quad=-\sum_{i=0}^{2}\partial_{x}^{i}(\rho^{n+1} (u^{n})_{x})-(u^{n})_{x}(\rho^{n+1})_{x}-2(u^{n})_{x}(\rho^{n+1})_{xx}-(u^{n})_{xx}(\rho^{n+1})_{x}.
  \end{split}
 \end{equation}
 We multiply (\ref{nablarhon}) by $\sum_{i=0}^{2}\partial_{x}^{i}\rho^{n+1}$ and integrate the resulted equation by parts over $\mathbb{T}$ to get
  \begin{equation}\label{nablarhon1}
 \begin{split}
 &\frac{1}{2}\frac{d}{dt}\|\rho^{n+1}(t)\|_{H^2}^2\\
 &~~\leq \frac{1}{2}\|(u^{n})_{x}(t)\|_{L^{\infty}}\|\rho^{n+1}(t)\|_{H^2}^2+N\|(\rho^{n+1} (u^{n})_{x})(t)\|_{H^2}\|\rho^{n+1}(t)\|_{H^2}\\
 &~~\leq N\|(u^{n})_{x}(t)\|_{H^2}\|\rho^{n+1}(t)\|_{H^2}^2.
   \end{split}
 \end{equation}
Thus, it follows from (\ref{n}), (\ref{nablarhon1}) and the Gr${\rm{\ddot{o}}}$nwall's inequality that
   \begin{equation}\label{rhonH2}
 \begin{split}
 &\sup_{t\in[0,T_{0}]}\|\rho^{n+1}(t)\|_{H^2}^2\leq e^{N\|(u^{n})_{x}\|_{L^1(0,T;H^2)}}\|\rho_{0}\|_{H^2}^2\leq e^{N E_{0}^{\frac{1}{2}}T_{0}^{\frac{1}{2}}}\|\rho_{0}\|_{H^2}^2.
  \end{split}
 \end{equation}

We are ready to estimate $u^{n+1}$. By $(\ref{m1n})_{2}$ and (\ref{n}), it is easy to obtain
       \begin{equation}\label{uL2}
 \begin{split}
 &\frac{1}{2}\frac{d}{dt}\|u^{n+1}(t)\|_{H^1}^2+\frac{1}{2\rho_{0+}}\|(u^{n+1})_{x}(t)\|_{H^1}^2\\
 &~~\leq N\big{(}E_{0}^{\frac{1}{2}}\|(u^{n+1})_{x}(t)\|_{L^2}\|u^{n+1}(t)\|_{L^2}+(E_{0}^{\frac{1}{2}}+1)E_{0}^{\frac{1}{2}}\|u^{n+1}(t)\|_{L^2}\big{)}.
     \end{split}
 \end{equation}
Then we differentiate $(\ref{m1})_{2}$ with respect to $x$ to have
        \begin{equation}\label{unxxx}
 \begin{split}
 & (u^{n+1})_{xt}-\frac{\mu(\rho^{n})}{\rho^{n}}(u^{n+1})_{xxx}=(\frac{\mu'(\rho^{n})}{\rho^{n}}(\rho^{n})_{x}(u^{n+1})_{x})_{x}+(\frac{\mu(\rho^{n})}{\rho^{n}})_{x}(u^{n+1})_{xx}\\
 &\quad\quad\quad\quad\quad\quad\quad\quad\quad\quad\quad\quad+\big{(}-u^{n}(u^{n})_{x}-\gamma(\rho^{n})^{\gamma-2}(\rho^{n})_{x}-\int_{\mathbb{R}}(u^{n}-v)f^{n}dv\big{)}_{x}.
     \end{split}
 \end{equation}
Multiplying $(\ref{unxxx})$ by $-(u^{n+1})_{xxx}$, integrating the resulted equation by parts over $\mathbb{T}$, and making use of (\ref{n}), we derive
  \begin{equation}\nonumber
 \begin{split}
  &\frac{1}{2}\frac{d}{dt}\|(u^{n+1})_{xx}(t)\|_{L^2}^2+\frac{1}{2\rho_{0+}}\|(u^{n+1})_{xxx}(t)\|_{L^2}^2\\
  &~~\leq N\Big{(}\|(\rho^{n})_{x}(t)\|_{L^{\infty}}\|(u^{n+1})_{xx}(t)\|_{L^2}+\|(\rho^{n})_{xx}(t)\|_{L^{2}}\|(u^{n+1})_{x}(t)\|_{L^{\infty}}+\|u^{n}(t)\|_{H^1}\|(u^{n})_{x}(t)\|_{H^1}\\
    &\quad~~+\|(\rho^{n})_{x}(t)\|_{H^1}+\big{(}1+\|u^{n}(t)\|_{H^1}\big{)}\|\int_{\mathbb{R}}(1+|v|)f^{n}(x,\cdot,t)dv\|_{H^1}\Big{)}\|(u^{n+1})_{xxx}(t)\|_{L^2}\\
 &~~\leq \frac{1}{4\rho_{0+}}\|(u^{n+1})_{xxx}(t)\|_{L^2}^2+ N E_{0}\|u^{n+1}(t)\|_{H^2}^2+N(1+E_{0})E_{0},
     \end{split}
 \end{equation}
which together with (\ref{uL2}) gives rise to
   \begin{equation}\label{un1H2}
 \begin{split}
 &\sup_{t\in[0,T_{0}]}\|u^{n+1}(t)\|_{H^2}^2+\frac{1}{2\rho_{0+}}\|(u^{n+1})_{x}\|_{L^2(0,T_{0};H^2)}^2\\
 &\quad\leq e^{NE_{0}T_{0} }\big{(}\|u_{0}\|_{H^2}^2+N(1+E_{0})E_{0}T_{0}\big{)}.
     \end{split}
 \end{equation}

To estimate $f^{n+1}$, we obtain by (\ref{XVn}) that
 \begin{equation}\label{xv123}
\begin{split}
&v=e^{-\int_{0}^{t}\rho^{n}(X_{n}^{x,v,t}(\tau),\tau)d\tau}V_{n}^{x,v,t}(0)+\int_{0}^{t}e^{-\int_{\tau}^{t}\rho^{n}(X_{n}^{x,v,t}(\omega),\omega)d\omega}\rho^{n} u^{n}(X_{n}^{x,v,t}(\tau),\tau)d\tau.
\end{split}
\end{equation}
It follows from $(\ref{alocal})_{2}$, (\ref{n}), (\ref{fn1}) and (\ref{xv123}) for any $(x,t)\in\mathbb{T}\times [0,T_{0}]$ that
\begin{align}
&\sup_{\{v\in\mathbb{R}~|~f^{n+1}(x,v,t)\neq 0\}}|v|\leq R_{0}+2\rho_{0+}\|u^{n}\|_{L^{\infty}(0,T_{0};L^{\infty})}T_{0}\leq R_{0}+N E_{0}^{\frac{1}{2}}T_{0},
\end{align}
from which we get
   \begin{align}
  &\text{{\rm{Supp}}}_{v}f^{n+1}(x,\cdot,t)\subset \{v\in\mathbb{R}~\big{|}~|v|\leq  R_{0}+N E_{0}T_{0}\},\quad (x,t)\in\mathbb{T}\times[0,T_{0}].\label{fn1compact}
     \end{align}
Similarly to (\ref{fxfv})-(\ref{fw1inftycnsv}), one has by a direct computation that
\begin{equation}\label{fw1inftycnsvppp}
\begin{split}
\sup_{t\in[0,T_{0}]}\|f^{n+1}(t)\|_{\mathcal{C}^{1}}\leq e^{N E_{0}^{\frac{1}{2}}T_{0}}\big{(}\|f_{0}\|_{\mathcal{C}^{1}}+N(1+E_{0}^{\frac{1}{2}})E^{\frac{1}{2}}_{0}T_{0}\big{)}.
\end{split}
\end{equation}
Choose
     \begin{equation}\label{T0}
     \left\{
 \begin{split}
 &M_{0}=: N \max\big{\{} (1+E_{0})E_{0},(1+E_{0}^{\frac{1}{2}})E^{\frac{1}{2}}_{0}\big{\}} ,\\
 &T_{0}=:\min\big{\{}\frac{\log{2}}{NM_{0}},\frac{(\log{2})^2}{N^2M_{0}^2},\frac{1}{2NM_{0}}\big{\}}.
     \end{split}
     \right.
 \end{equation}
Then by the previous estimates (\ref{fn1})-(\ref{upperrhon}), (\ref{rhonH2}), (\ref{un1H2}) and (\ref{fn1compact})-(\ref{fw1inftycnsvppp}), we conclude that  for each $n\geq 0$, the strong solution $(\rho^{n+1},u^{n+1},f^{n+1})$ to the IVP (\ref{m1n})-(\ref{dn}) belongs to $\mathcal{W}(T_{0})$ as long as $(\rho^{n},u^{n},f^{n})$ satisfies (\ref{n}). By repeating the procedure used above, we can obtain an approximate sequence $(\rho^{n},u^{n},f^{n})$ satisfying the uniform estimates in $\mathcal{W}(T_{0})$.

  \emph{Step 3: Compactness and convergence.} Let $M_{0}>0$ and $T_{0}>0$ be given by (\ref{T0}).  We aim to prove the convergence of the approximate sequence $(\rho^{n},u^{n},f^{n})$ to the corresponding strong solution of the original IVP (\ref{m1})-(\ref{kappa}). Since $(\rho^{n},u^{n},f^{n})$ is uniformly bounded in $\mathcal{W}(T_{0})$, one can show after a tedious computation that
 \begin{equation}\nonumber
 \begin{split}
 &\sup_{\tau\in [0,t]}\Big{(}\|(\rho^{n+1}-\rho^{n}, u^{n+1}-u^{n})(\tau)\|_{H^1}^2+\|(f^{n+1}-f^{n})(\tau)\|_{\mathcal{C}^{0}}\Big{)}+\|(u^{n+1}-u^{n})_{x}\|_{L^2(0,t;H^1)}^2\\
 &\quad\leq N_{*} M_{0}t\sup_{\tau\in [0,t]}\Big{(}\|(\rho^{n}-\rho^{n-1}, u^{n}-u^{n-1})(\tau)\|_{H^1}^2+\|(f^{n}-f^{n-1})(\tau)\|_{\mathcal{C}^{0}}\Big{)},\quad t\in[0,T_{0}],
 \end{split}
 \end{equation}
 where $N_{*}>0$ is a constant depending only on $\rho_{0+},\rho_{0-}$ and $R_{0}$. Choosing
$$
 T_{*}=\min\big{\{}\frac{1}{2N_{*}M_{0}},T_{0}\big{\}},
$$
 we have
 \begin{equation}\nonumber
 \begin{split}
&\sum_{n=0}^{\infty} \sup_{\tau\in [0,T_{*}]}\Big{(}\|(\rho^{n+1}-\rho^{n}, u^{n+1}-u^{n})(\tau)\|_{H^1}^2+\|(f^{n+1}-f^{n})(\tau)\|_{\mathcal{C}^{0}}\Big{)}\\
&\quad\leq \sum_{n=0}^{\infty} \frac{1}{2^{n}} \sup_{\tau\in [0,T_{*}]}\Big{(}\|(\rho^{1}-\rho^{0}, u^{1}-u^{0})(\tau)\|_{H^1}^2+\|(f^{1}-f^{0})(\tau)\|_{\mathcal{C}^{0}}\Big{)}<\infty,
 \end{split}
 \end{equation}
 which implies that there is a limit $(\rho,u,f)$ so that as $n\rightarrow\infty$, it holds
 \begin{equation}\label{strongn}
 \left\{
 \begin{split}
 &(\rho^{n},u^{n})\rightarrow(\rho,u)\quad \text{in}\quad L^{\infty}(0,T_{*};H^1(\mathbb{T}))\times L^{\infty}(0,T_{*};H^1(\mathbb{T})),\\
 &f^{n}\rightarrow f\quad \text{in}\quad L^{\infty}(0,T_{*};C^{0}(\mathbb{T}\times\mathbb{R})).
 \end{split}
 \right.
 \end{equation}
 Due to the estimates of $(\rho^{n},u^{n},f^{n})$ uniformly in $n\geq0$, as $n\rightarrow \infty$, it holds up to a subsequence (still denoted by $(\rho^{n},u^{n},f^{n})$) that
 \begin{equation}\label{weakn}
 \left\{
 \begin{split}
 &(\rho^{n},f^{n}) \overset{\ast}{\rightharpoonup} (\rho,f) \quad\text{in}\quad L^{\infty}(0,T_{*};H^2(\mathbb{T}))\times L^{\infty}(0,T_{*};W^{1,\infty}(\mathbb{T}\times\mathbb{R})),\\
 &u^{n} \rightharpoonup u\quad\text{in}\quad L^2(0,T_{*};H^3(\mathbb{T})).
 \end{split}
 \right.
 \end{equation}
Thus, we can pass to the limit $n\rightarrow\infty$ for the equations $(\ref{m1n})$ in the sense of distributions and prove that $(\rho,u,f)$ is indeed a strong solution to the IVP $(\ref{m1})$-$(\ref{dn})$ on $\mathbb{T}\times\mathbb{R}\times[0,T_{*}]$. 

In addition, applying the theory of renormalized solutions in \cite{diperna1}, we have $\rho\in C([0,T_{*}];H^1(\mathbb{T}))$. By $(\ref{m1})_{2}$, (\ref{strongn})-(\ref{weakn}) and the uniform estimate of $(\rho^{n},u^{n},f^{n})$, one can show $u_{t}\in L^2(0,T_{*}; H^1(\mathbb{T}))$, and therefore $u\in C([0,T_{*}];H^2(\mathbb{T}))$ follows from $u\in L^2(0,T_{*};H^3(\mathbb{T}))\cap H^1(0,T_{*};H^1(\mathbb{T}))$. Moreover, with the help of $(\rho,u) \in C([0,T_{*}];H^2(\mathbb{T}))\hookrightarrow C([0,T_{*}];C^{1}(\mathbb{T}))$, it is easy to verify the $C^1$-regularity of the bi-characteristic curves $(X^{x,v,t}(s),V^{x,v,t}(s))$ defined by (\ref{XV}) with respect to  $(x,v)\in\mathbb{T}\times\mathbb{R}$, which together with (\ref{fformula}) implies $f\in C([0,T_{*}];C^{1}(\mathbb{T}\times\mathbb{R}))$. According to Proposition \ref{prop31}, the strong solution $(\rho,u,f)$ is unique. The proof of Lemma \ref{lemma41} is completed.
\end{proof}

Similarly, we have
\begin{lemma} \label{lemma42}
Let $\Omega=\mathbb{R}$. Suppose that the initial data $(\rho_{0},u_{0},f_{0})$ satisfies
\begin{equation}\nonumber
\left\{
\begin{split}
&\inf_{x\in\mathbb{R}}\rho_{0}(x)>0,\quad \rho_{0}-\widetilde{\rho}\in H^2(\mathbb{R} ),\quad u_{0}\in H^2(\mathbb{R} ),\\
&0\leq f_{0}\in L^1\cap L^{\infty}\cap H^1\cap C^{1}(\mathbb{R} \times\mathbb{R}),\quad \text{{\rm{Supp}}}_{v}f_{0}(x,\cdot)\subset \{v\in\mathbb{R}~\big{|}~|v|\leq  \widetilde{R}_{0}\},\quad x\in\mathbb{R},
\end{split}
\right.
\end{equation}
where $\widetilde{\rho}>0$ and $\widetilde{R}_{0}>0$ are two constants. Then there is a time $T_{**}>0$ such that the IVP $(\ref{m1})$-$(\ref{kappa})$ admits a unique strong solution satisfying
\begin{equation}\nonumber
\left\{
\begin{split}
&\inf_{(x,t)\in \mathbb{R}\times[0,T_{**}]}\rho(x,t)>0,\quad \rho-\widetilde{\rho}\in C([0,T_{**}];H^2(\mathbb{R} )),\\
&~~u\in C([0,T_{**}];H^2(\mathbb{R} ))\cap L^2(0,T_{**};H^3(\mathbb{R})),\\
&~~0\leq f\in C([0,T_{**}];L^1\cap H^1\cap C^{1}(\mathbb{R} \times\mathbb{R} ))\cap L^{\infty}(0,T_{**},L^{\infty}(\mathbb{R} \times\mathbb{R} )),\\
& ~~\text{{\rm{Supp}}}_{v}f(x,\cdot,t)\subset \{v\in\mathbb{R}~\big{|}~|v|\leq  \widetilde{R}_{0}+1\},~(x,t)\in\mathbb{R}\times [0,T_{**}].
\end{split}
\right .
\end{equation}
\end{lemma}

\section{Appendix B: exponential time-decay rate}

In this Section, we establish the a-priori estimates of $\rho$ uniformly in time and dynamical behaviors of the spatial periodic solution $(\rho, u,f)$ to the IVP (\ref{m1})-(\ref{kappa}) to the corresponding equilibrium state $(\overline{\rho_{0}}, u_{s},n\delta(v-u_{s}))$ in the terms of the relative entropy, where the constants $\overline{\rho_{0}}$ and $u_{s}$ are given by (\ref{rhoinfty}). First, we recall the definition of Wasserstein distance \cite{villani1} as follows:

 \begin{defn}\label{defn11}{\textit{
 Let $f_{1},f_{2}$ be two Borel probability measures on $X=\mathbb{T}\times\mathbb{R}$. For $p\geq 1$, the Wasserstein distance of order $p$ between $f_{1}$ and $f_{2}$ is defined by the formula
\begin{equation}\nonumber
\begin{split}
{\rm{W}}_{p}(f_{1},f_{2})=:(\inf_{\varpi\in \Gamma(f_{1},f_{2})}\int_{X\times X}|z-z'|^{p}d\varpi(z,z'))^{\frac{1}{p}},
\end{split}
\end{equation}
where $\Gamma(f_{1},f_{2})$ denotes the collection of all measures on $X\times X$ with first and second marginal respectively equal to $f_{1}$ and $f_{2}$. By the Monge-Kantorovitch duality, the Wasserstein distance of order $1$ can be represented equivalently by the formula
\begin{equation}\nonumber
\begin{split}
{\rm{W}}_{1}(f_{1},f_{2})=\sup\Big{\{}\int_{X}\phi(z)df_{1}(z)-\int_{X}\phi(z)df_{2}(z): \forall\phi\in Lip(X),~\|\nabla \phi\|_{L^{\infty}(X)}\leq 1\Big{\}}.
\end{split}
\end{equation}}}
 \end{defn}
~\par

Introduce the relative entropy and corresponding dissipation rate:
\begin{equation}\label{ED}
 \left\{
\begin{split}
&\mathcal{E}(t)=:\int_{\mathbb{T}}\big{(}\frac{1}{2}\rho|u-m_{1}|^2+\Pi_{\gamma}(\rho|\overline{\rho_{0}})\big{)}(x,t)dx+\int_{\mathbb{T}\times\mathbb{R}}\frac{1}{2}|v-m_{2}(t)|^2f(x,v,t)dvdx\\
&\quad\quad~~+\frac{\overline{\rho_{0}}~\overline{n_{0}}}{2(\overline{n_{0}}+\overline{\rho_{0}})}|(m_{2}-m_{1})(t)|^2,\\ &\mathcal{D}(t)=:\int_{\mathbb{T}}(\mu(\rho)|u_{x}|^2)(x,t)dx+\int_{\mathbb{T}\times\mathbb{R}}\rho(x,t)|u(x,t)-v|^2f(x,v,t)dvdx.
\end{split}
\right.
\end{equation}
where $m_{1}(t)$ and $\Pi_{\gamma}(\rho|\overline{\rho_{0}})$ are denoted by (\ref{m1m2}), and $m_{2}(t)$ is given by
\begin{equation}\nonumber
\begin{split}
m_{2}(t)=:\frac{\overline{nw}(t)}{\overline{n_{0}}}.
\end{split}
\end{equation}
We have the relative energy equality:
\begin{lemma}\label{lemma51}
Let $(\rho,u,f)$ be the global solution to the IVP $(\ref{m1})$-$(\ref{kappa})$ given by Theorem \ref{theorem11}. Then, under the assumptions of Theorem 1.1, it holds
\begin{equation}\label{entropydissipation}
\begin{split}
\frac{d}{dt}\mathcal{E}(t)+\mathcal{D}(t)=0.
\end{split}
\end{equation}
\end{lemma}
\begin{proof}
Multiplying (\ref{m21111}) by $(u-m_{1})$ and integrating the resulted equation by parts over $\mathbb{T}$, we obtain
\begin{equation}\label{um12}
\begin{split}
&\frac{1}{2}\frac{d}{dt}\int_{\mathbb{T}}\big{(}\rho|u-m_{1}|^2\big{)}(x,t)dx+\int_{\mathbb{T}}(\mu(\rho)|u_{x}|^2)(x,t)dx\\
&\quad=-\int_{\mathbb{T}}[(\rho^{\gamma}-\overline{\rho_{0}}^{\gamma})_{x}(u-m_{1})](x,t)dx-\int_{\mathbb{T}}[\rho n(u-w)(u-m_{1})](x,t)dx\\
&\quad\quad+\frac{1}{\overline{\rho_{0}}}\int_{\mathbb{T}} [\rho n (u-w)](x,t)dx\int_{\mathbb{T}}[\rho (u-m_{1})](x,t)dx\\
&\quad=\int_{\mathbb{T}}(\rho^{\gamma}u_{x})(x,t)dx-\int_{\mathbb{T}}[\rho n(u-w)(u-m_{1})](x,t)dx,
\end{split}
\end{equation}
where we have used the fact
$$
\int_{\mathbb{T}}[\rho (u-m_{1})](x,t)dx=0.
$$
We get after a direct computation that
\begin{equation}\label{rho12}
\begin{split}
&\frac{d}{dt}\int_{\mathbb{T}}\Pi_{\gamma}(\rho|\overline{\rho_{0}})(x,t)dx=\frac{d}{dt}\int_{\mathbb{T}}\frac{\rho^{\gamma}(x,t)}{\gamma-1}dx=-\int_{\mathbb{T}}(\rho^{\gamma}u_{x})(x,t)dx.
\end{split}
\end{equation}
Since it holds
\begin{equation}\label{m2t}
\begin{split}
\frac{d}{dt}m_{2}(t)=\frac{1}{\overline{n_{0}}}\frac{d}{dt}\int_{\mathbb{T}\times\mathbb{R}} vf(x,v,t)dvdx=\frac{1}{\overline{n_{0}}}\int_{\mathbb{T}\times\mathbb{R}}[\rho n(u-w)](x,t)dx,
\end{split}
\end{equation}
and
\begin{equation}\label{m20}
\begin{split}
\int_{\mathbb{T}\times\mathbb{R}}(v-m_{2}(t))f(x,v,t)dvdx=0,
\end{split}
\end{equation}
 multiplying the Vlasov equation $(\ref{m1})_{3}$ by $\frac{1}{2}|v-m_{2}|^2$ and integrating the resulted equation by parts over $\mathbb{T}\times\mathbb{R}$, we have
\begin{equation}\label{fm12}
\begin{split}
&\frac{d}{dt}\int_{\mathbb{T}\times\mathbb{R}}\frac{1}{2}|v-m_{2}(t)|^2f(x,v,t)dvdx\\
&\quad=\int_{\mathbb{T}\times\mathbb{R}}\frac{1}{2}|v-m_{2}(t)|^2f_{t}(x,v,t)dvdx+\int_{\mathbb{T}\times\mathbb{R}}(v-m_{2}(t))(m_{2}(t))_{t}f(x,v,t)dvdx\\
&\quad=\int_{\mathbb{T}\times\mathbb{R}}\frac{1}{2}|v-m_{2}(t)|^2\big{[}-vf_{x}-(\rho(u-v)f)_{v}\big{]}(x,v,t)dvdx\\
&\quad\quad+\frac{1}{\overline{n_{0}}}\int_{\mathbb{T}}[\rho n(u-w)](x,t)dx\int_{\mathbb{T}\times\mathbb{R}}(v-m_{2}(t))f(x,v,t)dvdx\\
&\quad=\int_{\mathbb{T}\times\mathbb{R}}\rho(x,t)(u(x,t)-v)(v-m_{2}(t))f(x,v,t)dvdx\\
&\quad=-\int_{\mathbb{T}\times\mathbb{R}}\rho(x,t)|u(x,t)-v|^2f(x,v,t)dvdx+\int_{\mathbb{T}}[\rho n(u-w)(u-m_{2})](x,t)dx.
\end{split}
\end{equation}
By virtue of $(\ref{basiccnsv})_{2}$, (\ref{momentumfluid}) and (\ref{m2t}), one can show
\begin{equation}
\begin{split}
&\frac{d}{dt}\big{(}\frac{\overline{\rho_{0}}~\overline{n_{0}}}{2(\overline{n_{0}}+\overline{\rho_{0}})}|(m_{2}-m_{1})(t)|^2\big{)}=(m_{2}-m_{1})(t)\int_{\mathbb{T}}[\rho n(u-w)](x,t)dx.\label{m1m212}
\end{split}
\end{equation}
Adding (\ref{um12}), (\ref{rho12}), (\ref{fm12}) and (\ref{m1m212}) together, we can obtain (\ref{entropydissipation}).
\end{proof}

In order to find the desired dissipative information associated with $\rho$, we need
\begin{lemma}
Let $T>0$, and $(\rho,u,f)$ be one regular solution to the IVP (\ref{m1})-(\ref{kappa}) for $t\in[0,T]$. Then, under the assumptions of Theorem 1.1, it holds
\begin{align}
&|\int_{\mathbb{T}}[\rho(u-m_{1})\mathcal{J}(\rho-\overline{\rho_{0}})](x,t)dx|\leq c_{1}(\|\big{(}\sqrt{\rho}(u-m_{1})\big{)}(t)\|_{L^2}^2+\|\Pi_{\gamma}(\rho|\overline{\rho_{0}})(t)\|_{L^1}),\label{perturb}\\
&\frac{d}{dt}\int_{\mathbb{T}}[\rho(u-m_{1})\mathcal{J}(\rho-\overline{\rho_{0}})](x,t)dx\geq c_{0}\|\Pi_{\gamma}(\rho|\overline{\rho_{0}})(t)\|_{L^1}- c_{1}(\|u_{x}(t)\|_{L^2}^2+\|\big{(}\sqrt{\rho f}(u-v)\big{)}(t)\|_{\mathcal{L}^2}^2),\label{perturb-}
\end{align}
where $c_{i}>0$, $i=1,2$, are two positive constants only depending on $\overline{\rho_{0}}$ and $\rho_{+}$, and the operator $\mathcal{J}$ is defined by (\ref{j}).
\end{lemma}
\begin{proof}
One deduce from the Young's inequality, (\ref{basiccnsv}) and (\ref{J}) that
\begin{equation}\label{ppp}
\begin{split}
&|\int_{\mathbb{T}}[\rho (u-m_{1})\mathcal{J}(\rho-\overline{\rho_{0}})](x,t)dx|\leq \frac{1}{2}(\|\big{(}\sqrt{\rho}(u-m_{1})\big{)}(t)\|_{L^2}^2+\|\rho(t)\|_{L^1}\|\mathcal{J}(\rho-\overline{\rho_{0}})(t)\|_{L^{\infty}}^2)\\
&\quad\quad\quad\quad\quad\quad\quad\quad\quad\quad\quad\quad\quad\quad~\leq \frac{1}{2}(\|\big{(}\sqrt{\rho}(u-m_{1})\big{)}(t)\|_{L^2}^2+\overline{\rho_{0}}\|(\rho-\overline{\rho_{0}})(t)\|_{L^{2}}^2).
\end{split}
\end{equation}
 Due to (\ref{rho2}) and (\ref{rhocnsv}), it holds
\begin{equation}\label{Pi2}
\left\{
\begin{split}
&\|\Pi_{\gamma}(\rho|\overline{\rho_{0}})(t)\|_{L^1}\geq \min{\{\frac{\gamma}{2}(\rho_{+}+\overline{\rho_{0}})^{\gamma-2},\overline{\rho_{0}}^{\gamma-2}\}}\|(\rho-\overline{\rho_{0}})(t)\|_{L^2}^2,\\
&\|\Pi_{\gamma}(\rho|\overline{\rho_{0}})(t)\|_{L^1}\leq \max{\{\frac{\gamma}{2}(\rho_{+}+\overline{\rho_{0}})^{\gamma-2},\overline{\rho_{0}}^{\gamma-2}\}}\|(\rho-\overline{\rho_{0}})(t)\|_{L^2}^2,
\end{split}
\right.
\end{equation}
which together with (\ref{ppp}) implies (\ref{perturb}) directly.

We are going to estimate (\ref{perturb-}). It follows from $(\ref{m1})_{1}$, (\ref{J1}) and (\ref{m1m2}) that
\begin{align}
&(\mathcal{J}(\rho-\overline{\rho_{0}}))_{t}=-\mathcal{J}((\rho u)_{x})=\overline{\rho u}-\rho u=-\rho(u-m_{1})-(\rho-\overline{\rho_{0}})m_{1}. \label{m2111}
\end{align}
Since it holds $\mathcal{J}(\rho-\overline{\rho_{0}})(1)=\mathcal{J}(\rho-\overline{\rho_{0}})(0)$, we deduce by the equations (\ref{m21111}) and (\ref{m2111}) and integration by parts that
\begin{equation}\label{perturb1}
\begin{split}
&\frac{d}{dt}\int_{\mathbb{T}}[\rho (u-m_{1})\mathcal{J}(\rho-\overline{\rho_{0}})](x,t)dx\\
&\quad=\int_{\mathbb{T}}[(\rho^{\gamma}-\overline{\rho_{0}}^{\gamma})(\rho-\overline{\rho_{0}})](x,t)dx-\int_{\mathbb{T}}[\overline{\rho_{0}}\rho|u-m_{1}|^2+\mu(\rho)u_{x}(\rho-\overline{\rho_{0}})](x,t)dx\\
&\quad\quad+\int_{\mathbb{T}}\mathcal{J}(\rho-\overline{\rho_{0}})(x,t)\Big{(}-[\rho n(u-w)](x,t)+\frac{\rho(x,t)}{\overline{\rho_{0}}}\int_{\mathbb{T}} [\rho n (u-w)](y,t)dy\Big{)}dx.
\end{split}
\end{equation}
To estimate the first term on the right-hand side of (\ref{perturb1}), by $(\ref{basiccnsv})_{1}$, we infer
\begin{equation}\label{perturb11}
\begin{split}
\int_{\mathbb{T}}[(\rho^{\gamma}-\overline{\rho_{0}}^{\gamma})(\rho-\overline{\rho_{0}})](x,t)dx=\int_{\mathbb{T}}\big{(}|\rho-\overline{\rho_{0}}|^2\frac{\rho^{\gamma}-\overline{\rho_{0}}^{\gamma}}{\rho-\overline{\rho_{0}}}\big{)}(x,t)dx\geq \gamma(\overline{\rho_{0}})^{\gamma-1}\|(\rho-\overline{\rho_{0}})(t)\|_{L^2}^2.
\end{split}
\end{equation}
One can show
\begin{equation}\label{perturb12}
\begin{split}
&|\int_{\mathbb{T}}[\overline{\rho_{0}}\rho|u-m_{1}|^2+\mu(\rho)u_{x}(\rho-\overline{\rho_{0}})](x,t)dx|\\
&\quad\leq \overline{\rho_{0}}^2\|(u-m_{1})(t)\|_{L^{\infty}}^2+\mu(\rho_{+})\|u_{x}(t)\|_{L^2}\|(\rho-\overline{\rho_{0}})(t)\|_{L^2}\\
&\quad\leq \frac{\gamma(\overline{\rho_{0}})^{\gamma-1}}{4}\|(\rho-\overline{\rho_{0}})(t)\|_{L^2}^2+(\frac{\mu^{2}(\rho_{+})}{\gamma(\overline{\rho_{0}})^{\gamma-1}}+\overline{\rho_{0}}^2)\|u_{x}(t)\|_{L^2}^2,
\end{split}
\end{equation}
where we have used $(\ref{rhocnsv})_{2}$ and (\ref{um1infty}). It also follows from $(\ref{basiccnsv})$, (\ref{rhocnsv}) and (\ref{J}) that
\begin{equation}\label{perturb13}
\begin{split}
&\Big|\int_{\mathbb{T}}\mathcal{J}(\rho-\overline{\rho_{0}})(x,t)\Big{(}-[\rho n(u-w)](x,t)+\frac{\rho(x,t)}{\overline{\rho_{0}}}\int_{\mathbb{T}} [\rho n (u-w)](y,t)dy\Big{)}dx \Big|\\
&\quad\leq \|\mathcal{J}(\rho-\overline{\rho_{0}})(t)\|_{L^{\infty}}\|\big{(}\rho f(u-v)\big{)}(t)\|_{\mathcal{L}^1}\\
&\quad\leq 2\rho_{+}^{\frac{1}{2}}\|(\rho-\overline{\rho_{0}})(t)\|_{L^1}\|f(t)\|_{\mathcal{L}^1}^{\frac{1}{2}}\|\big{(}\sqrt{\rho f}(u-v)\big{)}(t)\|_{\mathcal{L}^2}\\
&\quad\leq \frac{\gamma(\overline{\rho_{0}})^{\gamma-1}}{4}\|(\rho-\overline{\rho_{0}})(t)\|_{L^2}^2+\frac{4\rho_{+}\overline{n_{0}}}{\gamma(\overline{\rho_{0}})^{\gamma-1}}\|\big{(}\sqrt{\rho f}(u-v)\big{)}(t)\|_{\mathcal{L}^2}^2.
\end{split}
\end{equation}
Substituting (\ref{perturb11})-(\ref{perturb12}), (\ref{perturb13}) into (\ref{perturb1}) and making use of (\ref{Pi2}), we gain (\ref{perturb-}). 
\end{proof}

\underline{\it\textbf{Proof of Theorem \ref{decay} on exponential time-decay rate:~}}  
Let $\mathcal{E}(t)$ and $\mathcal{D}(t)$ be given by (\ref{ED}). For any $\alpha>0$, define
\begin{equation}
\left\{
\begin{split}
&\mathcal{E}^{\alpha}(t)=:\mathcal{E}(t)-\alpha\int_{\mathbb{T}}[\rho(u-m_{1})\mathcal{J}(\rho-\overline{\rho_{0}})](x,t)dx,\\
&\mathcal{D}^{\alpha}(t)=:\mathcal{D}(t)+\alpha\frac{d}{dt}\int_{\mathbb{T}}[\rho(u-m_{1})\mathcal{J}(\rho-\overline{\rho_{0}})](x,t)dx.
\end{split}
\right.
\end{equation}
Thus, it follows from (\ref{entropydissipation}) that
\begin{equation}\label{entropydissipation11}
\begin{split}
\frac{d}{dt}\mathcal{E}^{\alpha}(t)+\mathcal{D}^{\alpha}(t)= 0.
\end{split}
\end{equation}
Making use of (\ref{perturb}), one has
\begin{equation}\label{dissipating1}
\begin{split}
&(1-c_{1}\alpha)\mathcal{E}(t)\leq \mathcal{E}^{\alpha}(t)\leq(1+c_{1}\alpha)\mathcal{E}(t).
\end{split}
\end{equation}
We also deduce from (\ref{m20}) and (\ref{lowerboundrho}) that
\begin{equation}\nonumber
\begin{split}
&\int_{\mathbb{T}\times\mathbb{R}}\rho(x,t)|u(x,t)-v|^2f(x,v,t)dvdx\\
&\quad\geq \rho_{-}\int_{\mathbb{T}\times\mathbb{R}}|u(x,t)-m_{1}(t)+m_{1}(t)-m_{2}(t)+m_{2}(t)-v|^2f(x,v,t)dvdx\\
&\quad=\rho_{-}\int_{\mathbb{T}\times\mathbb{R}} |u(x,t)-m_{1}(t)|^2f(x,v,t)dvdx+\rho_{-}|(m_{1}-m_{2})(t)|^2\int_{\mathbb{T}\times\mathbb{R}} f(x,v,t)dvdx\\
&\quad\quad+\rho_{-}\int_{\mathbb{T}\times\mathbb{R}} |v-m_{2}(t)|^2f(x,v,t)dvdx+2\rho_{-}(m_{1}-m_{2})(t)\int_{\mathbb{T}\times\mathbb{R}}  (u(x,t)-m_{1}(t))f(x,v,t)dvdx\\
&\quad\quad+2\rho_{-}\int_{\mathbb{T}\times\mathbb{R}}  (u(x,t)-m_{1}(t))(m_{2}(t)-v)f(x,v,t)dvdx,
\end{split}
\end{equation}
which together with $(\ref{basiccnsv})_{2}$, (\ref{um1infty}) and the Young's inequality leads to
\begin{equation}\label{this}
\begin{split}
&\|\big{(}\sqrt{\rho f}(u-v)\big{)}(t)\|_{\mathcal{L}^2}^2\\
&\quad\geq \rho_{-}\overline{n_{0}}\|(u-m_{1})(t)\|_{L^{\infty}}^2+\rho_{-}\overline{n_{0}}|(m_{1}-m_{2})(t)|^2+\rho_{-}\|\big{(}(v-m_{2})f\big{)}(t)\|_{\mathcal{L}^2}^2\\
&\quad\quad-2\rho_{-}\overline{n_{0}}|(m_{1}-m_{2})(t)|\|(u-m_{1})(t)\|_{L^{\infty}}-2\rho_{-}\overline{n_{0}}^{\frac{1}{2}}\|\big{(}(v-m_{2})f\big{)}(t)\|_{\mathcal{L}^2}\|(u-m_{1})(t)\|_{L^{\infty}}\\
&\quad\geq -3\rho_{-}\overline{n_{0}}\|u_{x}(t)\|_{L^2}^2+\frac{\rho_{-}\overline{n_{0}}}{2}|(m_{1}-m_{2})(t)|^2+\frac{\rho_{-}}{2}\|\big{(}(v-m_{2})f\big{)}(t)\|_{\mathcal{L}^2}^2.
\end{split}
\end{equation}
In accordance to (\ref{ENSsim}),  (\ref{perturb-}), and (\ref{this}), there is a constant $c_{3}>0$ independent of time $T$ such that it holds
 \begin{equation}\label{dissipating2}
 \begin{split}
 &\mathcal{D}^{\alpha}(t) \geq \alpha c_{0}\|\Pi_{\gamma}(\rho|\overline{\rho_{0}})(t)\|_{L^1}\\
 &\quad\quad~\quad+c_{3}(1-c_{1}\alpha)\big{(}\|\big{(}\sqrt{\rho}(u-m_{1})\big{)}(t)\|_{L^2}^2+\|\big{(}(v-m_{2})f\big{)}(t)\|_{\mathcal{L}^2}^2+|(m_{1}-m_{2})(t)|^2\big{)}.
 \end{split}
 \end{equation}
 Therefore, choosing $\alpha=\frac{1}{2c_{1}}$ in (\ref{entropydissipation11}), one can show by (\ref{entropydissipation11})-(\ref{dissipating1}) and (\ref{dissipating2}) that
\begin{equation}\nonumber
\begin{split}
&\mathcal{D}^{\alpha}(t)\geq \lambda_{1}\mathcal{E}^{\alpha}(t),\quad \frac{d}{dt}\mathcal{E}^{\alpha}(t)+\lambda_{1}\mathcal{E}^{\alpha}(t)\leq 0,\quad \lambda_{1}=: \frac{1}{2}\min\{\frac{c_{0}}{c_{1}},c_{3}\}.
\end{split}
\end{equation}
Applying the Gr$\rm{\ddot{o}}$nwall's inequality and making use of (\ref{Pi2}), (\ref{lowerboundrho}) and $(\ref{dissipating1})$, one obtains
\begin{equation}
\begin{split}
\|(\rho-\overline{\rho_{0}})(t)\|_{L^2}+\|(u-m_{1})(t)\|_{L^2}+\|\big{(}\sqrt{f}(v-m_{2})\big{)}(t)\|_{\mathcal{L}^2}+|(m_{1}-m_{2})(t)|\leq Ce^{-\frac{\lambda_{1}}{2} t}.\label{b111}
\end{split}
\end{equation}

Moreover, by (\ref{rhoh1uniform}) and (\ref{b111}), it holds for any $t\in[0,T]$,
\begin{equation}
\begin{split}
\|(\rho-\overline{\rho_{0}})(t)\|_{C^{0}}\leq C\|(\rho-\overline{\rho_{0}})(t)\|_{L^{2}}^{\frac{1}{2}}\|\rho_{x}(t)\|_{L^2}^{\frac{1}{2}}\leq Ce^{-\frac{\lambda_{1}}{4} t}.\label{brho1}
\end{split}
\end{equation}
Since the conservation laws of mass and momentum implies
$$
 m_{1}=\frac{1}{\overline{\rho_{0}}}(\overline{\rho_{0}u_{0}}+\overline{n_{0}w_{0}}-\overline{n_{0}} m_{2}),
$$
we can get
\begin{equation}\label{m1andm2}
\begin{split}
&|(m_{1}-m_{2})(t)|=|(\frac{1}{\overline{\rho_{0}} }(\overline{\rho_{0}u_{0}}+\overline{n_{0}w_{0}}-\overline{n_{0}} m_{2})-m_{2})(t)|=\frac{\overline{\rho_{0}}+\overline{n_{0}}}{\overline{\rho_{0}}}|(m_{2}-u_{s})(t)|,
\end{split}
\end{equation}
where the constant $u_{s}$ is given by (\ref{rhoinfty}). The combination of (\ref{b111}) and (\ref{m1andm2}) leads to
\begin{equation}
\begin{split}
|(m_{1}-u_{s})(t)|+|(m_{2}-u_{s})(t)|\leq C|(m_{1}-m_{2})(t)|\leq Ce^{-\frac{\lambda_{1}}{2} t}.\label{m1m2b1}
\end{split}
\end{equation}
For the fluid velocity $u$, we also apply (\ref{lowerboundrho}) and (\ref{b111}) to have
\begin{equation}
\begin{split}
&\|(u-u_{s})(t)\|_{L^2}\leq\|(u-m_{1})(t)\|_{L^{2}}+|(m_{1}-u_{s})(t)|\\
&\quad\quad\quad\quad\quad\quad~~\leq \rho_{-}^{-\frac{1}{2}}\|\sqrt{\rho}(u-m_{1})(t)\|_{L^2}+|(m_{1}-u_{s})(t)|\leq Ce^{-\frac{\lambda_{1}}{2} t}.\label{ub1}
\end{split}
\end{equation}
To estimate $f$, one deduces from (\ref{b111}) and (\ref{ub1}) that
\begin{equation}
\begin{split}
&\|\big{(}\sqrt{f}(v-u_{s})\big{)}(t)\|_{\mathcal{L}^2}\leq \|\big{(}\sqrt{f}(v-m_{2})\big{)}(t)\|_{\mathcal{L}^2}+\|f(t)\|_{\mathcal{L}^1}^{\frac{1}{2}}|(m_{2}-u_{s})(t)|\leq C e^{-\frac{\lambda_{1}}{2} t}.\label{fb1}
\end{split}
\end{equation}
And thanks to (\ref{nw}), (\ref{fb1}) and the H$\rm{\ddot{o}}$lder's inequality, it holds
\begin{equation}
\begin{split}
&\int_{\mathbb{T}} (n|w-u_{s}|^2)(x,t)dx=\int_{\{n(x,t)\neq 0\}} \frac{|\int(v-u_{s})f(x,v,t)dv|^2}{ n(x,t)} dx\\
&\quad\quad\quad\quad\quad\quad\quad\quad\quad~~\leq \int_{\mathbb{T}\times\mathbb{R}} |v-u_{s}|^2f(x,v,t)dvdx\leq C e^{-\lambda_{1} t}.\label{bnwb1}
\end{split}
\end{equation}
Then according to the Definition \ref{defn11} and (\ref{fb1}), we have
\begin{equation}\label{fb11}
\begin{split}
 &{\rm{W_{1}}}\big{(}f(t),(n\delta(v-u_{s}))(t)\big{)}\\
 &\quad=\sup_{\|\nabla_{x,v}\psi\|_{\mathcal{L}^{\infty}}\leq 1}\Big{\{}\int_{\mathbb{T}\times\mathbb{R}}\big{(}f(x,v,t)\psi(x,v)dvdx-\int_{\mathbb{T}\times\mathbb{R}} \psi(x,v)  \delta(v-u_{s})\big{)} \int_{\mathbb{R}}f(x,\widetilde{v},t)d\widetilde{v}  dvdx\Big{\}}\\
&\quad=\sup_{\|\nabla_{x,v}\psi\|_{\mathcal{L}^{\infty}}\leq 1}\Big{\{}\int_{\mathbb{T}\times\mathbb{R}} f(x,v,t)(\psi(x,v)-\psi(x,u_{s}))dvdx\Big{\}}\\
&\quad\leq \int_{\mathbb{T}\times\mathbb{R}} |v-u_{s}|f(x,v,t)dvdx\\
&\quad\leq\|\big{(}\sqrt{f}(v-u_{s})\big{)}(t)\|_{\mathcal{L}^2}\|f(t)\|_{\mathcal{L}^1}^{\frac{1}{2}}\leq Ce^{-\frac{\lambda_{1}}{2}t}.
\end{split}
\end{equation}

We show that the compact support of $f$ in velocity concentrates at $u_{s}$ exponentially in time. It is easy to verify
\begin{equation}\nonumber
\begin{split}
u_{s}=e^{-\int_{0}^{t}\rho(X^{x,v,t}(\omega),\omega)d\omega} u_{s}+\int_{0}^{t}e^{-\int_{\tau}^{t}\rho(X^{x,v,t}(\omega),\omega)d\omega}\rho(X^{x,v,t}(\tau),\tau)u_{s}d\tau,
\end{split}
\end{equation}
which together with (\ref{xv1}) leads to
\begin{equation}\label{vus}
\begin{split}
&v-u_{s}=e^{-\int_{0}^{t}\rho(X^{x,v,t}(\tau),\tau)d\tau}(V^{x,v,t}(0)-u_{s})\\
&\quad\quad\quad\quad+\int_{0}^{t}e^{-\int_{\tau}^{t}\rho(X^{x,v,t}(\omega),\omega)d\omega}\big{(}\rho (u-u_{s})\big{)}(X^{x,v,t}(\tau),\tau)d\tau,
\end{split}
\end{equation}
where the bi-characteristic curves $(X^{x,v,t}(s),V^{x,v,t}(s))$ are given by (\ref{XV}). By (\ref{m1m2b1})-(\ref{ub1}) and the Gagliardo-Nirenberg's inequality, one can show
\begin{equation}\label{pppppp}
\begin{split}
&\|(u-u_{s})(t)\|_{L^{\infty}}\leq \|(u-\overline{u})(t)\|_{L^{\infty}}+|(\overline{u}-u_{s})(t)|\\
&\quad\quad\quad\quad\quad\quad\quad\leq C\|(u-\overline{u})(t)\|_{L^2}^{\frac{1}{2}}\|u_{x}(t)\|_{L^2}^{\frac{1}{2}}+|\int_{\mathbb{T}} (u-u_{s})(x,t)dx|\\
&\quad\quad\quad\quad\quad\quad\quad\leq C(\|(u-m_{1})(t)\|_{L^2}+|(m_{1}-\overline{u})(t)|)^{\frac{1}{2}}\|u_{x}(t)\|_{L^2}^{\frac{1}{2}}+\|(u-u_{s})(t)\|_{L^2}\\
&\quad\quad\quad\quad\quad\quad\quad\leq Ce^{-\frac{\lambda_{1}}{4}t}\|u_{x}(t)\|_{L^2}^{\frac{1}{2}}+Ce^{-\frac{\lambda_{1}}{2}t}.
\end{split}
\end{equation}
It derives from $(\ref{basiccnsv})_{4}$, (\ref{fformula}), (\ref{ub1}), (\ref{vus}), (\ref{pppppp}) and the H${\rm{\ddot{o}}}$lder inequality for any $(x,t)\in\mathbb{T}\times[0,\infty)$ that
\begin{equation}\label{vb1}
\begin{split}
&\underset{v\in \Sigma(x,t)}{\rm{sup}}|v-u_{s}|\\
&\quad\leq e^{-\rho_{-}t}(|u_{s}|+R_{0})+C\int_{0}^{t}e^{-\rho_{-}(t-\tau)}\|(u-u_{s})(\tau)\|_{L^{\infty}}d\tau\\
&\quad\leq e^{-\rho_{-}t}(|u_{s}|+R_{0})+C\int_{0}^{t}e^{-\rho_{-}(t-\tau)}\big{(}e^{-\frac{\lambda_{1}}{4}\tau}\|u_{x}(\tau)\|_{L^2}^{\frac{1}{2}}+e^{-\frac{\lambda_{1}}{2}\tau}\big{)}d\tau\\
&\quad\leq Ce^{-\rho_{-}t}+C\|u_{x}\|_{L^2(0,t;L^2)}^{\frac{1}{2}}\big{(}\int_{0}^{t}e^{-\frac{4\rho_{-}}{3}(t-\tau)}e^{-\frac{\lambda_{1}}{3}t}d\tau\big{)}^{\frac{3}{4}}+C\int_{0}^{t}e^{-\rho_{-}(t-\tau)}e^{-\frac{\lambda_{1}}{2}\tau}d\tau\\
&\quad\leq C(e^{-\frac{\rho_{-}}{2}t}+e^{-\frac{\lambda_{1}}{8}t}),
\end{split}
\end{equation}
where $\Sigma(x,t)$ is given by (\ref{Sigma}) and one has used the fact for some constants $a_{1}>0$ and $a_{2}>0$ that
\begin{equation}\nonumber
\begin{split}
&\int_{0}^{t}e^{-a_{1}(t-\tau)}e^{-a_{2}\tau}d\tau\leq \int_{0}^{\frac{t}{2}}e^{-a_{1}(t-\tau)}e^{-a_{2}\tau}d\tau+\int_{\frac{t}{2}}^{t}e^{-a_{1}(t-\tau)}e^{-a_{2}\tau}d\tau\\
&\quad\quad\quad\quad\quad\quad\quad\quad\quad\leq \frac{1}{a_{2}}e^{-\frac{a_{1}}{2}t}+\frac{1}{a_{1}}e^{-\frac{a_{2}}{2}t}.
\end{split}
\end{equation}
The combination of (\ref{brho1}), (\ref{ub1}), (\ref{fb1})-(\ref{fb11}) and(\ref{vb1}) leads to $(\ref{decay11})_{3}$. The proof of Theorem \ref{decay} is completed.

~\

\noindent
\textbf{Acknowledgments.} The authors are grateful to the referees for the
valuable comments and the helpful suggestions on the manuscript.

The research of the paper is supported by National Natural Science Foundation of China (No.11931010, 11671384 and 11871047) and by the key research project of Academy for Multidisciplinary Studies, Capital Normal University, and by the Capacity Building for Sci-Tech Innovation-Fundamental Scientific Research Funds (No.007/20530290068).

\end{document}